\documentclass[11pt,letterpaper]{amsart}
\usepackage{etex}
\usepackage{graphicx}
\usepackage{amsmath}
\usepackage{mathrsfs}
\usepackage{amssymb}
\usepackage{amsthm}
\usepackage[all]{xy}
\usepackage{color}
\usepackage{pictex}
\usepackage{subfigure}
\usepackage{xcolor}

\usepackage[colorlinks]{hyperref}

\newtheorem{thm}{Theorem}[section]
\newtheorem{cor}[thm]{Corollary}
\newtheorem{lem}[thm]{Lemma}
\newtheorem{prop}[thm]{Proposition}

\theoremstyle{definition}
\newtheorem{defn}[thm]{Definition}

\theoremstyle{remark}

\newtheorem{rem}[thm]{Remark}
\numberwithin{equation}{section}

\newcommand{\z}{\mathbb Z}
\newcommand{\cp}{\mathbb C \mathrm P}
\newcommand{\hp}{\mathbb H \mathrm P}
\newcommand{\m}{\mathrm{MCG}}
\newcommand{\T}{\mathcal I}
\newcommand{\K}{\mathcal K}
\newcommand{\J}{\mathcal J}

\newcommand{\omsp}{\Omega^{\mathrm{spin}}}
\newcommand{\omst}{\Omega^{\mathrm{string}}}
\newcommand{\om}{\Omega}
\newcommand{\ov}{\overline}
\newcommand{\wt}{\widetilde{\times}}

\DeclareMathOperator{\Sp}{Sp}

\title[Mapping class group of certain $6$-manifolds]{Mapping class group of manifolds which look like $3$-dimensional complete intersections}
\author{Matthias Kreck and Yang Su}%
\address{Mathematisches Institut, Universit\"at Bonn and Mathematisches Institut, Universit\"at Frankfurt}
\email{kreck@math.uni-bonn.de}
\address{HLM, Academy of Mathematics and Systems Science, Chinese Academy of Sciences and School of Mathematical Sciences, University of Chinese Academy of Sciences}
\email{suyang@math.ac.cn}

\date{}
\begin{document}
\begin{abstract}
In this paper we compute the mapping class group of closed simply-connected 6-manifolds $M$ which look like complete intersections, i.~e.~ $H_2(M;\z) \cong \mathbb Z $ and $x^3 \ne 0$ where $x \in H^2(M; \mathbb Z)$ is a generator. We determine some algebraic properties of the mapping class group; for example we compute its abelianization and its center. We show that modulo the center the mapping class group is residually finite and virtually torsion-free.  We also study low dimensional homology groups. The results are very similar to the computation of the mapping class group of Riemann surfaces. We give generators of the mapping class group, and generators and relations for the subgroup acting trivially on $\pi_{3}(M)$. 
\end{abstract}

\maketitle
\section{Introduction}  The mapping class group of Riemann surfaces plays an important role in different areas of mathematics, in particular in the study of moduli spaces in algebraic geometry . The mapping class group of complex algebraic surfaces is at present inaccessible due to the mysteries of smooth 4-manifolds. In this paper we will determine the mapping class group of certain 6-manifolds including all complex 3-dimensional complete intersections. 

The approach is similar to one of the methods for Riemann surfaces through the comparison with   the Torelli group, the subgroup acting trivially on the cohomology ring. For the Torelli group a central tool is a homomorphism into an abelian group which for Riemann surfaces was defined by Johnson. In our situation we define a similar homomorphism. Surprisingly the kernel of this homomorphism which for Riemann surfaces is still not understood can be complete computed. So one has a complete picture in terms of extensions of groups. 

Knowing this one can ask some obvious questions concerning the nature of the mapping class group. It turns out that the results for Riemann surfaces essentially generalize to our situation. We apply this to get some information about the cohomology of the mapping class group. 

To work with groups generators and relations are very helpful. We give explicit generators in form of generalized Dehn twists, and for the Torelli-type subgroup acting trivially on $\pi_3(M)$ relations, too. 

Recently Oscar Randal-Williams applied our results to the moduli space of hypersurfaces in $\mathbb CP^4$ and determined the image of the monodromy representation \cite{RW23}.

\section{Results}
\subsection{The setting}

All manifolds are assumed to be smooth and oriented. 

\begin{defn} The group of isotopy classes of orientation preserving diffeomorphisms of a manifold $M$ is denoted by $\m(M)$:
$$\m(M) := \pi_0 \mathrm{Diff}^+(M).$$
This is called the (oriented) {\bf mapping class group}. 

The kernel of the homomorphism $\rho: \m(M) \to \mathrm{Aut}(H^*(M;\mathbb Z))$ induced by the action on the cohomology ring is denoted by
$$\T(M):=\ker\{ \rho \colon \m(M) \to \mathrm{Aut}(H^*(M;\z))\}.$$
This is called the {\bf Torelli group}. 
\end{defn} 

These groups play a central role in the study of surfaces and related moduli spaces. In higher dimensions there is not very much known about these groups except for highly connected manifolds \cite{Kr78}. In this paper we study the mapping class group of the following manifolds.

\begin{defn} A {\bf manifold which looks like a 3-dimensional complete intersection} is an oriented simply-connected closed $6$-manifold $M$ with $H_2(M;\z)  \cong \mathbb Z$ and $x^3 \ne 0$ for a generator $x \in H^{2}(M ;\z)$. 
\end{defn}   

This notion is motivated by the fact that  all complex 3-dimensional complete intersections have this property. This follows from the Lefschetz Hyperplane Theorem which states that an $n$-dimensional  complete intersection $X \subset \mathbb C \mathrm P^{n+k}$ is simply-connected and the inclusion into $\mathbb C \mathrm P^{n+k}$ induces an isomorphism in homology below degree $n$ and is  surjective in degree $n$. 

From now on all manifolds are assumed to be of this type unless otherwise said. 

The \textbf{degree} of $M$ is defined by 
$$d(M) : = \langle x^3, [M] \rangle$$
where $x$ is the integral generator of $H^2(M)$ that makes the degree positive.

Since $x^3$ is non-zero and we consider orientation preserving diffeomorphisms the action on $H^2(M;\z) \cong \z$ is trivial and so the Torelli group is the kernel of the action on $H^3(M;\z)$ equipped with the intersection form. Since the intersection form is skew symmetric it is isometric to the skew symmetric hyperbolic form. Once and for all we choose a symplectic basis for $H^3(M;\z)$ and identify the group of isometries with the symplectic group $\mathrm{Sp}(b_3(M),\z)$, where $b_3(M)$ is the third Betti number of $M$. So $\rho$ is a homomorphism from $\m(M)$ to $\mathrm{Sp}(b_3(M),\z)$. 

In addition to the cohomology ring  characteristic classes are attached to a smooth manifold. In our case the second Stiefel Whitney class $w_2(M)$ shows up and is equivalent to $M$ being spin ($w_2(M)=0$) or non-spin. And the first Pontrjagin $p_1(M)\in H^4(M;\z)$ class has to be considered. By Poincar\'e duality $H^4(M;\z) \cong \z$ and since $x^2$ is non zero, there is a unique generator $y$, such that $x^2 = d(M)  y$. It can be shown, when $M$ is spin, $p_1(M)$ is divisible by $4$, and we define $k(M)$ to be the integer such that $p_1(M) = 4k(M)  y$. When $M$ is non-spin, we define $k(M)$ to be the integer such that $p_1(M)=(4k(M)-d(M))  y$.  For details see section \ref{sec:diff}. Finally we abbreviate $H^3(M;\z)$ by $H$.

Following the study of the mapping class group of Riemann surfaces (see for example \cite{F-M}) we determine the mapping class group in terms of two types of short exact sequences. The first type of exact sequences starts with the action on the cohomology ring, whereas the second begins instead with the action on $\pi_3(M)$. 

\subsection{Computation of the mapping class group via the Torelli group}

The Torelli group $\T(M)$ is the kernel of $\rho$ and we prove that $\rho$ is surjective. 
Denote by $\Delta_{d(M), k(M)} $ the subgroup $\{(d(M)u,k(M)u)\, |\, u \in H\} \subset H \times H$. Next we construct a homomorphism $$v_{x,p} : \T(M) \to  (H\times H)/_{\Delta_{d(M),k(M)}},
$$ which we call the {\bf variation  of the cup product and the Pontrjagin class},  and denote its kernel by $\K(M)$. We read off $v_{x,p} (f)$ from the Pontrjagin class and the ring structure of the mapping torus $M_f$. For details of the construction we refer to section \ref{sec:var}. We note that the variation of the cup product and the Pontrjagin class plays the role of the Johnson homomorphism for the mapping class group of surfaces.

\begin{thm}\label{thm:main1}
Let $M$ be a manifold which looks like a $3$-dimensional complete intersection. Then there are short exact sequences
\begin{equation}\label{seq:1}
1 \to \T(M) \to \m(M) \stackrel{\rho}{\longrightarrow} \Sp(b_3(M),\z) \to 1
\end{equation}
\begin{equation}\label{seq:2}
1 \to \K(M) \to \T(M) \stackrel{v_{x,p}}{\longrightarrow} (H\times H)/_{\Delta_{d(M),k(M)}} \to 0.
\end{equation}

The group $\K(M)$ only depends on whether $M$ is spin or not and  $d(M)$ and $p(M)$. In Theorem \ref{thm:kernel} we give tables of the groups $\K(M)$ in terms of $d(M)$ and $k(M)$. In the spin case it is isomorphic to a subgroup of 
$\z/2 \times \z/4 \times \z/84$, all subgroups except the trivial subgroup and $(\z/2)^3$, $(\z/2)^3 \times \z/3$, $(\z/2)^3 \times \z/7$, $(\z/2)^3 \times \z/21$ appear; in the non-spin case it is a subgroup of $\z/168$, and all subgroups except the trivial one appear.

The group $\K(M)$ is the center of $\m(M)$. 
\end{thm}

The determination of $\K(M)$ is the  heart of our work. 

\smallskip

From Theorem \ref{thm:main1} we see that the Torelli group $\T(M)$ is an infinite group when $b_3(M)>0$. This applies especially to $3$-dimensional complex complete intersections and hence contradicts  \cite[Theorem 3.4]{Ver13}. Another class of counterexamples with second Betti number $\ge 2$ was found by R.~Hain \cite{Hai19}. For more general results on this issue see \cite{KS19}.

Looking at the exact sequences (\ref{seq:1}) and (\ref{seq:2}) 
 above, one might wonder if one can combine them into one. In the case of surfaces Morita \cite{Mor93} has made a step in that direction. He has combined the action on the first homology and Johnson homomorphism to a homomorphism 
$$\m(F_g) \to (\frac{1}{2}\wedge^3H^1(F_g)/H^1(F_g)) \rtimes \Sp(2g,\z)$$ 
whose kernel is equal to the kernel of the Johnson homomorphism and its image has finite index. In our situation we have a similar but stronger  result.

\begin{thm}\label{thm:c}
Let $M$ be a manifold which looks like a $3$-dimensional complete intersection with 3rd Betti number $b_3(M)$. When $b_3(M)=2$, or when $b_3(M) \ge 4$ with $M$  nonspin and $d(M)/2$ odd\footnote{The degree of a non-spin manifold is always even, see the footnote in section \ref{sec:diff}.}, there is a surjective homomorphism
$$\tau \colon \m(M) \to (H \times H)/\Delta_{d(M),k(M)} \rtimes \Sp(b_3(M),\z)$$
whose kernel is $\K(M)$.

In the other cases there is a homomorphism
$$ \tau \colon \m(M) \to (H \times H)/\Delta_{d(M),2k(M)} \rtimes \Sp(b_3(M),\z)$$
whose image is
$$ \{([a, b],A) \ | \ a, b \in H, A \in \Sp(b_3(M),\z), \ b \equiv s(A) \pmod 2\}, $$
whose kernel is $\K(M)$, and where $s \colon \Sp(b_3(M),\z) \to H \otimes \z/2$ is a $1$-cocycle which we give  in Theorem \ref{thm:quotient}.

\end{thm}

For example we determine the mapping class group of the quintic $X_5$ in $\mathbb CP^4$ defined by the Fermat equation
$$z_0^{5} + z_1^{5} + z_2^5 + z_3^5  + z_4^{5}=0,$$
which is of particular interest as a Calabi--Yau manifold. 
Note that the 3rd Betti number of $X_5$ is $204$. 

\begin{thm}\label{cor:quintic}
Let $X_5$ be the quintic hypersurface in $\mathbb CP^4$. The center of $\m(X_5)$ is $\K(X_5)$, which is isomorphic to $\z/2$. The group $\m(X_5)/\K(X_5)$ is a subgroup of the semi-direct product $(H \times (H/5))\rtimes \Sp(204,\z)$ of index $2^{204}$, where the action of symplectic group on the abelian group is the tautological one on both factors. More precisely,
$$\m(X_5)/\K(X_5) \cong \{((a, b),A) \ | \ a \equiv s(A) \pmod 2 \} \subset (H \times (H/5))\rtimes \Sp(204,\z)$$
where $s \colon \Sp(204,\z) \to H/2$ is a $1$-cocycle  given in Theorem \ref{thm:quotient}.
\end{thm}

The nontrivial element in the center is represented by a diffeomorphism supported in a closed disk $D^{6}$. We will give an explicit formula of this diffeomorphism in \S \ref{subsec:genrel}.

A celebrated theorem of Sullivan \cite[Theorem 13.3]{Sul77} says that the mapping class group of a simply-connected closed smooth manifold of dimension $>5$ is commensurable to an arithmetic group. In our situation one may see this in view of Theorem \ref{thm:c} and \ref{thm:main1}.

\subsection{Computation of the mapping class group via the special Torelli group}

The Torelli group is the kernel of the action on $H_3(M)$ (which by Poincar\'e duality is isomorphic to $H^3(M)$). Instead one can consider the action on $\pi_3(M)$. In our situation the Hurewicz homomorphism $h$ is surjective with kernel isomorphic to $\mathbb Z/d(M)$ (see section \ref{sec:diff}):
$$0 \to \z/d(M) \to \pi_3(M) \stackrel{h}{\longrightarrow} H_3(M) \to 0$$
Via the Hurewicz homomorphism we can define the intersection form $\lambda$ on $\pi_3(M)$, its radical is the kernel of the Hurewicz map. 

An orientation preserving diffeomorphism of $M$ induces an isomorphism of $\pi_3(M)$ which preserves the intersection form $\lambda$. This isomorphism is the identity on the kernel of the Hurewicz homomorphism (see the paragraph below (\ref{eq:hure})). Denote the group of such isometries by $\Sp(\pi_{3}(M))$ and the kernel of this action by $\J(M)$, a subgroup of $\T(M)$, which we call the {\bf special Torelli group}:
\begin{equation}\label{def:j}
\J(M):= \ker\{\m(M) \to \mathrm{Sp}(\pi_3(M)) \}.
\end{equation}
The group $\mathrm{Sp}(\pi_3(M))$ is a split extension
$$0 \to H \otimes \z/d(M) \to \mathrm{Sp}(\pi_3(M)) \to \mathrm{Sp}(b_3(M),\z) \to 0.$$
On $\J(M)$ we define an invariant $v_p$ landing in $H$, which is derived from the first Pontrjagin class of the mapping torus and called the {\bf variation of the Pontrjagin class}
and in this context plays the role of the homomorphism $v_{x,p}$. The kernel of this invariant is again $\K(M)$.  

So altogether  we have a filtration of subgroups of $\m(M)$:
$$\m(M) \supset \T(M) \supset \J(M) \supset \K(M).$$
These groups are related in a commutative diagram of exact sequences
$$\xymatrix{
& 0 \ar[d]   & 0 \ar[d] & & \\
& \K(M) \ar[d] \ar[r]^= & \K(M) \ar[d]  & & \\
0 \ar[r] &\J(M) \ar[r] \ar[d]^{v_p} & \T(M) \ar[d]^{v_{x,p}} \ar[r] & H/d(M) \ar[d]^= \ar[r] & 0 \\
0 \ar[r] & H \ar[d] \ar[r]^{i_2 \ \ \ \ \ \ \ } &  (H \times H)/\Delta_{d(M),k(M)} \ar[d] \ar[r]  & H/d(M) \ar[r] &  0 \\
& 0 & 0 & & }
$$
To summarize, we have 
\begin{thm}\label{thm:main2}
Let $M$ be a manifold which looks like a $3$-dimensional complete intersection. Then 
there are short exact sequences
\begin{equation}
1 \to \J(M) \to \m(M) \to \Sp(\pi_{3}(M))\to 1
\end{equation}
\begin{equation}\label{seq:j}
1 \to \K(M) \to \J(M) \stackrel{v_p}{\longrightarrow} H \to 0
\end{equation}
\end{thm}
 
We can give an explicit finite presentation of $\J(M)$. We will return to this point in \S \ref {subsec:genrel}.

\subsection{Properties of the mapping class group}\label{subset:prop}
Given Theorem \ref{thm:main1}  one can try to generalize some of the fundamental properties of the mapping class group of surfaces to our situation. In many cases this works, sometimes with obvious modifications which often involve the center.
It's known that the mapping class group of surfaces is centerless when $g \ge 3$.  In our situation, as stated in Theorem \ref{thm:main1}, the center $Z(\m(M))$ of $\m(M)$ equals to $\K(M)$.

\smallskip

The mapping class group of surfaces is perfect when $g \ge 3$.  The same holds in our context in most cases with some exceptional cases where the abelianization is $\z/2$.

\begin{thm}\label{thm:ab}
When $b_3(M)>0$ the homomorphism $\m(M)\to \mathrm{Sp}(b_3(M),\z)$ induces an isomorphism $H_{1}(\m(M)) \to H_{1}(\mathrm{Sp}(b_3(M),\z))$, except for the case $M$ spin and $d$ even,  where there is a split short exact sequence
$$0 \to \z/2 \to H_{1}(\m(M)) \to H_{1}(\mathrm{Sp}(b_3(M),\z)) \to 0.$$
\end{thm}

The group $H_{1}(\mathrm{Sp}(b,\z))$ is isomorphic to  (c.~f.~\cite[Appendix]{BCRR18})
$$\begin{array}{c|c|c}
  b=2 & b = 4 & b \ge 6 \\
\hline
 \z/12 & \z/2  & 0
 \end{array}.$$
 Therefore by Theorem \ref{thm:ab} and the above table $H_1(\m(M))$ is known. 

\smallskip

For the mapping class group of surfaces, the homomorphism $\m(F_g) \to \Sp(2g, \z)$ induces an isomorphism $H_2(\m(F_g) ) \to H_2(\Sp(2g,\z))$ and the latter is isomorphic to $\z$ when $g \ge 4$ (\cite{F-M}). In our situation we are able to obtain a lower bound on the second homology group with $\mathbb Q$-coefficients and we compute this group when $b \ge 6$.

\begin{thm}\label{thm:h2q}
When $b_3(M)>0$ the homomorphism $\m(M) \to \mathrm{Sp}(b_3(M),\z)$ induces a surjective homomorphism $H_2(\m(M)) \to H_2(\Sp(b_3(M),\z))$. When $b_3(M) \ge 6$ we have $H_2(\m(M);\mathbb Q) \cong \mathbb Q^2$.
\end{thm}

The study of the moduli spaces of high dimensional manifolds is currently an active research field. As an example, for $M$ a hypersurface in $\mathbb C \mathrm P^4$,  it's shown in \cite[\S5.3]{GaRW18} that $H_2(B\mathrm{Diff}(M);\mathbb Q) \cong \mathbb Q^4$. For manifolds which look like complete intersection, as a corollary of Theorem \ref{thm:h2q} we can provide some integral information.

\begin{cor}\label{cor:h2}
When $b_3(M)>0$ there is a surjective homomorphism
$$H_{2}(B\mathrm{Diff}(M)) \to H_{2}(\Sp(b_3(M),\z)).$$
\end{cor}


Let $\mathrm{Diff}_0(M)$ be the identity component of $\mathrm{Diff}(M)$. Combining the computation in \cite{GaRW18} with Theorem \ref{thm:h2q}  we have an estimation of the second homology of $B\mathrm{Diff}_0(M)$.
\begin{cor}\label{cor:diff0}
$\dim H_2(B\mathrm{Diff}_0(M);\mathbb Q) \ge 2$.
\end{cor}

The mapping class group of surfaces is virtually torsion free, i.~e., there exists a torsion free subgroup of finite index, and is residually finite, which can be characterized by the property that for every non-trivial element there is a homomorphism to a finite group which maps this element non-trivially. In our situation, after dividing the center (which in the case of surfaces of genus $\ge 3$ doesn't change anything) we obtain the same result:

\begin{thm}\label{thm:resfin}
$\m(M) /Z(\m(M))$ is virtually torsion free and residually finite.
\end{thm}

It's shown by Randal-Williams \cite{RW23} that $\m(M)$ is not always residually finite.

\subsection{Generators and presentations}\label{subsec:genrel}

The celebrated theorem of Sullivan \cite[Theorem 13.3]{Sul77} says that the mapping class group of a simply-connected closed smooth manifold of dimension $>5$ is finitely presented. An obvious question is to describe the generators and relators of $\m(M)$. In the case of surfaces generators are given by a finite number of Dehn twists and this is a very useful fact, since these are very concrete diffeomorphisms. To describe generators of $\m(M)$ in our situation we need a generalization of Dehn twists about $3$-spheres.

\subsubsection{Dehn twists}\label{sec:dt}
 A Dehn twist on a surface is given by an embedding of $S^1 \times [0,1]$ into the surface and the corresponding Dehn twist is the identity outside this embedding and there it is given by rotating the circle $S^1 \times {t}$ by the angle $2 \pi t$.
Passing from $S^1 \times [0,1]$ to thickenings of higher dimensional spheres we obtain one possible generalization of Dehn twist: let $S$ be an embedding of $S^k \times D^{n-k}$ in an $n$-dimensional manifold $V$,  $\alpha \in \pi_{n-k}SO(k+1)$ be represented by a smooth map $f \colon D^{n-k} \to SO(k+1)$, which maps an open neighborhood of $\partial D^{n-k}$ to $I$. Then we define the {\bf (generalized) Dehn twist} 
$$\mathbf{t}_{S,\alpha} \colon V \to V$$
as the diffeomorphism which maps the complement of $S^k \times D^{n-k}$ by the identity and $(x,y) \in S^k \times D^{n-k}$ to $(f(y)x, y)$.

\begin{rem}
An embedding $S$ of $S^k \times D^{n-k}$ is essentially given by two data: an embedding of the core $S^k \times \{0\}$ with trivial normal bundle, and a framing of the normal bundle. The mapping class of the Dehn twist $\mathbf{t}_{S,\alpha}$ depends on the isotopy class of the embedding of $S^k \times \{0\}$ and the homotopy class of the framing. It's interesting to study the change of the mapping class of $\mathbf{t}_{S,\alpha}$ when these data are changed.  It's interesting to study the change of the mapping class of $\mathbf{t}_{S,\alpha}$ when these data are changed. But we won't do it in this paper, since this is not relevant to the main results in this paper, and a systematic study of this could be the content of another paper. A simple observation is that the difference of two Dehn twists with the same isotopy class of $S^k \times \{0\}$ and $\alpha$, but different framings is a disk-supported diffeomorphism, since up to homotopy, two framings coincide in $D^k \times \{0\}$.
\end{rem}

For Dehn twists about $3$-spheres in $6$-manifolds, the action on the 3rd homotopy group is given as follows. 

\begin{lem}(Picard-Lefschetz formula)\label{lem:PL}
Let $S_1$ and $S_{2}$ be embeddings $S^3 \times D^3 \subset M$, the core of $S_{1}$ intersects with  the core of $S_{2}$ transversely at one point. Let $e_{i} \in \pi_{3}(M)$ be the homotopy class represented by the core of $S_i$. Let $\alpha \in \pi_{3}SO(4)$, $\mathbf{t}_{S_1, \alpha}$ be the corresponding Dehn twist. Then 
$$\mathbf{t}_{S_1,\alpha *}(e_{1})=e_{1}, \ \ \ \mathbf{t}_{S_1,\alpha *}(e_{2})=\langle e(\alpha), [S^{4}] \rangle \cdot  e_{1}+e_{2},$$  
where $e(\alpha)$ is the Euler class of the $4$-dimensional vector bundle over $S^{4}$ represented by $\alpha$.
\end{lem}
\begin{proof}
Note that the evaluation $\langle e(\alpha), [S^{4}] \rangle $ equals the image of $\alpha$ under
$$p_{*} \colon \pi_{3}SO(4) \to \pi_{3}(S^{3}) = \z$$
where $p \colon SO(4) \to S^{3}$ is the natural fiber projection. Then the behaviour of $e_1$ and $e_2$ under $\mathbf{t}_{S_1,\alpha *}$ is clear by inspection.
\end{proof}

\subsubsection{Generators}

We distinguish generators of $\pi_3(SO(4))$ as follows (c.f. \cite[\S 3]{Mil56}). The fibration $SO(3) \to SO(4) \to S^3$ gives a split short exact sequence
$$0 \to \pi_3(SO(3)) \to \pi_3(SO(4)) \to \pi_3(S^3)\to 0$$
We denote the class given by the inclusion of the unit quaternions $S^3$ to $SO(4)$ by $\epsilon$
$$\epsilon \colon S^3 \to SO(4)$$
and the element given by the projection $S^3 \to SO(3)$ composed with the inclusion by $\rho$
$$\rho \colon S^3 \to SO(3) \to SO(4).$$
The Euler class of vector bundle over $S^4$ with clutching function $\epsilon$ is $1$, whereas the Euler class for $\rho$ is $0$, and  the first Pontrjagin class for $\rho$ is $4$ (the minimal value with trivial Euler class).

A manifold $M$ which looks like $3$-dimensional complete intersections is the connected sum of a simply-connected manifold $N$ with $b_3(M)/2$ copies of $S^3 \times S^3$'s,
\begin{equation}\label{decomp}
M = N \# \frac{b_3(M)}{2} (S^3 \times S^3),
\end{equation}
with $H_*(N)$ isomorphic to $H_*(\cp^3)$.
Let $S_i$ and $T_i$ be the standard embedding of $S^3 \times D^3$ and $D^3 \times S^3$ in the $i$-th copy of $S^3 \times S^3$, respectively; the homology classes of the cores of $S_i$ and $T_i$ form a symplectic basis of the intersection form on $H_{\z}$. Let $Z$ be the specific embedding of $S^3 \times D^3$ in $S^2 \times D^4 \subset N$, described below Equation (\ref{eqn:y}), the core of $Z$ represents $1 \in \z/d(M) \subset \pi_3(M)$ (see (\ref{eq:hure})). Now we define the Dehn twists  for $1 \le i \le b_3(M)/2$
$$\mathbf{a}_i=\mathbf{t}_{S_i,\rho}, \ \ \mathbf{b}_i=\mathbf{t}_{T_i,\rho}, \ \ \mathbf{e}_i=\mathbf{t}_{S_i, \epsilon}, \ \ \mathbf{f}_i=\mathbf{t}_{T_i,\epsilon}, \ \ \mathbf{z}=\mathbf{t}_{Z,\epsilon}.$$
Let $T_i \# T_{i+1}$ ($1 \le i \le b_3(M)/2-1$), $S_i \# Z$ and $T_i \# Z$ be embeddings of $S^3 \times D^3$, whose core is the connected sum of the core of $T_i$ and $T_{i+1}$, of $S_i$ and $Z$, and of $T_i$ and $Z$, respectively.  Define
$$\mathbf{f}_{i,i+1}=\mathbf{t}_{T_i \# T_{i+1}, \epsilon}, \ \ \mathbf{g}_{i}=\mathbf{t}_{S_i \# Z, \epsilon}, \ \ \mathbf{h}_{i}=\mathbf{t}_{T_i \# Z, \epsilon}$$

\begin{thm}\label{thm:gen}
Let $M$ be a manifolds which looks like a  $3$-dimensional complete intersection with $b_3(M)>0$. When $M$ is spin and $d(M)$ is even, the mapping class group $\m(M)$ is generated by $7b_3(M)/2$  Dehn twists
$$\mathbf{z}, \ \ \mathbf{a}_i, \ \ \mathbf{b}_i, \ \ \mathbf{e}_i, \ \ \mathbf{f}_i, \ \ \mathbf{g}_i, \ \ \mathbf{h}_i, \ \ \mathbf{f}_{i,i+1}.
$$
In the other cases $\m(M)$ is generated by $7b_3(M)/2-1$ Dehn twists
$$\mathbf{a}_i, \ \ \mathbf{b}_i, \ \ \mathbf{e}_i, \ \ \mathbf{f}_i, \ \ \mathbf{g}_i, \ \ \mathbf{h}_i, \ \ \mathbf{f}_{i,i+1}.
$$
\end{thm}

When $b_3(M)=0$, the case is a little different. First of all, we still have the specific  embedding $Z \colon S^3 \times D^3 \subset S^2 \times D^4 \subset M$ as above, and Dehn twists $\mathbf{t}_{Z, \epsilon}$ and $\mathbf{t}_{Z,\rho}$. We need one more generator of different type. Fix an embedding of the $6$-disc $D^6$ in $M$,  we may extend a diffeomorphism of $D^6$ which is the identity near the boundary to a diffeomorphism of $M$ by the identity in the complement. We call these \textbf{disk supported diffeomorphisms} --- diffeomorphisms which are the identity outside a disc in $M$. 
Denote the group of isotopy
classes of diffeomorphisms on $D^6$ which are the identity near the boundary by $\m(D^6, \partial)$, then above construction gives a homomorphism $\m(D^6, \partial) \to \m(M)$. The glueing map $\m(D^6, \partial) \to \Theta_7$ is an isomorphism between $\m(D^6, \partial)$ and the group of oriented  homotopy $7$-spheres (surjective by the $h$-corbordism Theorem and injective by Cerf's theorem). The group $\Theta_7$ is isomorphic to $\mathbb Z/28$ \cite{KeMi}, a generator can be taken as the boundary of the $E_8$-plumbing manifold. A diffeomorphism $f_c \colon D^6 \to D^6$ corresponding to this generator is constructed as follows \cite{Las65}.

Represent a generator of $\pi_3 SO(3)$ by a smooth compactly supported function $\alpha \colon \mathbb R^3 \to SO(3)$ and define the following diffeomorphisms of $\mathbb R^6=\mathbb R^3 \times \mathbb R^3$,
$$F_1 \colon \mathbb R^3 \times \mathbb R^3 \to \mathbb R^3 \times \mathbb R^3, \ (x,y) \mapsto (x, \alpha(x) y)$$
$$F_2 \colon \mathbb R^3 \times \mathbb R^3 \to \mathbb R^3 \times \mathbb R^3, \ (x,y) \mapsto (\alpha(y)x, y)$$
and 
$$f_c=F_1 \circ  F_2 \circ F_1^{-1} \circ F_2^{-1}.$$ 
If follows that $f_c$ is compactly supported and so can be viewed as a diffeomorphism of $D^6$ which is the identity near the boundary. We denote by $\mathbf{f}_c$  the extended diffeomorphism of $M$.  

\begin{thm}\label{thm:gen1}
Let $M$ be a manifold which looks like a  $3$-dimensional complete intersection with $b_3(M)=0$. When $M$ is spin, $\m(M)$ is generated by $\mathbf{f}_c$, $\mathbf{t}_{Z, \epsilon}$ and $\mathbf{t}_{Z,\rho}$; when $M$ is non-spin, $\m(M)$ is generated by $\mathbf{t}_{Z,\rho}$.
\end{thm}

We don't have a finite list of relations between the generators. But we can do this for the group $\J(M)$, which is the kernel of the representation of $\m(M)$ on $\mathrm{Sp}(\pi_3(M))$ (see (\ref{def:j})).  Let $\mathbf{u}=\mathbf{t}_{Z,\rho}$ be the Dehn twists about $Z$ with parameter $\rho$.

\begin{thm}\label{thm:pres}
A presentation of $\J(M)$ is given as follows:

When $M$ is spin, $\J(M)$ has generators $\mathbf{f}_c$, $\mathbf{u}$, $\mathbf{z}$, $\mathbf{a}_i$ and $\mathbf{b}_i$ ($1 \le i \le r$)  and relations:
\begin{enumerate}
\item all the commutators are trivial except for $[\mathbf{a}_{i}, \mathbf{b}_{i}]=\mathbf{f}_c$, $1 \le i \le r$; 
\item the order of $\mathbf{f}_{c}$ is $2^p \cdot 7^{1-c}$, the order of $\mathbf{u}$ is $2^q \cdot 3^{1-b}$, the order of $\mathbf{z}$ is $2^r$, where $b$ and $c$ are as in Theorem \ref{thm:kernel}, and  
$$p = \left \{ \begin{array}{cl}
0 & \mathrm{otherwise} \\
1 & \textrm{$d \equiv 4 \pmod 8$ and $l$ even, or  $d \equiv 0 \pmod 8$ and $l \equiv 2 \pmod 4$} \\
2 &  \textrm{$d \equiv 0 \pmod 8$ and $l \equiv 0 \pmod 4$} \end{array} \right. $$
$$q = \left \{ \begin{array}{cl}
1 & \mathrm{otherwise} \\
2 &  \textrm{$l$ even, or $d \equiv 2 \pmod 4$ and $l$ odd} \end{array} \right. , \ \ \ \ r = \left \{ \begin{array}{cl}
0 & \textrm{$d$ odd} \\
1 & \textrm{$d$ even} \\
\end{array} \right. $$
\end{enumerate}

\smallskip

When $M$ is non-spin,  $\J(M)$ has generators $\mathbf z$, $\mathbf{a}_{i}$ and $\mathbf{b}_{i}$ ($1 \le i \le r$) and relations:
\begin{enumerate}
\item all the commutators are trivial except for $[\mathbf{a}_{i}, \mathbf{b}_{i}]=\mathbf{z}^{6}$, $1 \le i \le r$; 
\item the order of $\mathbf{z}$ is $r=2^{3-a} \cdot 3^{1-b} \cdot 7^{1-c}$, where $a$, $b$, $c$ are as in Theorem \ref{thm:kernel}.

\end{enumerate}

\end{thm}


\noindent \textbf{Outline of the paper.} Since the arguments are in part a bit involved and indirect we give here a rather detailed outline hoping that it makes it easier for a reader to follow the lines.

The first goal is the proof of the exact sequences in Theorem 
\ref{thm:main1}. For this we first recall in Section \ref{sec:pre}  the diffeomorphism classification of manifolds which are like complete intersections. This is a special case of results of Wall and Jupp. In particular manifolds $M$ which look like complete intersections are diffeomorphic to a connected sum of a closed simply connected manifold $N$ with $H_*(N) \cong H_*(\mathbb C \mathrm P^3)$ and $r$ copies of $S^3 \times S^3$, so $M \cong N \# \, r (S^3 \times S^3)$. Given a self diffeomorphism of $\# \,  r (S^3 \times S^3)$ one can isotope it such that it fixes a disc and using this disc for the connected sum with $N$ one can form the connected sum of $f$ with the identity on $N$ to obtain a diffeomorphism on $M$. Thus the mapping class group of $\# \, r (S^3 \times S^3)$ plays a role.   We recall the computation of this mapping class group \cite[Theorem 2]{Kr78} and as an immediate consequence one obtains the first exact sequences in Theorem \ref{thm:main1}. 

In Section \ref{sec:torelli} we  prove the second exact sequence in Theorem \ref{thm:main1} and Theorem \ref{thm:c}. For this we have to construct the variation of the cup product and the Pontrjagin class $v_{x,p}$ on the Torelli group, which is in the same spirit as the Johnson homomorphism on the Torelli group of surfaces, it is read off from the mapping torus of  a diffeomorphism. But the construction is a bit indirect, we first construct a closely related map $\bar v_{x,p}$ which is a bit easier to work with and derive $v_{x,p}$ from it near the end of the section. The aim is to prove that $v_{x,p}$ is a surjective homomorphism and this is achieved via studying $\bar v_{x,p}$ systematically. This study involves again a comparison with the mapping class group of $\# \, r (S^3 \times S^3)$ leading to essentially half of the image of $\bar v_{x,p}$. To understand the full image one looks at the value of $\bar v_{x,p}$ for certain Dehn twists. The arguments are a bit lengthy and we abstain from presenting the numerous steps which finally lead to the definition of $v_{x,p}$ and that it is a surjective homomorphism.

This leaves the main problem, the determination of the kernel $\K(M)$ of the variation homomorphism which is the hard part of our paper. On the one hand we have to define the homomorphism from $\K(M)$ to an abelian group, which is the quotient of $\mathbb Z^3$ by a lattice $L(M)$ depending  on $M$. The construction of this map goes (like the variational homomorphism) via the mapping torus. We call this invariant which is cooked up from three $\mathbb Z$ valued invariants generalized Kreck-Stolz invariants: They occur naturally when one applies the following method for deciding when an element in $\K(M)$ is trivial. Let $f$ be a diffeomorphism representing an element of $\K(M)$, then one considers again the mapping torus but we consider it as a manifold with corners built from two copies of $M \times I$ (red) to which on the one side we attach $M\times I$ (blue) and on the other side attach the mapping cylinder of $f$ (green). This manifold with corners is homeomorphic to the mapping torus. The two last pieces are interpreted as two $h$-cobordisms from $M$ to $M$, the trivial $h$-cobordism and the one twisted by $f$,  giving an $h$-cobordism from $M \times S^0$ to itself. The aim is to fill this manifold with corners by a relative $h$-cobordism $W$ (brown) extending the given $h$-cobordism from $M \times S^0$ to itself. The picture in Figure 1 indicates this situation (where $M$ is a point). 





The relative $h$-cobordism Theorem implies that $f$ is pseudo isotopic to the identity. Cerf's pseudo isotopy theorem \cite{Cerf} implies that it is isotopic to the identity. 

\begin{figure}
\begin{center}
\setlength{\unitlength}{0.75mm}
\begin{picture}(70,40)
\put (10,10){\circle*{2}}
\put (60,10){\circle*{2}}
\put (10,30){\circle*{2}}
\put (60,30){\circle*{2}}
\put (35, 10){\circle*{2}}
\put (35, 30){\circle*{2}}
\put (8, 0){$M$}
\put (30,0){\vector(1,0){10}}
\put (33, 3){id}
\put (30,33){\vector(1,0){10}}
\put (33, 35){$f$}
\put (32, 18){$W$}
\thicklines
{\color{blue}\put (10,10){\line(1,0){50}}}
{\color{green} \put (10,30){\line(1,0){23}} \put (34, 30){\line(1,0){23}} }
{\color{red}\put (7,10){\line(0,1){20}}   \put (57,10){\line(0,1){20}}}
\thinlines
{\color{brown}\put (12, 20){\line(1,1){5}}
\put (15, 15){\line(1,1){10}}
\put (40, 15){\line(1,1){10}}
\put (50, 15){\line(1,1){5}}}

\end{picture}
\end{center}
\caption{}
\end{figure}

A first step towards the construction of $W$ is a weaker step: One looks just for a zero bordism of the mapping torus considered as manifold with corners. Then one tries to replace such a bordism by a sequence of surgeries by an $h$-cobordism. This is a problem which in principle is answered in \cite{Kr99}, where a certain obstruction for this is constructed. In a more recent paper \cite {Kr18} this obstruction is made more explicit and the explicit formulas lead to the generalized Kreck-Stolz invariants which by \cite {Kr18} are an injection. 

There is a delicate point which we ignored. Consider the case where $M$ is spin (in the non-spin case we use twisted spin-structures instead). In the process described above it is important that all manifolds involved are also spin. The mapping torus has a spin structure, but this is not unique. For this reason we pass from the mapping class group of $M$ to the mapping class group $\K(M,D)$  of diffeomorphisms which are the identity on a fixed disc $D \subset M$. For those diffeomorphisms the mapping torus has a canonical spin structure.  So actually   we define the generalized Kreck-Stolz invariants on $\K(M,D)$ with values in $\mathbb Z^3/{L(M)}$ and we prove that these invariants are an injection. 

To show that they are surjective one has to use the decomposition $M = N \# \, r (S^3 \times S^3)$ and decompose $N$ which has the same homology as $\mathbb C \mathrm P^3$ itself as a twisted double of $S^2 \times D^4$ if $M$ is spin, or of $S^2 \widetilde \times D^4$, the non-trivial $D^4$-bundle over $S^2$, if $M$ is non-spin. The relative mapping class group $\m(S^2 \times D^4, \partial)$ or  $\m (S^2 \widetilde \times D^4, \partial)$ acts by extension via the identity on $\m (M)$ and the computation of this relative mapping class group leads to the proof of the surjectivity of the Kreck-Stolz invariants. 

The arguments are different in the spin and non-spin case and for simplicity we first carry the details out in the spin case and explain  the modifications needed for the non-spin case in  Section \ref{sec:nonspin}. The details in the spin case are given in three sections: 
In Section \ref{sec:gks} we define the generalized Kreck-Stolz invariant on $\K(M,D)$. In Section \ref{sec:rel} we compute the mapping class group of $S^2 \times S^4$ and $(S^2 \times D^4, \partial)$ in terms of the generalized Kreck-Stolz invariants. In Section \ref{sec:K} we show that the extension homomorphism $\m(S^2 \times D^4, \partial ) \to \K(M,D)$ is surjective and use it to compute the group $\K(M,D)$, leading to the table in Theorem \ref{thm:kernel}. 

In Section \ref{sec:bd} we show that the forgetful homomorphism $\K(M,D) \to \K(M)$ is an isomorphism, finishing the proof of Theorem \ref{thm:kernel}. As mentioned above Section \ref{sec:nonspin} contains the modifications for the non-spin case. The proofs of the group theoretic properties of $\m(M)$ are in Section \ref{sec:gp}. The computation of the relevant cobordism groups is in the Appendix.

\

\noindent \textbf{Acknowledgements.} We would like to thank Benson Farb, Daniel Huy- brechts, Manuel Krannich, and Eduard Looijenga for helpful comments and Richard Hain and Oscar Randal-Williams for pointing out errors and suggestions how to correct them, and the referees for very valuable suggestions.

\section{Preliminaries}\label{sec:pre}
\subsection{Diffeomorphism classification}\label{sec:diff}
Based on the work of  Wall \cite{Wall} and Jupp \cite{Jupp}, there is a complete diffeomorphism classification of manifolds which look like $3$-dimensional complete intersections. This classification is the starting point of  our computation of their mapping class groups. In this subsection we summarize the classification and invariants of these manifolds. 
 
Let $M \to \mathbb C \mathrm P^{\infty}$ be a map representing the chosen generator $x \in H^2(M)$, replace  $ \mathbb C \mathrm P^{\infty}$ by the mapping cylinder $K$, which is homotopy equivalent to $\mathbb C \mathrm P^{\infty}$, then $M$ is a subcomplex of $K$, and the pair $(K,M)$ is $3$-connected. Therefore the Hurewicz map 
$\pi_4(K,M) \to H_4(K,M)$ is an isomorphism. Also the map $\pi_4(K,M) \to \pi_3(M)$ is an isomorphism since $\pi_3(K) =\pi_4(K)=0$. Note that the homomorphism $H^4(K) \cong \z \to H^4(M) \cong \z$ is a multiplication by $d(M)$, so is the dual map $H_4(M) \to H_4(K)$. Therefore we have 
 an exact sequence (of $\m(M)$-modules)
\begin{equation}\label{eq:hure}
0 \to \z/d(M) \to \pi_3(M) \stackrel{h}{\rightarrow} H_3(M) \to 0,
\end{equation}
where $h$ is the Hurewicz homomorphism. A generator of the subgroup $\z/d$ is represented by the composite of the Hopf map $S^3 \to S^2$ and a map $S^2 \to M$ representing a generator of $\pi_2(M) \cong H_2(M)$. Hence $\z/d$ is a trivial $\m(M)$-module.

By Poincar\' e duality, $y:=x^2/d(M)$ is an integral generator of $H^4(M)$. When $M$ is spin, the first Pontrjagin class $p_1(M)$ of $M$ can be written as (\cite[Theorem 5]{Wall})
\begin{equation}\label{eq:kspin}
p_1(M)=4k(M) y
\end{equation}
with $k(M)=6l(M)+d(M)$, and $l(M)$ can take all integers.
When $M$ is non-spin, $d$ is an even integer\footnote{Let $x_{2}$ and $y_{2}$ be the mod $2$ reduction of $x$ and $y$ respectively, if $d$ is odd, then $y_2=x_2^2$ and by the Wu formula $w_{2}(M) \cup y_{2} = \mathrm{Sq}^{2} y_{2}= \mathrm{Sq}^{2}(x_{2}^{2}) =0$. This shows that $M$ is spin. } \label{page} and we can write (\cite[Theorem 1]{Jupp})
\begin{equation}\label{eq:knspin}
p_1(M)=(4k(M)-d) y
\end{equation}
with $k(M)=12l(M)+d(M)/2$, $l(M)$ takes all integers. 
The diffeomorphism type of $M$ is determined by the property $M$ being spin/nonspin and the integers $b_3(M)$, $d(M)$ and $k(M)$ (or $l(M)$). 

For a $3$-dimensional complete intersection $X$ of multi-degree $(d_1, \cdots, d_m)$, all these invariants can be calculated from the equation of vector bundles 
$$TX \oplus \gamma^{\otimes d_1} \oplus \cdots \oplus \gamma^{\otimes d_m} = i^*T\cp^{3+m}$$
where $\gamma$ is the tautological line bundle over $\cp^{3+m}$ and $i \colon X \to \cp^{3+m}$ the inclusion map (\cite[\S 5.3]{Dim92}). For example, if $X$ is a hypersurface of degree $d$, then 
$$
d(X)=\prod_{i=1}^m d_i, \ \ w_2(X) \equiv (m-\sum_{i=1}^m d_i)x \pmod 2, \ \ p_1(X)=(5-d^2)dy$$
\begin{equation}\label{eq:ci}
 b_3(X)=(d^4-5d^3+10d^2-10d+4).
\end{equation}

\subsection{The mapping class group of $\#\, r(S^{3} \times S^{3})$}\label{sec:review}
In \cite{Kr78} the mapping class group of $(k-1)$-connected almost-parallelizable $2k$-manifolds is studied. For our purpose in this section we recall the result for the connected sum of $r$ copies of $S^3 \times S^3$, denoted by $\#\,r(S^{3} \times S^{3})$.

The intersection form $$I \colon H_{3}(\# \,r(S^{3} \times S^{3})) \times H_{3}(\# \, r(S^{3} \times S^{3})) \to \z$$
is a skew-symmetric unimodular bilinear form over $\z$. After taking a symplectic basis, the isometry group of this intersection form is identified with $\Sp(2r,\z)$. An orientation preserving diffeomorphism $f \colon  \# \, r(S^{3} \times S^{3}) \to \# \, r(S^{3} \times S^{3}) $ induces an isometry of $I$. Therefore there is  a homomorphism 
$$\m(\# \, r(S^{3} \times S^{3})) \to \Sp(2r,\z).$$ 

Denote the kernel of this homomorphism by $\T(\# \, r(S^{3} \times S^{3}))$, we define a homomorphism
$$\chi \colon \T(\# \, r(S^{3} \times S^{3})) \to H^{3}(\# \, r(S^{3} \times S^{3}))$$
as follows. Let $[f] \in \T(\# \, r(S^{3} \times S^{3}))$, denote the mapping torus by $(\# \, r(S^{3} \times S^{3}))_{f}$, consider the collapsing map
$$c \colon (\# \, r(S^{3} \times S^{3}))_{f} \to (\# \, r(S^{3} \times S^{3}))_{f}/\{0 \} \times  \# \, r(S^{3} \times S^{3})=\Sigma (\# \, r(S^{3} \times S^{3}) ^{+})$$
If $f_{*}=\mathrm{id}$ on homology, then 
$$c^{*}\colon H^{4}(\Sigma (\# \, r(S^{3} \times S^{3}) ^{+})) \to H^{4}((\# \, r(S^{3} \times S^{3}))_{f})$$ 
is an isomorphism. Then $\chi (f)$ is defined to be $\frac{1}{4}p'(f)$, where $p'(f)$ is the desuspension of $(c^{*})^{-1}(p_{1}(\# \, r(S^{3} \times S^{3}))_{f})$ (\cite[Lemma 2]{Kr78}). The geometric meaning of this invariant is the following: let $x \in H_{3}(\# \, r(S^{3} \times S^{3}))$ be represented by an embedded $S^{3}$, which is fixed pointwisely by $f$,  since the stable normal bundle of the embedded $S^{3}$ is trivial, the differential $df$ induces a map $S^{3} \to SO(4)$, giving an element in $\pi_{3}SO(4)$. This element lies in the image of the suspension map $\pi_{3}SO(3) \to \pi_{3}SO(4)$. Identify this group with $\z$, then this element equals $\langle  \chi(f), x \rangle$ (see \cite[Lemma 2]{Kr78}, and a correction of a factor $2$ by \cite[Lemma 1.6]{Crowley09}).

\begin{thm}\cite[Theorem 2]{Kr78}\label{thm:kr}
There are exact sequences
$$1 \to \T(\# \, r(S^{3} \times S^{3}))  \to  \m(\# \, r(S^{3} \times S^{3}))\to \Sp(2r,\z) \to 1$$
$$0 \to \Theta_{7} \to \T(\# \, r(S^{3} \times S^{3})) \stackrel{\chi}{\longrightarrow} H^{3}(\# \, r(S^{3} \times S^{3})) \to 1$$
\end{thm}

\begin{rem}\label{rem:ext}
If we consider the quotient by the subgroup $\Theta_{7}$, then
we get a group extension with abelian kernel
$$0 \to H^{3}(\# \, r(S^{3} \times S^{3}) )\to \m(\# \, r(S^{3} \times S^{3}))/\Theta_{7} \to \Sp(2r,\z) \to 1$$
in which the action of $\Sp(2r,\z)$ on $H^{3}(\# \, r(S^{3} \times S^{3}))$ is the standard linear action. This extension problem is studied in \cite{Crowley09}: it is split if and only if $r=1$; for $r \ge 2$ the cohomology class in $H^{2}(\Sp(2r,\z) , H^{3}(\# \, r(S^{3} \times S^{3})))$ corresponding to this extension is described in  \cite[Corollary 3.5]{Crowley09}.
\end{rem}

\subsection{The surjectivity of $\rho$}
Now we may prove the surjectivity of the representation $\rho \colon \m(M) \to \mathrm{Sp}(b_3(M), \mathbb Z)$. Recall that we have a decomposition (\ref{decomp}) 
$$M = N \# \frac{b_3(M)}{2}(S^{3} \times S^{3}).$$ 
Denote $r=b_3(M)/2$, 
given a diffeomorphism of $\# \, r (S^3 \times S^3)$ which is the identity on an embedded disk $D$, we extend it to a diffeomorphism of $M$ via the identity on the complement. Denote the group of  isotopy classes of diffeomorphisms of $\# \, r(S^3 \times S^3)$ which are the identity on $D$ by $\m(\# \, r(S^{3} \times S^{3}), D)$, we have a homomorphism $\m(\# \, r(S^{3} \times S^{3}), D) \to \m(M)$. By \cite[Lemma 3]{Kr78} the natural homomorphism $\m(\# \, r(S^{3} \times S^{3}), D) \to \m(\# \, r(S^{3} \times S^{3}))$ is an isomorphism. Therefore we have a homomorphism $\m(\# \, r(S^{3} \times S^{3})) \to \m(M)$ which clearly fits into the following commutative diagram
\begin{equation}
\xymatrix{
\m(\# \, r (S^3\times S^3)) \ar[d] \ar[r] & \Sp(2r,\z) \ar[d]^= \\
\m(M) \ar[r]^{\rho} & \Sp(b_3(M),\z) }
\end{equation}
By Theorem \ref{thm:kr} the homomorphism in the upper row is surjective, therefore $\rho \colon \m(M) \to \Sp(b_3(M),\z)$ is surjective.
This gives the exact sequence (\ref{seq:1}) in Theorem \ref{thm:main1}.

\section{The Torelli group $\T(M)$}\label{sec:torelli}
\noindent \emph{Convention}. Recall that in section \ref{sec:diff} we attach numerical invariants $d(M)$, $k(M)$ and $l(M)$ to the manifold $M$. From now on we simply write these numbers as $d$, $k$ and $l$, respectively.

\subsection{Variation of the cup product and the Pontrjagin class}\label{sec:var}
\subsubsection{The definition of the variation map}\label{subsec:var}
Next we define a homomorphism
$$\bar v_{x,p} \colon \T(M) \to (H^{3}(M) \times H^{3}(M))/\Delta_{d,2k}$$
where $\Delta_{d,2k}=\{ (du, 2ku) \ | \ u \in H^3(M) \}$ is the diagonal of slope $(d,2k)$, as follows. 

Let $f \colon M \to M$ be a diffeomorphism such that $[f] \in \T(M)$, 
$$M_{f}= M \times [0,1] / (x,0) \sim (f(x),1)$$ 
be the mapping torus of $f$. Then $M_f$ is an $M$-bundle over the circle. Let $i \colon M \to M_f$ be the inclusion of a fiber. Since $f$ acts trivially on homology, the Wang sequence of the fiber bundle implies
\begin{enumerate}
\item  $i^* \colon H^2(M_f) \to H^2(M)$ is an isomorphism;
\medskip
\item there is a short exact sequence
\begin{equation}\label{seqwang}
0  \to H^{3}(M) \stackrel{\delta}{\rightarrow} H^{4}(M_{f}) \stackrel{i^{*}}{\rightarrow} H^{4}(M) \to 0
\end{equation}
\end{enumerate}
Let $x \in H^2(M)$, $y \in H^{4}(M)$ be generators, such that $x^2=dy$ and $d > 0$. Identifying $H^2(M)$ with $H^2(M_f)$ via $i^*$, we still denote the corresponding generator of $H^2(M_f)$ by $x$; let $\overline y \in H^{4}(M_{f})$ be a preimage of $y$.

The homomorphism $i^* \colon H^2(M_f;\z/2) \to H^2(M;\z/2)$ is an isomorphism, and $i^*w_2(M_f) = w_2(M)$.  Therefore $M_f$ is spin if and only if $M$ is spin. Recall (\cite{Tho62}) that for the universal spin vector bundle $\xi_{\mathrm{spin}}$ over $B\mathrm{Spin}$, its first Pontrjagin class $p_1(\xi_{\mathrm{spin}}) \in H^4(B\mathrm{Spin};\mathbb Z) \cong \mathbb Z$ equals 2 times a generator of $H^4(B\mathrm{Spin};\mathbb Z)$, which we denote by $\frac{p_1}{2}(\xi_{\mathrm{spin}})$. For a spin vector bundle $\xi$, the pullback of $\frac{p_1}{2}(\xi_{\mathrm{spin}})$ by the classifying map of $\xi$ is  denoted by $\frac{p_1}{2}(\xi)$, and we have $2 \cdot \frac{p_1}{2}(\xi)= p_1(\xi)$.  Therefore if $M$ is spin, we have (see (\ref{eq:kspin}))
$$\frac{p_1}{2}(M)=2ky;$$ 
if $M$ is non-spin, then $TM \oplus \eta$ is a spin vector bundle, where $\eta$ is the complex line bundle over $M$ with $c_1(\eta)=x$. Since (see (\ref{eq:knspin}))
$$p_1(TM\oplus \eta) = p_1(M) + x^2 =p_1(M) + dy=4ky $$
we have
$$\frac{p_1}{2}(TM \oplus \eta)=2ky.$$ 

Therefore if $M$ is spin, in the homomorphism $i^* \colon H^4(M_f) \to H^4(M)$, we have  
\begin{equation}\label{relat}
i^{*}(x^{2}-d\overline y)=x^{2}-dy=0, \ \ \ \ i^{*}(\frac{p_{1}}{2}(M_{f})-2k\overline y)=\frac{p_{1}}{2}(M)-2ky=0
\end{equation}
Therefore $x^2 -d\ov y$ and $\frac{p_{1}}{2}(M_{f})-2k\overline y$ have unique preimages under $\delta$ in $H^{3}(M)$
$$\delta^{-1}(x^{2}-d\overline y) \in H^{3}(M), \ \ \ \delta^{-1}(\frac{p_{1}}{2}(M_{f})-2k\overline y) \in H^{3}(M)$$
Note that $\overline y$ is not unique, we may change it by
$\overline y \to \overline y + \delta(u)$, $u \in H^{3}(M)$. The corresponding change of the preimages under $\delta$ follows from the equations
$$
\delta^{-1}(x^{2}-d(\overline y+\delta(u)))=\delta^{-1}(x^{2}-d\overline y)+du, $$
$$\delta^{-1}(\frac{p_{1}}{2}(M_{f})-2k(\overline y+\delta(u)))=\delta^{-1}(\frac{p_{1}}{2}(M_{f})-2k\overline y)+2ku
$$
Thus we have a well-defined map
$$\bar v_{x,p} \colon \T(M) \to (H^{3}(M) \times H^{3}(M)) /\Delta_{d,2k},$$
$$f \mapsto [(\delta^{-1}(x^{2}-d\overline y), \delta^{-1}(\frac{p_{1}}{2}(M_{f})-2k\overline y))].$$

If $M$ is non-spin, we just need to replace $\frac{p_1}{2}(M_f)$ in the second equation in (\ref{relat}) by $\frac{p_1}{2}(TM_f \oplus \overline{\eta})$, where $\overline{\eta}$ is the complex line bundle with $c_1(\overline{\eta})=x$. Now since
$$i^*(\frac{p_1}{2}(TM_f \oplus \overline{\eta})-2k\ov y)= \frac{p_1}{2}(TM\oplus i^*\overline{\eta})-2ky=0$$
we have an element $\delta^{-1}(\frac{p_1}{2}(TM_f \oplus \overline{\eta})-2k\ov y) \in H^3(M)$.

\begin{lem}\label{prop:psi}
The map $\bar v_{x,p} \colon \T(M) \to (H^{3}(M) \times H^{3}(M)) /\Delta_{d,2k}$ is a homomorphism.
\end{lem}
\begin{proof}
The exact sequence (\ref{seqwang}) is identified with
$$0 \to H^4(M_f, M) \to H^4(M_f) \to H^4(M) \to 0.$$
We construct $M_{fg}$ from $M_f$ by cutting $I \times M$ along $1/2 \times M$ and reglueing it via $g$ from the left to the right. Denote the two fibres over $1/4$ and $3/4$ by $M_1$ and $M_2$. Consider the restriction map
$$
H^4(M_{fg}, M_1 \cup M_2) \to H^4(M_{fg}, M_1).
$$
We have an isomorphism
$$H^4(M_f, M) \oplus H^4(M_g, M) \stackrel{\cong}{\rightarrow} H^4(M_{fg}, M_1 \cup M_2).$$
If we identify $H^4(M_{fg},M_1)$, $H^4(M_f, M)$ and $H^4(M_g,M)$ with $H^3(M)$, the homomorphism
$$H^4(M_f, M) \oplus H^4(M_g, M) \stackrel{\cong}{\rightarrow} H^4(M_{fg}, M_1 \cup M_2) \to H^4(M_{fg}, M_1)$$
corresponds to the homomorphism $H^3(M) \oplus H^3(M) \to H^3(M)  $ given by the sum, and the corresponding cohomology classes for $M_f$ and $M_g$ are mapped to that of $M_{fg}$. This shows that $\bar v_{x,p}$ is a homomorphism.
\end{proof}

\subsubsection{The image of the variation map}
Next we determine the image of $\bar v_{x,p}$.
Consider the Hurewicz homomorphism $h \colon \pi_3(M) \to H_3(M)$. Let $g \colon M \to \mathbb C \mathrm P^\infty$ be a representative of the classifying map of the chosen generator $x$. We replace $g$ by the fibration $p \colon E_g \to \mathbb C \mathrm P^\infty$, where $E_g = \{(x,\alpha)\,|\, \alpha \colon I \to \mathbb C P^\infty, \alpha(0) = g(x)\}$ with the compact open topology, and $p$ mapping  $(x, \alpha)$ to $\alpha(1)$. The inclusion $i \colon M \to E_g$ mapping $x$ to $(x, \alpha)$ with $\alpha (t) = g(x)$ is a homotopy equivalence (with homotopy inverse mapping $(x,\alpha) $ to $p_1(x)$, the projection to $M$) such that $g$ is homotopic to $p \circ i$. We denote the fibre of $p$ by $F$. Then $\pi_3(M) \cong \pi_3(F) \cong H_3(F)$. The fibration $F \to E_g \to \mathbb C \mathrm P^\infty$ extends to the left to give a fibration $\Omega \mathbb C \mathrm P^\infty \to F \to E_g$, with $\Omega \mathbb C \mathrm P^\infty \simeq S^1$. From the Gysin sequence of this fibration we get an exact sequence
\begin{equation}\label{seq1}
0 \to H_4(M) \stackrel{\cap x}{\longrightarrow} H_2(M) \to H_3(F) \to H_3(M) \to 0,
\end{equation}
which induces a short exact sequence
$$ 0 \to \z/d \to \pi_3(M) \stackrel{h}{\rightarrow} H_3(M) \to 0.$$
This short exact sequence is compatible with the action of $f_*$,  and for $f \in \T(M)$, $f_{*}=\mathrm{id}$ on $H_{3}(M)$ and $\z/d$. The homomorphism $f_*-\mathrm{id} \colon \pi_3(M) \to \pi_3(M)$ gives rise  to a homomorphism $\beta \colon H_{3}(M) \to \z/d$, which we view as an element in $H^{3}(M;\z/d)$.

\begin{lem}\label{lem:pi3}
Under the composition
$$\T(M) \stackrel{\bar v_{x,p}}{\longrightarrow} (H^{3}(M) \times H^{3} (M))/\Delta_{d,2k} \stackrel{\mathrm{pr_{1}}}{\longrightarrow} H^{3}(M)/d = H^{3}(M;\z/d)$$
 the image of $f$ is $\beta$. And this map is surjective.
\end{lem}
\begin{proof}
We consider the fibration $p \colon E_g \to \mathbb C \mathrm P^\infty$ replacing the classifying map $g$ of our chosen generator $x$ constructed above. Since our diffeomorphism $f \colon M\to M$ acts trivially on $H^2(M)$ the composition $gfp_1: E_g  \to \mathbb C \mathrm P^\infty$ is homotopic to $p$, where $p_1$ is the projection to $M$. We choose a homotopy and lift it to a map $\hat f: E_g \to E_g$ commuting with $p$. This is a homotopy equivalence and we consider its mapping torus $(E_g)_{\hat f}$. This is on the one hand a fibration over $S^1$ with monodromy $\hat f$ and on the other hand a fibration over $\mathbb C \mathrm P^\infty$ with fibre the mapping torus of the fibre $F$ of $p$ equipped with the restriction of $\hat f$.

Thus we have a commutative diagram of maps
$$\xymatrix{
F \ar[r] \ar[d] & E_g \ar[r] \ar[d] & \cp^{\infty} \ar[d]^= \\
F_{\hat f }\ar[r] \ar[d] & (E_g)_{\hat f} \ar[r] \ar[d] & \cp^{\infty} \\
S^1 \ar[r]^= & S^1 & \\}$$
where the two horizontal lines and the two vertical lines are fibrations. Consider the Gysin sequences of the horizontal fibrations and the Wang sequences of the vertical fibrations. We use the canonical homotopy equivalences to identify $H_*(M)$ with $H_*(E_g)$, and $H_*(M_f)$ with $H_*((E_g)_{\hat f})$. Then we have the  following commutative diagram, where the rows and columns are exact sequences
$$\xymatrix{
H_{4}(M) \ar[r] \ar[d] & H_{2}(M) \ar[r] \ar[d] & H_{3}(F) \ar[d]^{f_{*}-\mathrm{id}} \ar[r]^{j_*} & H_3(M) \ar[d]  \\
H_{4}(M) \ar[r] \ar[d] & H_{2}(M) \ar[r] \ar[d]^{\cong} & H_{3}(F) \ar[d] \ar[r]^{j_*} & H_3(M) \ar[d] \\
H_{4}(M_{f}) \ar[r]^{x \cap -} \ar[d]^{\partial} & H_{2}(M_{f}) \ar[r] & H_{3}(F_{f}) \ar[r] & H_3(M_f) \\
H_3(M) & & & }$$
The composition $H_4(M_f) \to H_2(M_f) \cong H_2(M) 
\to H_3(F)$ of maps in the above diagram  induces a map $\gamma \colon H_3(M) \to H_3(F)$, such that $\gamma \circ j_*=f_*-\mathrm{id} \colon H_3(F) \to H_3(F)$.

Let $\alpha \in H^3(M)$ be the preimage of $x^2 -d\ov y \in H^4(M_f)$ under $\delta$ (see the sequence (\ref{seqwang})). For any $a \in H_3(M)$, let $\overline a \in H_3(F)$ be a preimage of $a$.  Since $(f_*-\mathrm{id})(a)=0$, there is an element $\overline b \in H_2(M)$ which is the preimage of $(f_*-\mathrm{id})(\overline a)$ under the map $H_2(M) \to H_3(F)$. Let $ b \in H_2(M_f)$ be the image of $\overline b$ under the isomorphism $H_2(M) \to H_2(M_f)$, then there is an element $c \in H_4(M_f)$ whose image under the map $H_4(M_f) \to H_2(M_f)$ is $b$. From the construction of $\gamma$ we have  $\partial c =a$. Now 
$$ \langle \alpha , a \rangle = \langle \alpha , \partial c \rangle = \langle \delta \alpha, c \rangle = \langle x^2 -d\ov y , c\rangle  \equiv \langle x^2, c \rangle = \langle x ,x \cap c \rangle = \langle x, b \rangle \pmod d.$$
This shows that the homomorphism $\beta \colon H_3(M) \to \z/d$ coincides with the homomorphism induced by the cohomology class $\delta^{-1}(x^2 -d\ov y) \in H^{3}(M;\z/d)$.

Next we show the surjectivity. Recall that there is a decomposition $M = N \# \, r(S^{3} \times S^{3})$, where $r=b_3(M)/2$. We have $\pi_3(M) = \pi_3(N) \oplus \pi_3(\# \, r(S^3 \times S^3)) \cong \z/d \oplus \z^{2r}$. Let $e_0$ be a generator of $\pi_3(N) \cong \z/d$, let $e_1, f_1, \cdots, e_r, f_r$ be the standard basis of $\pi_3(\# \, r(S^3 \times S^3))$, whose images under the Hurewicz homomorphism form a symplectic basis of $H_3(M)$. Let $S_1, S_2 \colon S^3 \times D^3 \subset M$ be embeddings whose cores represent $e_i$ and $e_0 + e_i$, respectively. Choose $\alpha \in \pi_3SO(4)$ such that $\langle e(\alpha), [S^{4}] \rangle =1$,  where $e(\alpha)$ is the Euler class of the vector bundle over $S^4$ with clutching function $\alpha$. Consider the Dehn twists $\mathbf{t}_{S_2, \alpha}$ and $\mathbf{t}_{S_1, -\alpha}$. Then by Lemma \ref{lem:PL} 
$$\mathbf{t}_{S_2, \alpha*}(f_{i})=e_{0}+e_{i}+f_{i}, \ \ \mathbf{t}_{S_1,-\alpha*}(f_{i})=f_{i}-e_i{}.$$ 
The other generators of $\pi_{3}(M)$ are fixed. Then the homomorphism $(\mathbf{t}_{S_1,-\alpha} \circ \mathbf{t}_{S_2, \alpha})_{*}-\mathrm{id} \colon H_3(M) \to H_3(M)$ maps $f_{i}$ to $e_{0}$, and other generators of $H_{3}(M)$ to $0$. This finishes the proof of Lemma \ref{lem:pi3}.
\end{proof}


Next we consider the composition
$$\T(\# \, r(S^{3} \times S^{3}))\to \T(M) \stackrel{\bar v_{x,p}}{\longrightarrow} (H \times H )/\Delta_{{d,2k}}.$$
\begin{lem}\label{lem:chi}
There is a commutative diagram
$$\xymatrix{
\T(\# \, r (S^3 \times S^3)) \ar[r]^{\ \ \ \ \chi} \ar[d] & H \ar[d]^{2 \cdot i_2} \\
\T(M) \ar[r]^{\bar v_{x,p} \ \ \  \ \ \ } & (H \times H)/\Delta_{d,2k}}$$
where $i_2 \colon (H \times H)/\Delta_{d,2k}$, $\alpha \mapsto [0,\alpha]$ is the inclusion of the second factor composed with the projection map.\end{lem}
\begin{proof}
Recall that we fix a decomposition $M=N \# \, r(S^3 \times S^3)$  (\ref{decomp}) where $r=b_3(M)/2$.  Let $f \colon M \to M$ be a diffeomorphism which is the identity on $N-\mathring D$, then the mapping torus is 
$$M_{f}=(N \times S^{1}) \#_{S^{1}} (\# \, r(S^3 \times S^3))_{f},$$ 
where $\#_{S^1}$ denotes the fiber connected-sum along $S^1$. Under the decomposition
$$H^{4}(M_{f}) =H^{4}(N) \oplus H^{4}((\# \, r(S^3 \times S^3))_{f})$$
we may choose $\overline y = (y,0)$, where $y=x^2/d \in H^4(N)$ is the canonical generator. Therefore in $H^4(M_f)$ we have the equation
$$x^{2}-d\overline y=0.$$
For the Pontrjagin class, in the spin case we have $$\frac{p_{1}}{2}(M_{f})=(\frac{p_{1}}{2}(N), \frac{p_{1}}{2}((\# \, r(S^3 \times S^3))_{f})),$$
therefore
$$\frac{p_{1}}{2}(M_{f})-2k\overline y=\frac{p_{1}}{2}((\# \, r(S^3 \times S^3))_{f}).$$
In the non-spin case we have 
$$\frac{p_{1}}{2}(TM_{f}\oplus \overline{\eta})-2k \ov y= \frac{p_{1}}{2}((\# \, r(S^3 \times S^3))_{f}). $$  

The coboundary map $\delta \colon H^3(M) \to H^{4}(M_f)$
in the Wang sequence (\ref{seqwang}) can be identified with 
$$H^3(M) \stackrel{\Sigma}{\longrightarrow} H^{4}(\Sigma(M^+)) \stackrel{c^*}{\longrightarrow} H^{4}(M_f).$$
We conclude the proof by identifying $H^3(\# \, r(S^3 \times S^3))$ with $H^3(M)$ and comparing the homomorphism $\bar v_{x,p}$ with the invariant $\chi$ defined in \S \ref{sec:review}.
\end{proof}

Moreover, we have the following

\begin{lem}\label{lem:image2}
$\mathrm{Im}(\bar v_{x,p}) \cap i_2(H) = (0 \times 2H)/\Delta_{d,2k}$
\end{lem}
\begin{proof}
That the first component of $\bar v_{x,p}(f)$ equals $0$ is equivalent to that $f$ induces the identity on $\pi_3(M)$ (Lemma \ref{lem:pi3}). 

We first assume $M$ is spin. Recall the second component of $\bar v_{x,p}(f)$ is $\delta^{-1}(\frac{p_1}{2}(M_f)-2k \bar y)$. It suffices to show that $p_1(M_f)$ is divisible by $4$. This is done by showing that the evaluation of $p_1(M_f)$ on a basis of $H_4(M_f)$ is divisible by $4$. Choose embeddings $S_i^3 \subset M$ representing a basis of $H_3(M)$. By Haefliger's theorem \cite[Theorem 1b]{Hae61}, $S_i^3$ and $f(S_i^3)$ are isotopic in $M \times [0,1]$. We choose an isotopy and obtain a smooth map $g_i \colon S^3_i \times S^1 \to M_f$, which is an embedding on each $S^3_i \times \{t\}$.  Then from the exact sequence 
$$0 \to H_4(M) \to H_4(M_f) \to H_3(M) \to 0$$
a basis of $H_4(M_f)$ is given by a generator of $H_4(M)$ and $g_{i*}[S^3_i \times S^1]$. The evaluation of $p_1(M_f)$ on the generator of $H_4(M)$ equals $4k$. Next we consider the evaluation $\langle p_1(M_f), g_{i*}[S^3_i \times S^1] \rangle$. We drop the subscript $i$ in the sequel. 

We form the map $g' \colon S^3 \times S^1 \to  M_f \times S^1$, $(x,t) \mapsto (g(x,t), t)$, which is an embedding of a $4$-manifold into an $8$-manifold. Then 
$$\langle p_1(M_f ), g_*[S^3 \times S^1] \rangle = \langle p_1(M_f \times S^1), g'_*[S^3 \times S^1] \rangle =\langle p_1(\nu), [S^3 \times S^1] \rangle$$
where $\nu$ is the normal bundle of $g'(S^3 \times S^1)$ in $M_f \times S^1$. Note that $g'$ is homotopic to $g'' \colon S^3 \times S^1 \to M_f \times S^1$ with $g''(x,t)=(g(x,t), -t)$, and the intersection of $g'$ and $g''$ is empty. This shows that the Euler class of $\nu$ is trivial, hence $\nu$ is reduced to a $3$-dimensional vector bundle. For such vector bundles, the first Pontrjagin class is divisible by $4$.

In the non-spin case we need to replace $p_1(M_f)$ by $p_1(TM_f \oplus \bar\eta)$. But since any complex line bundle over $S^3 \times S^1$ is trivial, the above computation still works. 

\end{proof}

By Lemma \ref{lem:pi3}, \ref{lem:chi} and \ref{lem:image2}, we may conclude that $\mathrm{Im} \bar v_{x,p} = (H \times 2H)/\Delta_{d,2k}$. Now we may remove the factor $2$ by modifying the definition of the variation map. Namely,  in the spin case we define
$$v_{x,p}(f)=(\delta^{-1}(x^2-d\ov y), \delta^{-1}(\frac{1}{4}(p_1(M_f)-4k\ov y))) \in (H \times H)/\Delta_{d,k};$$
in the non-spin case 
$$v_{x,p}(f)=(\delta^{-1}(x^2-d\ov y), \delta^{-1}(\frac{1}{4}(p_1(TM_f \oplus \overline \eta )-4k\ov y))) \in (H \times H)/\Delta_{d,k}.$$
Then we have the following
\begin{prop}\label{prop:tau}
$v_{x,p} \colon \T(M)\to (H \times H) /\Delta_{d,k}$ is a surjective homomorphism.
\end{prop}

This gives the proof of the exact sequence (\ref{seq:2}) in Theorem \ref{thm:main1}.

\smallskip

Theorem \ref{thm:c} is a consequence of the following stronger statement.

\begin{thm}\label{thm:quotient}
Taking the quotient by $\K(M)$ we have a short exact sequence
\begin{equation}\label{seqthm}
0 \to (H \times H )/\Delta_{d,k} \to \m(M)/\K(M) \to \Sp(b_3(M),\z) \to 1
\end{equation}
 in which as a $\Sp(b_3(M),\z)$-module $ (H  \times H )/\Delta_{d,k}$ fits into an exact sequence of $\Sp(b_3(M),\z)$-modules
$$0 \to H  \stackrel{i_{2}}{\longrightarrow}  (H  \times H )/\Delta_{d,k} \stackrel{pr_{1}}{\longrightarrow} H \otimes \z/d \to 0$$
where the $\Sp(b_3(M),\z)$-actions on $H$ and $H \otimes \z/d$ are the standard linear actions.

When $b_3(M)=2$, or when $b_3(M) \ge 4$ and $M$ is nonspin with $d/2$ odd, the exact sequence (\ref{seqthm}) splits, therefore
$$\m(M)/\K(M) \cong ((H \times H)/\Delta_{d,k} )\rtimes \Sp(b_3(M),\z).$$

In the other cases the exact sequence (\ref{seqthm}) does not split.
Fix a stable framing of $\# \, r(S^3 \times S^3)$, let $\psi \colon H \to \z/2$ be a quadratic refinement of the standard symplectic form on $H$, and let
$$s \colon \Sp(b_3(M),\z) \to H\otimes \z/2, \ \ A \mapsto \psi \cdot A -\psi$$ 
be the $1$-cocycle associated with $\psi$. Then 
there is an injective homomorphism 
$$\m(M)/\K(M) \to ((H \times H)/\Delta_{d,2k} )\rtimes \Sp(b_3(M),\z)$$
whose image is the subgroup
$$\{([a, b],A) \ | \ b \equiv s(A) \pmod 2 \} \subset (H \times H)/\Delta_{d,2k}\rtimes \Sp(b_3(M), \z).$$
\end{thm}

\begin{rem}
Actually the injective homomorphism and the quadratic refinement $\psi$ are both constructed from a stable framing of $\# \, r(S^3 \times S^3)$. 
\end{rem}

\begin{proof}
The exact sequence (\ref{seqthm}) follows from the exact sequences in Theorem \ref{thm:main1}. We examine the action of $\Sp(b,\z)$ on $(H \times H )/\Delta_{d,k}$.
There is a commutative diagram
\begin{equation}\label{diag:module}
\xymatrix{
0 \ar[r] & H  \ar[r] \ar[d]^{i_{2}} & \m(\#_r(S^{3} \times S^{3}))/\Theta_{7} \ar[d] \ar[r] & \Sp(b,\z) \ar[d]^{=} \ar[r] & 1\\
0 \ar[r] & (H  \times H ) / \Delta_{d,k} \ar[r] & \m(M)/\K(M) \ar[r] & \Sp(b,\z) \ar[r] & 1}
\end{equation}
in which the inclusion $i_{2} \colon H  \to (H \times H )/\Delta_{d,k}$ is a $\Sp(b,\z)$-module homomorphism, and the action of $\Sp(b,\z)$ on $H$ is the standard linear action \cite{Kr78}. In the exact sequence
$$ 0 \to  H \stackrel{i_{2}}{\longrightarrow} (H \times H )/\Delta_{d,k} \stackrel{\mathrm{pr_{1}}}{\longrightarrow} H \otimes \z/d \to 0$$
the $\Sp(b,\z)$-module structure on $H \otimes \z/d$ can be obtained as follows: for any $f \in \T(M)$, $\varphi \in \m(M)$, we consider $(\varphi ^{-1} f \varphi)_{*}-\mathrm{id}$ as a homomorphism from $H$ to $\z/d$. For any $a \in H$, let $\bar a \in \pi_{3}(M)$ be a pre-image of $a$, then
$$(\varphi^{-1}f \varphi)_{*}(\bar a) - \bar a=\varphi^{-1}_{*}(f_{*}(\varphi_{*}(\bar a))-\varphi_{*}(\bar a))=(f_{*}-\mathrm{id})(\varphi_{*}(\bar a))$$
where the second equality follows from the facts that $f_{*}(\varphi_{*}(\bar a))-\varphi_{*}(\bar a) \in \z/d$ and $\varphi^{-1}_{*}$ is trivial on $\z/d$. This shows that the induced action of $\Sp(b,\z)$ on $H\otimes \z/d$ is the standard linear action.

Let $\alpha \in H^{2}(\Sp(b,\z), H)$ be the cohomology class of the extension  of the upper exact sequence in diagram (\ref{diag:module}), then it is a standard fact that $i_{2*}(\alpha) \in H^{2}(\Sp(b,\z),  (H \times H ) / \Delta_{d,k} )$ is the cohomology class of the extension  of the lower exact sequence (see \cite[p.94 Ex.~1(b)]{Brown}).

When $b=2$ we have $\alpha=0$ (by \cite{Kry02}  or the fact $H^2(\Sp(2,\z), \z^2)=0$ \cite[Lemma A.3]{Kra19}), hence $i_{2*}(\alpha)=0$. Therefore $$\m(M)/\K(M) \cong ((H \times H)/\Delta_{d,k} )\rtimes \Sp(b,\z).$$

When $b \ge 4$, the following commutative diagram of exact sequences of $\mathrm{Sp}(b,\z)$-modules
$$\xymatrix{
0 \ar[r] & H \ar[r]^{\cdot d} \ar[d]^{\cdot k} & H \ar[r] \ar[d]^{i_1} & H\otimes \z/d \ar[r] \ar[d]^= & 0 \\
0 \ar[r] & H \ar[r]^{i_2 \ \ \ \ } & (H \times H)/\Delta_{d,k} \ar[r] & H \otimes \z/d \ar[r] & 0}$$
induces a commutative diagram of exact sequence of cohomology groups
$$\xymatrix{
H^1(\Sp(b,\z) ,H\otimes \z/d) \ar[r] \ar[d]^= & H^2(\Sp(b,\z), H) \ar[d]^{\cdot k} \ar[r]^{\cdot d} & H^2(\Sp(b,\z) , H) \ar[d]^{i_{1*}} \\
H^1(\Sp(b,\z) ,H \otimes \z/d) \ar[r]^{\delta} & H^2(\Sp(b,\z) ,H) \ar[r]^{i_{2*} \ \ \ \ \ \ \ } & H^2(\Sp(b,\z) , (H \times H)/\Delta_{d,k})}$$
The cohomology groups $H^*(\Sp(b,\z), H)$ are annihilated by $2$ (\cite[Lemma A.3]{Kra19}). When $M$ is spin, the integers $d$ and $k$ are both even or odd (see (\ref{eq:kspin})). When $d$ is odd, the group $H^1(\Sp(b,\z) ,H \otimes \z/d)$ is $0$; when $k$ is even, the map $k \cdot \colon H^2(\Sp(b,\z) ,H ) \to  H^2(\Sp(b,\z) ,H )$ is trivial. Therefore $\delta=0$, and $i_{2*}$ is injective. By \cite[Theorem 1.1]{Crowley09} $\alpha \ne 0$, therefore $i_{2*}\alpha \ne 0$. When $M$ is nonspin, the integers $d/2$ and $k$ are both even or odd. When $k$ is even, we again have $\delta=0$, and hence $i_{2*}(\alpha) \ne 0$; when $k$ is odd, the homomorphism $d \cdot \colon H^2(\Sp(b,\z) ,H) \to  H^2(\Sp(b,\z) , H)$ is trivial, and the homomorphism $k \cdot \colon H^2(\Sp(b,\z) ,H) \to  H^2(\Sp(b,\z) , H)$ is an isomorphism. This implies $i_{2*}=0$. This shows the statements about the splitting of the exact sequence (\ref{seqthm}).

The cocycle $\alpha \in H^2(\Sp(b,\z), H)$ corresponding to the extension problem of the upper line of diagram (\ref{diag:module}) is described in \cite[Corollary 3.5]{Crowley09}. Choosing a stable framing of $\#_r(S^3 \times S^3)$, the global derivative with respect to this framing induces an embedding $\m(\#_r(S^3 \times S^3))/\Theta_7 \to H \rtimes \Sp(b,\z)$, extending the embedding $2 \cdot \colon H \to H$, 
\begin{equation}
\xymatrix{
0 \ar[r] & H  \ar[r] \ar[d]^{2\cdot} & \m(\#_r(S^{3} \times S^{3}))/\Theta_{7} \ar[d] \ar[r] & \Sp(b,\z) \ar[d]^{=} \ar[r] & 1\\
0 \ar[r] & H \ar[r] & H \rtimes  \Sp(b,\z) \ar[r] & \Sp(b,\z) \ar[r] & 1}
\end{equation}
whose image is
$$\{ (b,A) \in H \rtimes \Sp(b,\z) \ | \ b \equiv s(A) \pmod 2 \},$$
where $s \colon \Sp(b,\z) \to H \otimes \z/2$, $A \mapsto \psi \cdot A - \psi$ is a $1$-cocycle associated with a quadratic refinement $\psi \colon H \to \z/2$ of the standard symplectic form constructed from the stable framing (c.~f.~\cite[Proposition 2.3, Definition 2.4]{Crowley09}). Therefore the embedding 
$$(H \times H)/\Delta_{d,k}=(H \times 2H)/\Delta_{d,2k} \subset (H \times H)/\Delta_{d,2k}$$ extends to an embedding $\m(M)/\K(M) \to ((H \times H)/\Delta_{d,2k}) \rtimes \Sp(b,\z)$, whose image is
$$\{( [a,b],A) \ |\ a,b \in H, A \in \Sp(b,\z), \ b \equiv s(A) \pmod 2 \}.$$
Note that by \cite[Lemma 3.3]{Crowley09}, every quadratic refinement $\psi$ can be obtained from a stable framing. This finishes the proof. 
\end{proof}

\begin{proof}[Proof of Theorem \ref{cor:quintic}]
The quintic hypersurface $X_5$ is spin, with $d=5$, $k=-25$, $l=-5$ and $b_3=204$. By Theorem \ref{thm:main1} and Theorem \ref{thm:kernel}, the center of $\m(X_5)$ is $\K(X_5) \cong \z/2$. Notice that the short exact sequence 
$$0 \to H \stackrel{i_2}{\rightarrow} (H \times H)/\Delta_{(5,-25)} \to H/5 \to 0$$
has a splitting $H/5 \to  (H \times H)/\Delta_{(5,-25)}$, $v \mapsto (v, -5v)$. 
Then the exact sequences in Theorem \ref{thm:main1} give a short exact sequence of $\mathrm{Sp}(204, \z)$-modules
$$0 \to H \times H/5 \to \m(X_5)/\K(X_5)\to \mathrm{Sp}(204, \z) \to 1$$
The embedding of $\m(X_5)/\K(X_5)$ into $(H \times H/5) \rtimes \mathrm{Sp}(204,\z)$ is given by  Theorem \ref{thm:quotient}. 
\end{proof}

\section{On certain simply-connected $7$-manifolds}\label{7-mfds} 
In this section we consider closed, simply-connected spin $7$-manifolds $N$ whose homology groups are isomorphic to that of $S^2 \times S^5$. Let $\mathcal N$ be the set of oriented diffeomorphism classes of these manifolds, by the results of \cite{KrSt}, it is determined by the Kreck-Stolz invariant. We recall the definition of the Kreck-Stolz invariant here. For details see \cite{KrSt}.

The normal $2$-type of $N$ is
$$p \colon B= \cp^{\infty} \times B\mathrm{Spin} \to BSO$$
where $p$ is the canonical projection on the second factor. Choose a generator of $H^2(N)$, there is a unique lifting $\ov{\nu} \colon N \to B$ of the normal Gauss map $\nu \colon N \to BSO$. Since the corresponding $B$-bordism group $\Omega_{7}(B,p)=\omsp_{7}(\cp^{\infty})=0$, there exists a normal $B$-bordism $G \colon W^{8} \to B$ such that $\partial (W, G)=(N, \ov{\nu})$. After taking connected sum with $\hp^{2}$'s (or $\ov{\hp^2}$'s) we may assume that the signature of $W$ is zero.

Let $x \in H^{2}(B)$ be a generator, $\frac{p_{1}}{2} \in H^{4}(B)$ be the pull-back of the universal spin Pontrjagin class in $H^{4}(B\mathrm{Spin})$, $\alpha=G^{*}(x^{2})$, $\beta=G^{*}(\frac{p_{1}}{2})$. Since $H^{3}(N)=H^{4}(N)=0$, $\alpha$ and $\beta$ restrict to $0$ on $\partial W=N$,  therefore we have unique cohomology classes $\ov{\alpha}$, $\ov{\beta} \in H^{4}(W, \partial W)$ which are the preimage of $\alpha$ and $\beta$ respectively. Now we define characteristic numbers
\begin{equation}\label{eqn:ks}
\begin{array}{rcl}
s_1(N) &= & \langle \ov{\beta}^2, [W,\partial W] \rangle
\\
s_2(N) & = & \langle \ov{\alpha}^2 + \ov{\alpha}\cup \ov{\beta}, [W,\partial W] \rangle \\
s_3(N) & = & \langle \ov{\alpha}\cup \ov{\beta}, [W,\partial W] \rangle \\
\end{array}
\end{equation}
These numbers don't depend on the choice of the generator of $H^2(N)$. 
We also need to consider the variation of these characteristic numbers by changing the coboundary $W$. This is in a lattice $L\subset \z^{3}$ generated by the corresponding characteristic numbers of closed $8$-dimensional $B$-manifolds with signature $0$. Therefore we obtain a well-defined invariant 
$$KS(N)=[(s_1(N),s_2(N),s_3(N))] \in \z^3/L.$$ 

A set of generators of the corresponding $8$-dimensional $B$-bordism group is 
\begin{equation}\label{eqn:gen1}
\begin{array}{rcl}
e_1 & = & [B^8]-8\cdot 28 [\hp^2] \\
e_2 & = & [S^2 \times S^2 \times S^2 \times S^2 \stackrel{\sum_{i=1}^4 x_i}{\longrightarrow} \cp^{\infty} \stackrel{i}{\rightarrow} \cp^{\infty} \times K(\z,4)] \\
e_3 & = & [V(2) \hookrightarrow \cp^{\infty} \stackrel{i}{\rightarrow} \cp^{\infty} \times K(\z,4)] - 2 [\hp^2] 
\end{array}
\end{equation}
where $B^8=\natural_{28}M(E_8) \cup D^8$ with $M(E_8)$ the $E_8$-plumbing manifold,  $V(2)$ is a degree $2$ hypersurface in $\cp^5$, and $x_i \in H^2(S_{i}^2)$ are generators. The characteristic numbers are give in the following table
\begin{equation}\label{lattice1}
\begin{array}{c|ccc}
& e_1 & e_2 & e_3 \\
\hline
s_1 &  -8 \cdot 28 & 0 & 0 \\
s_2 & 0 & 24 & 0 \\
s_3 &0
 & 0 &  -2
\end{array}
\end{equation}
The lattice $L$ is generated by the column vectors of the table, and 
$$\z^{3} / L \cong \z/(8 \cdot 28)\oplus \z/24 \oplus \z/2.$$ 

Let $(W,G)$ be a $B$-bordism between $(N_1, \ov{\nu}_1)$ and $(N_2,\ov{\nu}_2)$, then the surgery obstruction $\theta(W,G)$ to $(W,G)$ being $B$-bordant to an $h$-cobordism is in $l_{8}(\z)$. By \cite[Theorem 6]{Kr99}, $\theta(W,G)$ is elementary if $KS(N_1)=KS(N_2)$.
Therefore $KS(N)$ is a complete invariant of manifolds under consideration. 

\begin{lem}\label{lem:inv}
There is an injective map between sets
$$KS=[(s_1,s_2, s_3)] \colon \mathcal N \to \z^3/L = \z/(8 \cdot 28) \oplus \z/24 \oplus \z/2$$
\end{lem}

\begin{rem}\label{rem:Ss}
In \cite{KrSt} the Kreck-Stolz invariants $S_i \in \mathbb Q /\mathbb Z $ ($i=1,2,3$) are defined by the defect of (twisted) Dirac operators. An explicit formula expressing the $S_i$-invariants by characteristic numbers is given in \cite[(2.4)]{KrSt2}. The relation between the $S_i$-invariants and the $s_i$-invariants defined here is $$s_1/(8\cdot 28)=S_1, \ \ \ s_2/24=S_2 \ \ \ s_3/2=S_3-2^4S_2$$
\end{rem}

\smallskip

Next we construct certain manifolds in $\mathcal N$ and compute their $KS$-invariant, which will be used in the next section.
Let $\xi_m$ be a $4$-dimensional oriented real vector bundle over $\cp^2$, with
$$w_2(\xi_m) \equiv x \pmod 2, \ \ \ e(\xi_m)=x^2, \ \ \ p_1(\xi_m)=(3+4m)x^2, \ \ \ m \in \z$$
where $x \in H^2(\cp^2)$ is a generator. (These bundles can be constructed as follows: let $p \colon \cp^2 \to \cp^2 \vee S^4$ be the pinch map, $\xi=p^*(\eta \oplus \varepsilon^2 \vee \xi')$, where $\eta$ is the tautological line bundle over $\cp^2$ and $\xi'$ is a $4$-dimensional vector bundle over $S^4$. Then $w_2(\xi)\ne 0$, $e(\xi)=p^*e(\xi')$ and $p_1(\xi)=p^*(p_1(\xi'))+x^2$. Note that $4$-dimensional vector bundles over $S^4$ can have $e(\xi')=2a+b$ and $p_1(\xi')=-2b$ for all $a$, $b \in \z$.)

Let $N_m=S(\xi_m)$ be the sphere bundle of $\xi_m$. Since the Euler class of $\xi_{m}$ is $x^{2}$, from the Gysin sequence of the sphere bundle $S(\xi_m)$ one gets $H_*(N) \cong H_*(S^2 \times S^5)$, hence  $N_m \in \mathcal N$. Using $W=D(\xi_m) \# \ov{\hp^2}$ we compute the Kreck-Stolz invariant of $N_m$  from equations (\ref{eqn:ks}) and get
\begin{eqnarray}\label{eqn:nk}
s_1(N_m) & = & 4m(m+3)  \nonumber  \\
s_2(N_m) & = & -2(m+1)   \\
s_3(N_m) & = & 2m+3  \nonumber
\end{eqnarray}

\section{Generalized Kreck-Stolz invariant}\label{sec:gks}
Now we come to the problem of determining the kernel 
$$\K(M)=\mathrm{Ker} \{ v_{x,p} \colon \T(M) \to H \times H / \Delta_{d,k}\}.$$ 
This task will be done by defining an invariant of the mapping torus. This invariant is not only crucial for the computation of $\mathcal K(M)$, but also important for the study of the properties of $\m(M)$. We first present the results of $\K(M)$ in the following theorem. Recall from \S \ref{sec:diff} that the first Pontrjagin class of $M$ is $p_1(M)=(24 l(M) + 4d(M))y$, where $l(M)$ is an integer. In the sequel we denote $l(M)$ by $l$ for simplicity.

\begin{thm}\label{thm:kernel}
The group $\K(M)$ is the center of $\m(M)$. There is an isomorphism 
$$S \colon \K(M) \to \z^3 / L(M),$$ 
where $L(M)$ is a lattice of rank $3$. The group $\K(M)$ is given as follows
\begin{enumerate}
\item When $M$ is spin 
$$\K(M) \cong (\K(M))_{(2)} \times \z/3^{b} \times \z/7^{c}$$
where  
$$b = \left \{ \begin{array}{cl}
1 & d \equiv 0 \pmod 3 \\
0 & \mathrm{otherwise} \end{array} \right. \ \ \ c = \left \{ \begin{array}{cl}
1 & d- l \equiv 0 \pmod 7 \\
0 & \mathrm{otherwise} \end{array} \right.$$
The group $(\K(M))_{(2)}$ is given in the following table
$$\begin{array}{c|c|c|c|c}
 & d \mathrm{\ odd} & d \equiv 2 \pmod 4 & d \equiv 4 \pmod 8 & d \equiv 0 \pmod 8 \\
\hline
l \mathrm{\ odd } & \z/2 & \z/2 \times \z/4 & (\z/2)^{2}& (\z/2)^{2}    \\
\hline
l \mathrm{\ even } &\z/4 & \z/2 \times \z/4 & (\z/2)^{2} \times \z/4 &  \z/2 \times \z/a \times \z/4

\end{array}
$$
where
$$ a = \left \{ \begin{array}{ll}
2 &  l \equiv 2 \pmod {4}  \\
4 &  l \equiv 0 \pmod {4} \\
\end{array} \right.   $$

\item When $M$ is non-spin 
$$\K(M) \cong \z/2^{a} \times \z/3^{b} \times \z/7^{c},$$ 
where
$$b = \left \{ \begin{array}{cl}
1 & d \equiv 0 \pmod 3 \\
0 & \mathrm{otherwise} \end{array} \right. \ \ \ c = \left \{ \begin{array}{cl}
1 & d \equiv l \equiv 0 \pmod 7 \\
0 & \mathrm{otherwise} \end{array} \right.$$
The value of $a$ is given in the following table (where $\delta =d/2$)
$$\begin{array}{c|c|c|c|c}
& \delta \mathrm{\ odd} & \delta \equiv 2 \pmod 4 & \delta \equiv 4 \pmod 8 & \delta \equiv 0 \pmod 8 \\
\hline
l \mathrm{\ odd} & 0 & 1 & 2 & 2 \\
\hline
l \mathrm{\ even} & 0 & 1 & 2 & 3 \\
\end{array}$$
\end{enumerate}

\end{thm}

We first prove Theorem \ref{thm:kernel} in the spin case in \S \ref{sec:gks}-\ref{sec:bd}. The treatment of the non-spin case is similar but with slight modifications. This is put in \S \ref{sec:nonspin}.

\medskip

For the mapping class group of surfaces, the Birman-Craggs-Johnson invariants are defined using the relative mapping class group $\m(F_g,D)$ which is the mapping class group of diffeomorphisms fixing an embedded disk pointwise. In our situation we have to do the same and pass to $\m(M,D)$. The group $\m(M,D)$ is the group of isotopy classes of diffeomorphisms of $M$ which are the identity on an embedded $6$-dimensional disk $D \subset M$. The groups $\m(M)$ and $\m(M,D)$ are related by the following exact sequence (\cite[p.265]{Wall1})
\begin{equation}\label{seq:br}
\pi_1GL(6,\mathbb R) \to \m(M,D) \to \m(M) \to 1.
\end{equation}
The image of the generator of $\pi_1(GL_6(\mathbb R))$ in $\m(M,D)$ is called the \textbf{boundary mapping class}, it is represented by the following \textbf{boundary diffeomorphism}:
let $\alpha \colon [0,1] \to SO(6)$ be a loop representing a generator of $\pi_{1}SO(6)$, identify a collar neighborhood of $\partial (M-\mathring D)$ with $S^{n-1} \times [0,1]$, define
$$\delta \colon S^{n-1} \times [0,1] \to S^{n-1} \times [0,1], \ \ \ (x,t) \mapsto (\alpha (t) \cdot x, t)$$
and extend it to a diffeomorphism of $M-\mathring D$ by identity on the complement of the collar neighborhood. 

Since the representation $\rho \colon \m(M) \to \mathrm{Sp}(b_3(M),\z)$ is surjective, the homomorphism $\m(M,D) \to \Sp(b_3(M), \z)$ is also surjective, whose kernel  we denote by $\T(M,D)$. Completely analogous to \S \ref{subsec:var} we define the ``variation of the cup product and the Pontrjagin class" homomorphism $v_{x,p} \colon \T(M,D) \to (H \times H)/\Delta_{d(M),k(M)}$, whose kernel we denote by $\K(M,D)$.  Then the exact sequence (\ref{seq:br}) reduces to 
 the exact sequence
\begin{equation}\label{seq:br1}
\pi_1GL(6,\mathbb R) \to \K(M,D) \to \K(M) \to 1. 
\end{equation}
The forgetful map $\K(M,D) \to \K(M)$ is an isomorphism if and only if the boundary mapping class is trivial in $\K(M,D)$.

In this section, based on the recent development of modified surgery theory in \cite{Kr18}, we define a Kreck-Stolz type invariant for elements in $\K(M, D)$. This invariant defines an injective homomorphism
$$S \colon \K(M,D) \rightarrowtail \z^3/\bar L(M),$$
where $\bar L(M)$ is a subgroup of finite index, which will be precisely described in the paragraph before Proposition \ref{prop:injhomo}.
In \S \ref{sec:rel} we use this invariant to study the group $\m(S^{2} \times D^{4}, \partial)$, which is the mapping class group of diffeomorphisms which are the identity in a neighborhood of the boundary, and where isotopies are also the identity in a neighborhood of the boundary.
In \S \ref{sec:K} we determine the image of $S$ by comparing $\K(M,D)$ with $\m(S^{2} \times D^{4}, \partial)$.

\smallskip

In the following we extend the class of manifolds under consideration slightly. Namely, we consider simply-connected $6$-manifolds whose even dimensional homology groups are isomorphic to that of $\cp^3$, i.e.~we include the case where the degree of the manifold equals $0$. The reason for this extension is that the case of $S^2 \times S^4$ is crucial for the computation. We first deal with the spin case. Let $M$ be such a manifold, $f \colon (M, D) \to (M, D)$ be a diffeomorphism such that the isotopy class of $f$, denoted by  $[f]$, is in $\K(M,D)$. Let $M_{f}$ be the mapping torus. Let
$$B=\cp^{\infty} \times K(\z,4) \times B\mathrm{Spin},$$
and let $p \colon B \to B\mathrm{Spin} \to  BSO$ be the canonical projection. Then the normal Gauss map $M_{f} \to BSO$ is lifted to $B$ as follows: Fix a generator $x \in H^2(M)=H^{2}(M_{f})$, take a map $M_{f} \to \cp^{\infty}$ representing $x$, which is unique up to homotopy. For $d \ne 0$, let $y=x^{2}/d \in H^{4}(M)$ be the generator; for $d=0$ but but $p_{1}(M) \ne 0$, let $y=p_{1}(M)/4k \in H^{4}(M)$; for $d=0=p_{1}(M)$, take $y \in H^{4}(M)$ an arbitrary generator. For $[f] \in \K(M,D)$, since $v_{x,p}(f)=0$,  there is a unique $y' \in H^{4}(M_{f})$ which is a preimage of $y \in H^{4}(M)$ under $H^{4}(M_{f}) \to H^{4}(M)$, such that $x^{2}-dy'=0$ and $\frac{p_{1}}{2}(M_{f})-2ky'=0$.  Take a map $M_{f} \to K(\z,4)$ representing $y'$,  which is unique up to homotopy. Finally choose the spin structure on the stable normal bundle of $M_{f}$, which restricts to the trivial spin structure on the preferred embedding $D \times S^{1} \subset M_{f}$. This induces a map $M_{f} \to B\mathrm{Spin}$,  which is unique up to vertical homotopy covering the classifying map of the stable normal bundle of $M_f$. Putting all these maps together we get $g \colon M_{f} \to B$ which is a lift of the normal Gauss map $M_{f}\to BSO$.

The pair $(M_f, g)$ represents a bordism class  in the $B$-bordism group $\Omega_{7}(B,p)= \omsp_{7}(\cp^{\infty} \times K(\z,4))$, which is zero by a  computation with the Atiyah-Hirzebruch spectral sequence, c.~f.~\cite[Theorem 6]{Kr18}. Therefore there exists a normal $B$-bordism $G \colon W^{8} \to B$ such that $\partial (W, G)=(M_{f}, g)$. After taking connected sum with quaternionic  projective planes $\hp^2$'s (or $\ov{\hp^2}$'s) we may assume that the signature of $W$ is zero.

Let $\ov x \in H^2(B)$ be the pull-back of the canonical generator of $H^{2}(\cp^{\infty})$, $\ov y \in H^4(B)$ be the pull-back of the canonical generator  of $H^4(K(\z,4))$ and $\frac{p_1}{2} \in H^4(B)$ be the pull-back of the universal spin Pontrjagin class in $H^4(B\mathrm{Spin})$. Let 
$$\alpha = G^*(\ov x^2-d\ov y), \ \ \beta =G^*(\frac{p_1}{2}+2k\ov y).$$ 
Then by the construction above, the restriction of  $\alpha$ and $\beta$ on $\partial W=M_{f}$ are zero, so we have relative cohomology classes $\ov{\alpha}$, $\ov{\beta}\in H^4(W, \partial W)$ (not canonical), which maps to $\alpha$ and $\beta$ respectively. Now define characteristic numbers
$$\begin{array}{rcl}
s_1 & = & \langle \ov{\beta}^2, [W,\partial W] \rangle, \\
s_2 & = & \langle \ov{\alpha}^2 + \ov{\alpha}\cup \ov{\beta}, [W,\partial W] \rangle, \\
s_3 & = & \langle \ov{\alpha}\cup \ov{\beta}, [W,\partial W] \rangle. 
\end{array}$$
Elementary algebraic topology shows that these numbers don't depend on the choices of the cohomology classes $\ov{\alpha}$ and $\ov{\beta}$. These characteristic numbers depend on the choice of $W$. The variation of these characteristic numbers for different $W$ is in a lattice $\bar L(M) \subset \z^{3}$,  consisting of characteristic numbers of closed $8$-dimensional $B$-manifolds with signature $0$.  We get a well-defined invariant 
$$S([f])=[(s_1,s_2,s_3)] \in \z^3/\bar L(M).$$ 
called the generalized Kreck-Stolz invariant of $f$.

To determine the lattice $\bar L(M)$ we need to compute the subgroup of the $8$-dimensional $B$-bordism group $\Omega_8(B)=\omsp_8(\cp^{\infty} \times K(\z,4))$, consisting of manifolds whose signature is zero. Computation with the Atiyah-Hirzebruch spectral sequence (c.~f.~\cite[Theorem 10]{Kr18}) shows that this group is isomorphic to $\z^6$, with generators
\begin{equation}\label{gen:2}
\begin{array}{rcl}
e_1 & = & [B^8]-8\cdot 28 [\hp^2] \\
e_2 & = & [S^2 \times S^2 \times S^2 \times S^2 \stackrel{\sum_{i=1}^4 x_i}{\longrightarrow} \cp^{\infty} \stackrel{i}{\rightarrow} \cp^{\infty} \times K(\z,4)] \\
e_3 & = & [V(2) \hookrightarrow \cp^{\infty} \stackrel{i}{\rightarrow} \cp^{\infty} \times K(\z,4)] - 2 [\hp^2] \\
e_4 & = & [\hp^2 \stackrel{y}{\rightarrow} K(\z,4) \to \cp^{\infty} \times K(\z,4)]-[\hp^2] \\
e_5 & = & [S^4 \times S^4 \stackrel{y_1+y_2}{\longrightarrow} K(\z,4)] \\
e_6 & = & \frac{1}{2}[S^2 \times S^2 \times S^4 \stackrel{(x_1+x_2, y)}{\longrightarrow} \cp^{\infty} \times K(\z,4)]
\end{array}
\end{equation}
where the first three elements are the same as Equation (\ref{eqn:gen1}),  $y_i \in H^4(S_{i}^4)$ and $y \in H^4(\hp^2)$ are generators.

The characteristic numbers of these generators are given in the following table, and the lattice $\bar L(M)$ is generated by column vectors.
\begin{equation}\label{lattice}
\begin{array}{c|cccccc}
 & e_1 & e_2 & e_3 & e_4 & e_5 & e_6 \\
 \hline  s_1 & -8 \cdot 28 & 0 & 0 & 4k^2 -4k & 8k^2 & 0 \\
 s_2 & 0 & 24 & 0 & d^2-d(2k-1) & 2d^2-4dk & 2k-2d \\
s_3 & 0 & 0 & -2 & -d(2k-1) & -4dk& 2k \\
\hline
 \end{array}
 \end{equation}


\smallskip

\begin{prop}\label{prop:injhomo}
The generalized Kreck-Stolz invariant defines an injective homomorphism
$$S \colon \K(M, D) \to \z^3/\bar L(M)$$
\end{prop}
\begin{proof}
We first show that $S$ is a group homomorphism. Given diffeomorphisms $f_{1}$ and $f_{2}$, with normal $B$-bordism $(W_{1}, G_{1})$ and $(W_{2}, G_{2})$ for the corresponding mapping tori. Note that we have the canonical embeddings of $M \times [1/3, 2/3]$ into the mapping tori. We may take $W=W_{1} \cup_{M \times [1/3,2/3]} W_{2}$ to be the normal $B$-bordism for the mapping torus of $f_1\circ f_2$.  In this situation, the correction term $\sigma(V;A,B,C)$ in Wall's formula of the additivity of the signature (\cite[Theorem]{Wall2}) vanishes, we have $\mathrm{sign}(W)=\mathrm{sign}(W_1) + \mathrm{sign}(W_2)=0$. Therefore $W$ can be taken as  the null-bordism for computing the $S$-invariant of $f_{1}\circ f_{2}$. 
Let $\alpha$, $\beta \in H^4(W)$, $\alpha_j$, $\beta_j \in H^4(W_j)$ ($j=1,2$) be the classes  as in the definition of the $S$-invariant, $i_j^* \colon H^4(W) \to H^4(W_j)$ be the homomorphism induced by the inclusion map, then we have $i_j^*(\alpha) = \alpha_j$, $i_j^*(\beta)=\beta_j$. Let $\ov{\alpha}$, $\ov{\beta} \in H^4(W, \partial W)$ be a preimage of $\alpha$, $\beta$, respectively, then from the commutative diagram 
$$\xymatrix{
H^{4}(W, \partial W) \ar[d] \ar[r]^{(\bar i_1^*, \bar i_2^*) \ \ \ \ \ \ \ \ \ \ \ \ } & H^{4}(W_{1}, \partial W_{1}) \oplus H^{4}(W_{2}, \partial W_{2}) \ar[d] \\
H^{4}(W) \ar[r]^{(i_1^*, i_2^*) \ \ \ \ \ \ \ } & H^{4}(W_{1}) \oplus H^{4}(W_{2})}$$
we see that the image of $\ov{\alpha}$ and $\ov{\beta}$ in $H^4(W_j, \partial W_j)$, denoted by $\bar \alpha_j$ and $\bar \beta_j$,  is a preimage of $\alpha_j$ and $\beta_j$, respectively.
From this one concludes that,  the generalized Kreck-Stolz invariant of $f_{1} \circ f_{2}$, defined as characteristic numbers of $W$, is the sum of the characteristic numbers of $W_{1}$ and $W_{2}$. For example, we have 
\begin{eqnarray*}
\langle \bar \beta^2, [W, \partial W] \rangle & = &\langle \bar \beta^2, \bar i_{1*}[W_1, \partial W_1]+\bar i_{2*}[W_2, \partial W_2] \rangle \\
& = & \langle \bar \beta_1^2, [W_1, \partial W_1\rangle + \langle \bar \beta_2^2, [W_2, \partial W_2] \rangle.
\end{eqnarray*}

To show the injectivity of $S$, notice that $(W,G)$ can be viewed as a $B$-bordism between $M \times [0,1]$ and $M \times [0,1]$, where the boundaries $M \times 0$ and $M \times 1$ are identified by $\mathrm{id}_{M}$ and $f$ respectively. By modified surgery theory \cite{Kr99}, the surgery obstruction $\theta(W,G)$ to $(W,G)$ being $B$-bordant to an $h$-cobordism relative boundary is in $l_{8}(\z)$. It is shown in \cite[Theorem 5b]{Kr18} that $\theta(W,G)$ is elementary if
$$S([f])=[(s_1(f), s_2(f), s_3(f))]=0\in \z^3/\bar L(M)$$
In this situation one may replace $W$ by an $h$-cobordism (\cite[Theorem 3]{Kr99}), still denoted by $W$.  The embedding $D \times S^{1} \subset M_{f}$ extends to an embedding $D \times D^{2} \subset W$, which can be viewed as an $h$-cobordism between $D \times I$ and $D \times I$. By the $h$-cobordism theorem, there is a diffeomorphism
$$(M \times [0,1], D \times [0,1]) \to (M \times [0,1], D \times [0,1])$$
which can be viewed as a pseudo-isotopy between $f$ and $\mathrm{id}_{M}$ rel $D$. Hence by Cerf's pseudo-isotopy theorem \cite{Cerf}, $f$ is isotopic to $\mathrm{id}_{M}$ rel $D$. This shows the injectivity of $S$.
\end{proof}

\section{The mapping class group of $(S^2 \times D^4, \partial)$}\label{sec:rel}

In this section we determine the mapping class group of $(S^{2} \times D^{4}, \partial )$, denoted by $\m(S^2 \times D^4, \partial)$. This is the group of isotopy classes of diffeomorphisms of $S^2 \times D^4$, which are the identity near the boundary.  This is crucial for the determination of the image of $S$, and the generators of the mapping class group.

\subsection{The group $\T(S^2 \times S^4)$}\label{subsec:s2}
Let $\T(S^2 \times S^4)$ be the subgroup of the mapping class group $\m(S^2 \times S^4)$, consisting of isotopy classes of diffeomorphisms of $S^2 \times S^4$ which act trivially on homology. Recall that $\mathcal N$ is the set of  diffeomorphism classes of closed oriented simply-connected spin $7$-manifolds $N$ with $H_*(N)$ isomorphic to $H_*(S^2 \times S^5)$ (see \S  \ref{7-mfds}). We define a map
$$\Psi \colon \T(S^2 \times S^4) \to \mathcal N$$
as follows: given a diffeomorphism $f \colon S^2 \times S^4 \to S^2 \times S^4$ which acts trivially on homology, let $N=(S^2 \times D^5) \cup_f (S^2 \times D^5)$. Then $N$ is a simply-connected spin manifold with $H_*(N) \cong H_*(S^2 \times S^5)$. The map $\Psi$ is  well-defined
since isotopic diffeomorphisms produce diffeomorphic $7$-manifolds.

\begin{lem}\label{lem:bij}
$\Psi$ is surjective.
\end{lem}
\begin{proof}
Given $N \in \mathcal N$, there is an embedding $\varphi_1 \colon S^2 \to N$ representing a generator of $H_2(N)$. The normal bundle of this embedding is trivial since $N$ is spin, thus $\varphi_1$ extends to an embedding $\varphi_1 \colon S^2 \times D^5 \hookrightarrow N$. The complement $N-\varphi_1(\mathrm{int}(S^2 \times D^5))$ is simply-connected, spin, and has the same homology as $S^2$. Therefore there is a second embedding $\varphi_2 \colon S^2 \times D^5 \hookrightarrow N-\varphi_1(\mathrm{int}(S^2 \times D^5))$. By computing the relative homology groups one shows that $N-(\varphi_1(\mathrm{int}(S^2 \times D^5))\cup \varphi_2(\mathrm{int}(S^2 \times D^5)))$ is an $h$-cobordism. Therefore there exists a diffeomorphism $f \colon S^2 \times S^4 \to S^2 \times S^4$ such that $N = S^2 \times D^5 \cup_f S^2 \times D^5$. Since $H_*(N) \cong H_*(S^2 \times S^5)$, the induced homomorphism of $f$ on homology must be trivial. 
\end{proof}

Let $\T(S^2 \times S^4, D)$ be the subgroup of the mapping class group $\m(S^2 \times S^4,D) $, consisting of those isotopy classes whose action on $H_*(S^2 \times S^4)$ is trivial. Recall in last section we defined the $S$-invariant $S \colon \T(S^2 \times S^4, D) \to \z^3/\bar{L}(S^2 \times S^4)$. The lattice $\bar{L}(S^2 \times S^4)$ is given by table \ref{lattice} for $d=k=0$. Hence it equals the lattice $L$ defined by table \ref{lattice1}.
By the following lemma, the forgetful map $\T(S^{2} \times S^{4}, D) \to \T(S^{2} \times S^{4})$ is an isomorphism. Therefore we identify these two groups and view $S$ as an injective homomorphism $S \colon \T(S^{2} \times S^{4}) \to \z^3/L$.

\begin{lem}\label{lem:S1action}
Let $M^{n}$ ($n \ge 3$) be a closed smooth manifold with a non-trivial $S^{1}$-action. Let $p \in M$ be a fixed-point, such that the isotropy representation at $p$ represents the non-trivial element in $\pi_{1}GL(n,\mathbb R)$. Then the boundary diffeomorphism is isotopic to the identity in $\m(M, D)$
\end{lem}
\begin{proof}
We take an equivariant neighborhood $D$ of the fixed-point $p$, s.~t.~$D \cong \mathring D^n$. The $S^1$-action induces a map $g \colon (M -D) \times [0,1] \to (M-D) \times [0,1]$, whose restriction on $(M-D) \times \{1\}$ is the identity, and whose restriction on $((M-D) \times \{0\} )\cup (\partial D \times [0,1])$ is the boundary diffeomorphism. After reparametrization of the boundary of $(M-D) \times [0,1]$, $g$ can be viewed as a pseudo-isotopy between $\mathrm{id}$ and the boundary diffeomorphism.
\end{proof}

\begin{lem}\label{lem:IN}
The following diagram is commutative.
\begin{eqnarray}\label{diag:sn}
\xymatrix{
\T(S^{2} \times S^{4}) \ar[r]^{\ \ \ \ \ \Psi} \ar[d]^{S} & \mathcal N \ar[d]^{KS} \\
\z^{3}/L \ar[r]^{=} & \z^{3}/L. }
\end{eqnarray}
\end{lem}
\begin{proof}
Let $V \to B$ be a coboundary of the mapping torus $(S^2 \times S^4)_{f}$ which is used to define the $S$-invariant of $f$. There is a canonical embedding $S^2 \times S^4 \times [0,1/2] \subset (S^2 \times S^4)_{f}$.  Let
$$W=V \cup_{S^{2}\times S^{4} \times [0,1/2]} (S^{2} \times D^{5} \times [0,1/2]),$$
then $\partial W=(S^{2} \times D^{5} )\cup_{f}  (S^{2} \times D^{5})=N$. Analogous to the proof of Proposition \ref{prop:injhomo}, Wall's formula (\cite[Theorem]{Wall2}) guarantees that the signature of $W$ is zero. Note that the mormal map $S^2 \times S^4 \to B$ extends to a normal map $S^2 \times D^5 \to B$ in the natural way. We use $W$ to compute the $KS$-invariant of $N$. The fundamental class of $(W, \partial W)$ equals the sum of the image of the fundamental classes of $(V,\partial V)$ and $(S^2 \times D^5 \times [0,1/2], \partial (S^2 \times D^5 \times [0,1/2]))$ under the inclusion map, but $H^4(S^2 \times D^5 \times [0,1/2], \partial ( S^2 \times D^5 \times [0,1/2]) )=0$. Therefore the characteristic numbers of $W$ equal that of $V$. 
\end{proof}

\subsection{The group $\m(S^2 \times D^4, \partial)$}\label{sec:product} 
Using the standard decomposition $S^{2} \times S^{4}=(S^{2} \times D^{4} )\cup (S^{2} \times D^{4})$ we define a homomorphism $ \m(S^{2} \times D^{4}, \partial) \to \T(S^{2} \times S^{4})$ by extending a diffeomorphism on the first copy of $S^2 \times D^4$, which is the identity near the boundary, to a diffeomorphism of $S^2 \times S^4$ by the identity on the second copy of $S^2 \times D^4$.

\begin{lem}
The homomorphism $ \m(S^{2} \times D^{4}, \partial) \to \T(S^{2} \times S^{4})$ is surjective.
\end{lem}
\begin{proof}
For any diffeomorphism $f$ of $S^{2} \times S^{4}$ which acts trivially on homology, since $\pi_{2}SO(4)=0$, by \cite[Theorem 1(b)]{Hae61}, after an isotopy we may assume that $f$ is the identity on the second copy of $S^{2} \times D^{4}$. 
\end{proof}



Taking an embedded $D^6$ in the interior of $S^2 \times D^4$, we have a homomorphism $\m(D^6, \partial) \to \m(S^2 \times D^4, \partial)$ given by extending a diffeomorphism of $(D^6,\partial)$ to the complement by the identity. The image are called disk supported diffeomorphisms. 
\begin{lem}\label{lem:d6}
The homomorphism $\m(D^6,\partial) \to  \m(S^2 \times D^4, \partial)$ is injective.
\end{lem}
\begin{proof}
Consider the composition 
$$\m(D^6,\partial) \to \m(S^2 \times D^4, \partial) \to \T(S^2 \times S^4).$$
We compute the generalized Kreck-Stolz invariant of disk supported diffeomorphisms in $\T(S^2 \times S^4)$. By diagram \ref{diag:sn} this is equivalent to computing the $KS$-invariant of $(S^2 \times S^5 )\# \Sigma$, where $\Sigma$ is an exotic $7$-sphere corresponding to a given diffeomorphism of $(D^6,\partial)$. Especially, let  $\Sigma$ be the boundary of the $E_8$-plumbing manifold $E_8$,  which is a generator of $\Theta_7$. Then we may take
$$W=((S^2 \times D^6)\natural E_8 )\# (8 \cdot \ov{\hp^2}) \to \cp^{\infty} \times B\mathrm{Spin}$$
as the null $B$-bordism,
where $\natural$ denotes the boundary connected-sum. Then from the equations (\ref{eqn:ks}) one may directly calculate the Kreck-Stolz invariant of $(S^2 \times S^5) \# \Sigma$
$$s_1((S^2 \times S^5) \# \Sigma)=-8, \ \ s_2((S^2 \times S^5) \# \Sigma)=0, \ \ s_3((S^2 \times S^5) \# \Sigma)=0.$$
Therefore the homomorphism $\m(D^6, \partial) \to \T(S^2 \times S^4)$ is injective, and so is the homomorphism $\m(D^6, \partial) \to \m(S^2 \times D^4, \partial)$.
\end{proof}

The next step to the computation of $\m(S^2 \times D^4, \partial)$ is to determine  the cokernel of $\m(D^{6}, \partial) \to \m(S^{2} \times D^{4}, \partial)$, denoted by $\ov{\m}(S^{2} \times D^{4}, \partial)$.   An upper bound of the order of $\ov{\m}(S^{2} \times D^{4}, \partial)$ is obtained as follows.

Let $P_3(S^2)$ be the 3rd stage Postnikov tower of $S^2$,  we define a homomorphism
$$\Phi \colon \omst_7(P_3(S^2)) \to \overline{\m}(S^2 \times D^4, \partial)$$
as follows. The normal $3$-type of $S^2 \times D^4$ is $$B=P_3(S^2)\times B\mathrm{String} \stackrel{p}{\rightarrow} BSO$$
where $p$ is the canonical projection on the second factor.
Take a generator of $H^2(S^2 \times D^4)$, there is a unique (up to  vertical homotopy) normal $3$-smoothing $\ov{\nu} \colon S^2 \times D^4 \to B$. The corresponding $B$-bordism group $\Omega_*(B,p)$ is the string bordism group $\omst_*(P_3(S^2))$. Given $[N^7,g] \in \omst_7(P_3(S^2))$, consider the bordism class $[S^2 \times D^4\times I, \ov{\nu}]+[N,g]$, where $I=[0,1]$. View this as a $B$-bordism between $(S^2 \times D^4,\ov \nu)$ and $(S^2 \times D^4,\ov \nu)$, where the boundaries are identified by the identity. By \cite[p.708]{Kr99} the surgery obstruction to this $B$-bordism being $B$-bordant rel $\partial$ to an $h$-cobordism rel $\partial$ is in $L_7(\z)=0$. Thus we obtain an $h$-cobordism rel $\partial$ between $S^2 \times D^4$ and $S^2 \times D^4$, and hence a diffeomorphism $f \colon (S^2 \times D^4, \partial) \stackrel{\cong}{\rightarrow} (S^2 \times D^4, \partial)$. We define $\Phi([N,g]) = [f] \in \ov{\m}(S^{2} \times D^{4}, \partial)$.

To consider the dependence of this construction on the choices of representatives of a given $B$-bordism class, let $(N,g)$ and $(N' ,g')$ be two such representatives, then the $h$-cobordisms $W$ and $W'$ obtained by surgery are also $B$-bordant, say, through $V^8$ (see Figure \ref{fig:1}).
\begin{figure}
\begin{center}
\mbox{\beginpicture \setcoordinatesystem units <0.018in,0.018in>
\setplotarea x from -50 to 50 , y from -40 to 40 %
\plot -10 30 40 30 10 10 -40 10 -10 30 / \plot -40 -30 10 -30 40 -10 /
\plot -40 10 -40 -30 10 -30 40 -10 40 30 / \plot  10 -30 10 10 /
\setdashes<1.5pt>
\plot -40 -30 -10 -10 -10 30 / \plot -10 -10 40 -10 /
\put {\mbox{$W$}} [cc] <0mm,0mm> at 0 20 %
\put {\mbox{$W'$}} [cc] <0mm,0mm> at 0 -20 %
\put {\mbox{$I$}} [cc] <0mm,0mm> at 15 38 %
\put {\mbox{$I$}} [cc] <0mm,0mm> at 45 10 %
\put {\mbox{$S^2 \times D^4$}} [cc] <0mm,0mm> at 28 20 %
\put {\mbox{$S^2 \times D^4$}} [cc] <0mm,0mm> at 28 -20 %
\put {\mbox{$V=$}} [cc] <0mm,0mm> at -50 0
\endpicture}
\end{center}
\caption{}  \label{fig:1}
\end{figure}
Then the surgery obstruction $\theta(V)$ to $V$ being an $h$-cobordism rel $\partial$ lies in $L_8(\z) \cong \z$. After making a boundary connected sum of $V$ with copies of $E_8$ along $W$, the surgery obstruction vanishes, and $W$ is replaced by $W \sharp \Sigma$ for some $\Sigma \in \Theta_7$. This means that we obtain a pseudo-isotopy between the diffeomorphism produced by $(N,g)$ composed with the disk-supported diffeomorphism corresponding to $\Sigma$ and the diffeomorphism produced by $(N',g')$. This shows that $\Phi$ is well-defined.

\begin{lem}\label{lem:surj}
$\Phi \colon \omst_{7}(P_{3}(S^{2}) )\to \ov{\m}(S^{2} \times D^{4}, \partial)$ is a surjective homomorphism.
\end{lem}
\begin{proof}
We first show that $\Phi$ is a homomorphism.  We have\begin{eqnarray*}
& & [S^2 \times D^4 \times I, \ov{\nu} ] + [N_1, g_1]+[N_2,g_2] \\
& = & [((S^2 \times D^4 \times [0,1], \ov{\nu} ) +(N_1,g_1)) \cup_{S^2 \times D^4 \times \{1\}} ((S^2 \times D^4 \times [1,2], \ov{\nu} )+ (N_2,g_2))]
\end{eqnarray*}
After surgery we obtain an $h$-cobordism rel $\partial$, which is the union of the $h$-cobordisms obtained from $(N_1,g_2)$ and $(N_2,g_2)$ along $S^2 \times D^4 \times \{1\}$. Therefore the diffeomorphism induced by this $h$-cobordism is the composition of the diffeomorphisms induced by $(N_1,g_1)$ and $(N_2,g_2)$.

To show that $\Phi$ is surjective, consider a given a diffeomorphism $f \colon (S^2 \times D^4, \partial) \to(S^2 \times D^4, \partial)$. The normal $3$-smoothings $\ov{\nu}$ and $\ov{\nu} \circ f$ are homotopic. Choose a homotopy $g \colon S^2 \times D^4 \times I \to B$, extend $g$ to a map $S^2 \times S^4 \times I \to B$, where on the second copy of $S^2 \times D^4 \times I$ we use $\ov{\nu}$. Still call this map $g$. We also extend $f$ to a diffeomorphism of $S^2 \times S^4$, where on the second copy of $S^2 \times D^4$ we use the identity diffeomorphism. Still call this diffeomorphism $f$. Glue the two boundary components of $S^2 \times S^4 \times I$ via $f$, the map $g$ induces a map $\ov{g} \colon (S^2 \times S^4)_f \to B$ which represents an element in $\omst_7(P_3(S^2))$. Then 
$$ (S^2 \times D^4 \times [0,1] \times I) \cup_F (S^2 \times D^4 \times [1,2] \times I) \cup_{\mathrm{id}}(S^2 \times D^4 \times S^1)$$
is a $B$-coboundary between $(S^2 \times D^4 \times I , \ov{\nu}) + ((S^2 \times S^4)_f, \ov g)$ and $(S^2 \times D^4 \times [0,1]) \cup_{f} (S^2 \times D^4 \times [1,2])$ (Figure \ref{fig:4}), where the gluing map $F$ is 
$$f \times \mathrm{id} \colon S^2 \times D^4 \times \{1\} \times [0,1/3] \to S^2 \times D^4 \times \{1\} \times [0,1/3]$$
$$\mathrm{id} \colon S^2 \times D^4 \times \{1\} \times [2/3,1] \to S^2 \times D^4 \times \{1\} \times [2/3,1]$$
 Therefore the image of $[S^2 \times D^4 \times I , \ov{\nu}] + [(S^2 \times S^4)_f, \ov g]$ under $\Phi$ is just $f$. This shows the surjectivity of $\Phi$.
\end{proof}

\begin{figure}
\begin{center}
\mbox{\beginpicture \setcoordinatesystem units <0.018in,0.018in>
\setplotarea x from -52 to 40 , y from -50 to 50 %
\plot -30 40 30 40 30 15 -30 15 -30 40 /
\plot -30 -15 30 -15 30 -40 -30 -40 -30 -15 /
\plot -20 -10 -20 10 / \plot -20 10 -18 8 / \plot -20 10 -22 8 /
\plot 20 -10 20 10 / \plot 20 10 18 8 / \plot 20 10 22 8 /
\setdashes<1.5pt>
\plot -30 -15 -30 15 / \plot -10 -15 -10 15 / \plot 10 -15 10 15 / \plot 30 -15 30 15 /
\put {\mbox{$S^2 \times D^4 \times [1,2] \ \ \ \ \ \ $}} [cc] <0mm,0mm> at -52 25 %
\put {\mbox{$S^2 \times D^4 \times [0,1] \ \ \ \ \ \ $}} [cc] <0mm,0mm> at -52 -25 %
\put {\mbox{$ I$}} [cc] <0mm,0mm> at 0 50  %
\put {\mbox{$ I$}} [cc] <0mm,0mm> at 0 -50 %
\put {\mbox{$f$}} [cc] <0mm,0mm> at -25 0 %
\put {\mbox{$\mathrm{id}$}} [cc] <0mm,0mm> at 25 0
\endpicture}
\end{center}
\caption{}  \label{fig:4}
\end{figure}


An upper bound of the order of $\omst_7(P_3(S^2))$ is obtained by a computation with the Atiyah-Hirzebruch spectral sequence. We put the calculation in the appendix (Lemma \ref{lem:order1}).

\begin{lem}\label{lem:order}
$|\omst_7(P_3(S^2))| \le 24 $.
\end{lem}

\begin{lem}\label{lem:isom}
$\ov{\m}(S^{2} \times D^{4}, \partial)$ is isomorphic to $\z/12 \oplus \z/2$, detected by $s_{2}/2$ and $s_{3}$.
\end{lem}
\begin{proof}
From the previous lemma and Lemma \ref{lem:surj} we see that the order of $\ov{\m}(S^2 \times D^4, \partial) \le 24$. Indeed the equality holds. To see this,
denote the quotient of $\mathcal N$ by the action of $\Theta_7$ via connected-sum with homotopy spheres by  $\ov{\mathcal N}$. Then from the proof of Lemma \ref{lem:d6} the $KS$-invariant on $\mathcal N$ induces an injective map $\ov{\mathcal N} \to \z/8 \oplus \z/24 \oplus \z/2$, and the composition of surjective maps $\m(S^2 \times D^4, \partial ) \to \T(S^2 \times S^4) \to \mathcal N$ induces a surjective map $\ov{\m}(S^2 \times D^4, \partial ) \to \ov{\mathcal N}$. Let $a$, $b \in \ov{\m}(S^2 \times D^4,\partial)$ be a preimage of $N_{-1}$ and $N_{-2}$ (defined in end of \S\ref{7-mfds}), respectively, then by Formula \ref{eqn:nk} and Lemma \ref{lem:order}, $a$ has order $2$, $b$ has order $12$ and $a \ne 6b$. This finishes the proof of the lemma.

\end{proof}

To summarize, we have 

\begin{thm}\label{thm:S2D4}
There are isomorphisms
$$\m(S^{2} \times D^{4}, \partial) \stackrel{\cong}{\longrightarrow} \T(S^{2} \times S^{4}) \stackrel{\cong}{\longrightarrow} \z/28 \oplus \z/12 \oplus \z/2$$
where the second isomorphism is given by the invariants $(s_{1}/8, s_{2}/2, s_{3})$.
\end{thm}

\begin{rem}\label{rem:sato}
It's shown in \cite[Theorem II]{Sato} that $\T(S^{2} \times S^{4})$ is isomorphic to $FC_{4}^{3} \oplus \Theta_{7}$, where $FC_{4}^{3}$ is the group of isotopy classes of framed $4$-knots in $S^{7}$. By \cite{Haef} there is an exact sequence
$$\pi_{4}SO_{3} \to FC_{4}^{3} \to C_{4}^{3} \to \pi_{3}SO_{3}$$
with $C_{4}^{3} \cong \z/12$. There is a splitting $FC_{4}^{3} \to \pi_{4}(S^{2}) \cong \pi_{4}SO_{3}$ given by the linking invariant of the image of $S^{4} \times \{0\}$ and $S^{4} \times \{x_{0}\}$ in $S^{7}$, where $x_{0} \in S^{2}$ is a fixed point.  Therefore $FC_{4}^{3} \cong \z/2 \oplus \z/12$. In Theorem \ref{thm:S2D4} we get the same result using a different method, and moreover, we have found explicit invariants detecting these summands. We will use these invariants to give explicit generators of $\T(S^{2} \times S^{4})$ in the next subsection.
\end{rem}

\subsection{Generators of $\T(S^2 \times S^4)$}\label{subsec:gen}

First we have  the following observation on the mapping torus of a Dehn twist.

\begin{lem}\label{lem:dehn2}
Let $S \colon S^{3} \times D^{3} \subset M$ be an embedding, $\alpha \in \pi_{3}SO(4)$, $\mathbf{t}_{S, \alpha}$ be the Dehn twist in $S^{3} \times D^{3}$ parametrized by $\alpha$ (see \S \ref{subsec:genrel}). We have an embedding $S^3 \times D^3 \times I \subset M \times S^1$, where $I \subset S^1$ is an embedded interval. Identify $S^3 \times D^3 \times I$ with $S^3 \times D^4$, then the mapping torus of $\mathbf{t}_{S, \alpha}$ is
$$M_{\mathbf{t}_{S, \alpha}}= (M \times S^{1}-S^{3 } \times D^{4} )\cup_{\varphi} (S^{3} \times D^{4})$$
where $\varphi \colon S^{3} \times S^{3} \to S^{3} \times S^{3}$, $(x,y) \mapsto (\alpha(y)x ,y)$. Furthermore, let $E_{\alpha}$ be the $\mathbb R^{4}$-bundle over $S^{4}$ corresponding to $\alpha$, $D(E_{\alpha})$ be the disk bundle, with the standard decomposition $\partial D(E_{\alpha})=(S^{3} \times D^{4}) \cup_{\varphi} (S^{3} \times D^{4})$. Then
$M_{\mathbf{t}_{S, \alpha}} = \partial ((M \times D^{2} )\cup_{S^{3}\times D^{4}} D(E_{\alpha}))$.
\end{lem}
\begin{proof}
The mapping torus of a diffeomorphism $f \colon M \to M$ is $M_f = M \times [0,1]/(x \sim f(x))$. Notice that  a Dehn twist $\mathbf{t}_{S,\alpha}$ fixes points in the complement of $S(S^3 \times D^3)$. Therefore its mapping torus has the described form $M_{\mathbf{t}_{S, \alpha}}= (M \times S^{1}-S^{3 } \times D^{4} )\cup_{\varphi} (S^{3} \times D^{4})$.
\end{proof}

Now let $S \colon S^{3} \times D^{3} \hookrightarrow S^{2}\times S^{4}$ be an embedding such that $S^{3} \times \{0\}$  represents a generator of $\pi_{3}(S^{2} \times S^{4}) \cong \z$, $\mathbf{t}_{S, \alpha} \colon S^2 \times S^4 \to S^2 \times S^4$ be the Dehn twist with parameter $\alpha \in \pi_{3}(SO(4))$. Assume that the corresponding $\mathbb R^{4}$-bundle $E_{\alpha}$ over $S^{4}$ has Euler class $e(E_{\alpha})=\chi$ and spin Pontrjagin class $\frac{1}{2}p_{1}(E_{\alpha})=\chi -2b$ with $\chi$, $b \in \z$.  We are going to compute the $S$-invariant of $\mathbf{t}_{S, \alpha}$ in terms of $\chi$ and $b$.

Let $W=(S^{2} \times S^{4} \times D^{2 })\cup_{S^{3}\times D^{4}} D(E_{\alpha})$ be the coboundary in Lemma \ref{lem:dehn2}, with a normal map to $B= \cp^{\infty} \times K(\z,4) \times B\mathrm{Spin}$, which is the normal map on $S^2 \times S^4 \times D^2$, and on $D(E_{\alpha})$ is induced by the spin structure on $D(E_{\alpha})$ (the map to $\cp^{\infty} \times K(\z,4)$ is trivial). By the Mayer-Vietoris sequence $H_{4}(W)$ is isomorphic to $\z^{3}$, a basis of which is given as follows: 
\begin{enumerate}
\item the base $S^{4}$ of $E_{\alpha}$, denoted by $[S^4_{b}]$; 
\item a class $[Y]$, where $Y=D^{4} \cup_{S^3} C$, where $D^4$ is a fiber of $D(E_{\alpha})$, and $C \subset S^2 \times S^4 \times D^2$ is a submanifold whose boundary is $S^3 \times \{0\}$;  
\item $pt \times S^{4} \times pt \subset S^{2} \times S^{4} \times D^{2}$, denoted by $[S_{0}^{4}]$. 
\end{enumerate}
The intersection form with respect to this basis is represented by the matrix
$$\left ( \begin{array}{ccc}
\chi & 1 & 0\\
1 & 0 & 0\\
0 & 0 & 0 \end{array} \right ).$$
(Note that $C$ can be taken as a submanifold in a neighborhood of $S^2 \times S^4 \times \{t\}$ in $S^2 \times S^4 \times D^2$, where $t \in S^1$. Therefore the self-intersection number of $[Y]$ is zero. The other intersection numbers are obvious.)

From the geometric construction we have 
$$\langle \frac{1}{2}p_{1}(W), [S_{b}^{4}] \rangle =\chi -2b, \ \ \langle \frac{1}{2}p_{1}(W), [S_{0}^{4}] \rangle =0.$$
Denote $\langle \frac{1}{2}p_{1}(W), [Y] \rangle = A$, an integer which doesn't depend on $\alpha \in \pi_3SO(4)$. Then
$\frac{1}{2}p_{1}(W)=(\chi-2b) [S_{b}^{4}]^{*} + A [Y]^*$, where $[S^4_b]^*$, $[Y]^*$ and $[S^4_0]^*$ form the dual basis.

Let $x \in H^{2}(W)$ be a generator, a representing cycle $X$ of $Dx \in H_{6}(W, \partial W)$ is constructed as follows: fix $p \in S^{2}$, consider $p \times S^{4} \times D^{2} \subset S^{2} \times S^{4} \times D^{2}$, which is the Poincar\' e dual of a generator of $H^{2}(S^{2} \times S^{4} \times D^{2})$. We need to extend this to $E_{\alpha}$.  First notice that the intersection of $S^{3} \times D^{4} \subset S^{2} \times S^{4} \times S^{1}$ with
$p \times S^{4} \times D^{2}$ is $S^{1}_{p} \times D^{4} \subset S^{3} \times D^{4}$, where $S^{1}_{p} \subset S^{3}$ is the Hopf fiber over $p \in S^{2}$.  We cone this $S^{1}_{p} \times D^{4}$ in $D^{4} \times D^{4} \subset E_{\alpha}$ and get $D^{2}_{p} \times D^{4} \subset D^{4} \times D^{4}$. The intersection of $D^{2}_{p} \times D^{4}$ with $D^{4} \times S^{3}$ is $D^{2}_{p} \times S^{3}$. The clutching function of the bundle $E_{\alpha}$ is $\varphi_{\alpha} \colon D^{4} \times S^{3} \to D^{4} \times S^{3}$, $(x,y) \mapsto (\alpha(y) \cdot x, y)$. Thus the image $\varphi_{\alpha}(D^{2}_{p} \times S^{3})=\{ (\alpha(y) \cdot x, y)\  | \  x \in D^{2}_{p}, \   y \in S^{3}\}$.
We take the cone of this over the $S^{3}$ factor and get $D^{2} \times D^{4} \subset D^{4} \times D^{4}$. Now let
$$X_{p}=(p \times S^{4} \times D^{2}) \cup_{S^{1}_{p} \times D^{4}} (D^{2}_{p}\times D^{4}) \cup_{\varphi_{\alpha}(D^{2}_{p}\times S^{3})} (D^{2} \times D^{4})$$
Then by the construction $X_{p}$ represents the Poincar\' e dual of $x \in H^{2}(W)$. Now we compute the self-intersection of $X_{p}$. We take another point $p' \in S^{2}$ and obtain $X_{p'}$ by the same procedure. The circles $S^{1}_{p}$ and $S^{1}_{p'}$ form a Hopf link in $S^{3}$ and $D^{2}_{p} \cap D^{2}_{p'} =0 \in D^{4}$. Therefore from the construction $X_{p} \cap X_{p'}=S^{4}_{b}$, the base sphere of the bundle $E_{\alpha}$. 

Let $j^* \colon H^4(W, \partial W) \to H^4(W)$ be the natural homomorphism, $\bar \alpha$, $\bar \beta \in H^{4}(W, \partial W)$ such that $j^*(\bar \alpha) =x^2$, $j^*(\bar \beta)=\frac{1}{2}p_1(\nu W)$. Let $D \colon H^4(W, \partial W) \to H_4(W)$ be the Poincar\' e dual map,  then from the intersection matrix we have 
$$j^*D^{-1}([S^4_b]) = \chi [S^4_b]^* + [Y]^*, \ \ j^*D^{-1}([Y]) = [S^4_b]^*, \ \ j^*D^{-1}([S^4_0]) = 0.$$
Therefore we may take $\bar \beta = -A D^{-1}([S^4_b]) + (2b-(1-A)\chi ) D^{-1}([Y])$. With these considerations we have
$$D(\bar \alpha)=[S^{4}_{b}], \ \ \ D(\bar \beta)=-A [S^4_b] + (2b-(1-A)\chi ) [Y].$$
Now we have 
$$\langle \bar \beta^2, [W, \partial W] \rangle = \langle D(\bar \beta), \beta \rangle = -A((2-A) \chi -4b).$$
Note that this integer is divisible by $8$ for any $\chi, b \in \z$, therefore we must have $A \equiv 2 \pmod 4$. If we reframe the embedding $S \colon S^3 \times D^3 \to S^2 \times S^4$ by some $v \in \pi_3SO(3)$ and denote the new embedding by $S_v$, then we have a new Dehn twist $\mathbf{t}_{S_v, \alpha}$ and a corresponding coboundary $W_v$ of the mapping torus. From the construction of $Y$ we still have $Y \subset W_v$, with 
\begin{equation}\label{eqn:y}
\langle p_1(W_v), [Y] \rangle = \langle p_1(W), [Y] \rangle +\langle p_1(v), [S^4] \rangle,
\end{equation} 
where $p_1(v)$ is the first Pontrjagin class of the vector bundle corresponding to $v$.
Since the map $\pi_3SO(3) \to 4\z$, $v \mapsto \langle p_1(v), [S^4] \rangle$ is an isomorphism, we may choose a $v$ such that $\langle p_1(W_v), [Y] \rangle =0$. We denote this embedding by $Z \colon S^3 \times D^3 \to S^2 \times S^4$, and have
\begin{eqnarray*}
\frac{1}{8}s_{1}& = & \frac{1}{8}\langle \bar \beta^{2}, [W, \partial W] \rangle = 0 \\
\frac{1}{2}s_{2} & = & \frac{1}{2}\langle (\bar \alpha ^{2} + \bar \alpha \bar \beta),[W, \partial W] \rangle =\frac{1}{2}(\chi + \chi  - 2b) = \chi - b \in \z/12 \\
s_{3} & = & \langle \bar \alpha \bar \beta, [W, \partial W] \rangle  = \chi -2b =\chi \in \z/2
\end{eqnarray*}
Thus we have the following  explicit geometric description of $\T(S^{2} \times S^{4})$.
\begin{thm}\label{thm:s2s4}
In the mapping class group $\T(S^{2} \times S^{4}) \cong \z/28 \times \z/12 \times \z/2$, the $\z/28$-factor is generated by the disk supported diffeomorphism $\mathbf{f}_{c}$, the $\z/12$-factor is generated by the Dehn twists $\mathbf{t}_{Z, \rho}$ (with $\chi=0$, $b=1$) and the $\z/2$-factor is generated by the Dehn twist $\mathbf{t}_{Z, \epsilon}$ (with $\chi=1$, $b=1$).
\end{thm}

Note that  Theorem \ref{thm:gen1} for spin manifolds is a corollary of this theorem and the fact that there is a surjective homomorphism $\m(S^2 \times D^4, \partial)\to \m(M)$ (see Lemma \ref{lem:commute}).

\begin{rem}
There is another natural Dehn twist, namely the Dehn twist along the embedded $S^{2} \times D^{4} \subset S^{2} \times S^{4}$ parametrized by the non-trivial element of $\pi_{4}(SO(3)) \cong \z/2$. The extension of this Dehn twist to $S^2 \times S^4$ is precisely the clutching function for the non-trivial $S^2$-bundle over $S^5$. It can be shown (c.~f.~\cite[\S3.2]{CS11}) that the total space of this bundle is diffeomorphic to $SU(3)/S^1$, whose Kreck-Stolz invariant is $(0,0,1) \in \z/28 \times \z/12 \times \z/2$. 
\end{rem}

\section{Determination of $\K(M,D)$}\label{sec:K}
In this section we show that all elements in $\K(M,D)$ are represented by diffeomorphims extended from diffeomorphisms  of $(S^{2} \times D^{4}, \partial)$. Using this fact and the generalized Kreck-Stolz invariant, we are able to compute the group $\K(M, D)$, as given in the tables in Theorem \ref{thm:kernel}. In next section we will show that the homomorphism $\K(M,D) \to \K(M)$ is an isomorphism. 

Recall that for a $6$-manifold $M$ which looks like a complete intersection we have a decomposition (\ref{decomp})
$$M= N \# \, r (S^3 \times S^3)$$
where $r=b_3(M)/2$ and $H_*(N) \cong H_*(\mathbb C \mathrm P^3)$. For such $N$ there is a decomposition $N=(S^{2} \times D^{4}) \cup_{\varphi} (S^{2} \times D^{4})$, with $\varphi \colon S^2 \times S^3 \to S^2 \times S^3$ a diffeomorphism (\cite[Theorem 2]{Wall}). 
Let $M_0=(S^2 \times D^4) \# \, r(S^{3} \times S^{3})$ be the connected-sum of the second copy of $S^2 \times D^4$ with $\# \, r(S^3 \times S^3)$. Given a diffeomorphism of $M_0$ whose restriction on $\partial M_0$ is the identity, we extend it to a diffeomorphism $f$ of $M$ by the identity on the first copy of $S^2 \times D^4$. If $[f]$ is in $\K(M,D)$, then there is a $y' \in H^{4}(M_{f})$ such that $x^{2}-dy'=0$ and $p_{1}(M_{f})-2ky'=0$. Let $i^{*} \colon H^{4}(M_{f}) \to  H^{4}((M_{0})_{f})$ be the homomorphism induced by the inclusion of the mapping torus, then $d i^{*}(y')=i^*(x^2)=0$. This implise $i^*(y')=0$ and $p_{1}((M_{0})_{f})=i^{*}(p_{1}(M_{f}))=0$.

\begin{prop}\label{prop:uk}
Each element in $\K(M,D)$ is represented by a diffeomorphism which is extended from $f \colon M_0 \to M_0$ such that $f_*=\mathrm{id}$ on $\pi_3(M_0)$ and $p_1((M_0)_f)=0$.
\end{prop}

\begin{proof}
Let $f \colon (M,D) \to (M,D)$ be a diffeomorphism such that $[f] \in \K(M,D)$. Since $f$ induces the identity on $H_{2}(M)$, and $\pi_{2}SO(4)=0$, after an isotopy we may assume that $f$ is the identity on the first copy of $S^{2} \times D^{4} \subset M$.  Then we get a diffeomorphism $f$ of $M_0$ which is the identity on the boundary.
By similar argument in \S \ref{sec:diff},
 $$\pi_{3}(M_{0}) \cong \pi_3(S^2 \times D^4) \oplus \pi_3(\# \, r(S^3 \times S^3) )\cong \z \oplus \z^{2r},$$
 where the second isomorphism is given by the standard basis $e_0, e_1, f_1, \cdots e_r , f_r$. The inclusion map $M_0 \to M$ induces a homomorphism
 $$\pi_{3}(M_{0}) \to \pi_{3}(M) \cong \pi_3(N) \oplus \pi_3(\# \, r(S^3 \times S^3))\cong \z/d \oplus \z^{2r},$$ 
which is the mod $d$ reduction on the first summand and the identity on the second. Since $[f]$ is in $\K(M,D)$,  by Lemma \ref{lem:pi3} the action $f_*$  on $\pi_3(M)$ is the identity,  therefore under the standard basis $e_0, e_1, f_1, \cdots e_r , f_r$, the action $f_*$ on  $\pi_{3}(M_0)$ is represented by the matrix
$$ \left ( \begin{array}{cc}
I & 0 \\
B & I \end{array} \right )$$
where $B \colon \pi_3(\# \, r(S^3 \times S^3)) \cong \z^{2r} \to d\z$ is a homomorphism. 

We construct diffeomorphisms $\varphi \colon M_0 \to M_0$ with  $\varphi |_{\partial M_0}=\mathrm{id}$ satisfying the following properties:

\begin{enumerate}
\item The action of $\varphi$ on $\pi_3(M_0)$ has the form  
$$ \left ( \begin{array}{cc}
I & 0 \\
B_1 & I \end{array} \right )$$
where $B_1$ maps a generator $e_i$ or $f_i$ of $\pi_3(\# \, r(S^3 \times S^3))$ to $d e_0$,  and maps other generators to $0$. 

\item We extend $\varphi$ to a diffeomorphism of $M$ by the identity on the complement, then $\varphi$ is isotopic to the $\mathrm{id}_M$. 
\end{enumerate} 
Then by taking composition of $f$ with these diffeomorphisms we get a diffeomorphism which is isotopic to $f$ and satisfies the requirement in the proposition. 

For simplicity, we give the construction of the diffeomorphism $\varphi$ with $B_1(f_1)=d e_0$. The other diffeomorphisms are similarly constructed. Let $S_1, S_2 \colon S^3 \times D^3 \to M_0$ be embeddings such that the core of $S_1$ represents $e_1$, and the core of $S_2$ represents $e_0+e_1$. Choose $\alpha_1, \alpha_2 \in \pi_3SO(4)$ such that $\chi(\alpha_1)=d$, $\chi(\alpha_2)=-d$, $\frac{1}{2}p_1(\alpha_1)=b_1$, $\frac{1}{2}p_1(\alpha_2)=b_2$, where $b_1$ and $b_2$ to be determined later. Let $\varphi =\mathbf{t}_{S_1,-\alpha} \circ \mathbf{t}_{S_2,\alpha}$.  
By Lemma \ref{lem:PL}, $\varphi_{*}(f_1)=d \cdot e_0+f_1$, and $\varphi_{*}$ fixes the other generators of $\pi_{3}(M_{0})$. By Lemma \ref{lem:pi3} $\varphi_*=\mathrm{id}$ on $\pi_3(M)$. Therefore the first component of the invariant $v_{x,p}(\varphi)$ is zero. Notice that the diffeomorphism $\varphi$ fixes embedded spheres $S^3 \subset M$ representing $e_1, \cdots , e_r$ and $f_2, \cdots , f_r$, we have $\langle \frac{1}{2}p_1(M_{\varphi})-2k\bar y, [S^1 \times S^3] \rangle =0$. By Lemma \ref{lem:chi} we may choose an appropriate $b_1$ such that  $\langle \frac{1}{2}p_1(M_{\varphi})-2k\bar y, [S^1 \times S^3]\rangle =0$, where $S^3 \subset M$ represents $f_1$. Therefore $[\varphi] \in \K(M,D)$.

To compute the generalized Kreck-Stolz invariant of $\varphi$, we use the construction given in  Lemma \ref{lem:dehn2}.  We have a coboundary $W= (M \times D^{2}) \cup D(E_{\alpha_1}) \cup D(E_{\alpha_2})$ of the mapping torus of $\varphi$, with a normal map $\bar \nu \colon W \to B$, whose restriction on $D(E_{\alpha_1})$ and $D(E_{\alpha_2})$ maps to $B\mathrm{Spin}$. The signature of $E_{\alpha_1}$ and  $E_{\alpha_2}$ is $1$ and $-1$ respectively, therefore the signature of $W$ is zero.  The homology group $H_4(W)$ is isomorphic to $\z^4$, a basis is: the base spheres of $E_{\alpha_1}$ and $E_{\alpha_2}$, denoted by $S^4_{\alpha_1}$ and $S^4_{\alpha_2}$, respectively, $Y=D^4 \cup C \cup D^4$, where $C$ is a submanifold of $M \times D^2$ such that $\partial C$ is the union of cores of $S_1$ and $S_2$, and a generator $u$ of $H_4(M)$. The intersection matrix with respect to this basis is $$\left ( \begin{array}{cccc}
d & 0 & 1 & 0\\
0 & -d  & -1 & 0\\
1 & -1 & 0 & 0 \\
0 & 0 & 0 & 0
\end{array} \right ).$$

We have $\langle \frac{1}{2}p_1(W), [S^4_{\alpha_1}] \rangle = d+2b_1$,  $\langle \frac{1}{2}p_1(W), [S^4_{\alpha_2}] \rangle = 2b_2-d$, $\langle \frac{1}{2}p_1(W), u \rangle = \langle \frac{1}{2}p_1(M), u \rangle = 2k$. By choosing an appropriate framing of $S_1$, we may assume $\langle \frac{1}{2}p_1(W), [Y] \rangle = A$ with $A=0$ or $1$.  Let $y \in H^4(K(\z,4))$ be the chosen generator, then
$$\langle \bar \nu^*y, [S^4_{\alpha_1}] \rangle = \langle \bar \nu^*y, [S^4_{\alpha_2}] \rangle =0, \ \ \langle \bar \nu^*y, u\rangle=1,$$ 
and by choosing $C$ appropriately we have $\langle \bar \nu^*y, [Y] \rangle =0$.

We take $\bar \alpha = D^{-1}([S_{\alpha_1}]+[S_{\alpha_2}])$ as a preimage of $\alpha=x^2 - d\bar \nu^*y$. If $A=0$, we may take $\bar \beta= 2(d-2b) D^{-1}([Y])$ as a preimage of $\beta=\frac{1}{2}p_1(\nu W)+2k\bar \nu^*y$. Then 
$$\langle \bar \beta^2,[W, \partial W] \rangle = \langle \bar \alpha^2,[W, \partial W] \rangle=\langle \bar \alpha \bar \beta,[W, \partial W] \rangle=0$$
and the generalized Kreck-Stolz invariant is of $\varphi$ is $0$.  If $A=1$,  the existence of $\bar \beta$ implies that $d=2(b_1+b_2)$. A preimage of $\beta$ is  $\bar \beta = d D^{-1} [S^2_{\alpha_1}]+(d-2b_2)D^{-1}[Y]$. Then 
$$\langle \bar \beta^2, [W,\partial W] \rangle = 2d(d-b_2)+(d-2b_2), \ \ \langle \bar \alpha^2, [W,\partial W] \rangle = 0, \ \ \langle \bar \alpha \bar \beta, [W,\partial W] \rangle =d.$$
Choosing $b_2$ such that $\langle \bar \beta^2, [W,\partial W] \rangle$ is divisible by $8$, then $S(\varphi) \in 8\z/(8\cdot 28) \times 2\z/24 \times \z/2$. There exists a diffeomorphism $g$ of $S^2 \times D^4$ realizing the invariant $S(\varphi)$. The composition of $\varphi$ and $g^{-1}$ is the desired diffeomorphism.
\end{proof}

The normal $3$-type of $M_0 = (S^2 \times D^4) \# \, r(S^3 \times S^3)$ is 
$$p \colon P_3(S^2) \times K(\z^{2r},3) \times B\mathrm{String} \to BO$$
where $p$ is the canonical projection on $B\mathrm{String}$. We fix the normal smoothing $\bar \nu \colon M_0 \to B$ given by the basis $e_0, e_1, f_1, \cdots, e_r, f_r$ and a framing of $M_0$. Then analogous to \S \ref{sec:rel}, from a bordism class in $\omst_{7}(P_{3}(S^{2})\times K(\z^{2r}, 3) )$ we may construct a diffeomorphism $f \colon M_0 \to M_0$, unique up to disk-supported diffeomorphisms. Since the diffeomorphism $f$ constructed in this way preserves the normal smoothing $\bar \nu$, the induced action of $f$ on $\pi_3(M_0)$ is trivial,  and $p_1((M_0)_f)=0$. Let $\m_0(M_0, \partial)$ be the mapping class group consisting of diffeomorphisms $f$, which is the identity near $\partial M_0$, whose action on $\pi_3(M_0)$ is trivial, and such that $p_1((M_0)_f)=0$.
Let $\overline \m_0(M_{0}, \partial)$ be the cokernel of $\Theta_{7} \to \m_0(M_{0}, \partial)$.
Similar to Lemma \ref{lem:surj} we have
\begin{lem}\label{lem:gamma-surj}
There is a surjective homomorphism
$$\Phi_{1} \colon \omst_{7}(P_{3}(S^{2})\times K(\z^{2r}, 3) )\to \overline \m_0(M_{0}, \partial)$$
\end{lem}

We extend a diffeomorphism $f$  of $S^2 \times D^4$ with $f|_{\partial }=\mathrm{id}$ to a diffeomorphism of $M_0$ by identity. This gives a homomorphism 
$\ov \m(S^{2} \times D^{4}, \partial)  \to \ov \m_{0}(M_{0}, \partial)$. By the constructions of $\Phi$ and $\Phi_1$ we have

\begin{prop}\label{prop:commU}
There is a commutative diagram
$$\xymatrix{
\omst_{7}(P_{3}(S^{2})) \ar[r]^{\Phi} \ar[d]^{\iota} & \ov \m(S^{2} \times D^{4}, \partial) \ar[d] &    \\
\omst_{7}(P_{3}(S^{2}) \times K(\z^{2r}, 3)) \ar[r]^{\ \ \ \ \ \ \ \ \Phi_{1}} & \ov \m_{0}(M_{0}, \partial)  }$$
in which the horizontal maps are surjective. Where $\iota$ is induced by the inclusion map of the first factor.
\end{prop}

In Lemma \ref{lem:string} we show that  the natural inclusions of both factors induce  an isomorphism
$$\omst_{7}(P_{3}(S^{2})) \oplus \omst_{7}(K(\z^{2r},3)) \to \omst_{7}(P_{3}(S^{2}) \times K(\z^{2r}, 3)) $$
But a diffeomorphism constructed from an element in the direct summand $\omst_{7}(K(\z^{2r},3))$ is the extensions of a diffeomorphism $f$ of $\#_r(S^{3} \times S^{3})$, where the action of $f$ on $\pi_{3}$ is trivial, and the first Pontrjagin class of the mapping torus of $f$ is zero. By  Theorem \ref{thm:kr} $f$ is isotopic to $\mathrm{id}$. Therefore the homomorphism 
$ \ov \m(S^{2} \times D^{4}, \partial) \to \ov \m_{0}(M_{0}, \partial)$ is surjective.
By Proposition \ref{prop:uk} we have 

\begin{lem}\label{lem:mk}
The homomorphism $\m(S^{2} \times D^{4}, \partial) \to \K(M,D)$ is surjective.
\end{lem}

\begin{lem}\label{lem:commute}
There is a  commutative diagram
$$\xymatrix{
\m(S^2 \times D^4, \partial)  \ar[d] \ar[r]^{\ \ \ \ S} & \z^3/L \ar[d] \\
\K(M,D) \ar[r]^{S} & \z^3/\bar L(M) }$$
\end{lem}
\begin{proof}
Let $f$ be a diffeomorphism of $S^{2} \times D^{4}$ with $f|_{\partial} = \mathrm{id}$, extend it to a diffeomorphism $f$ of $S^2 \times S^4$ by the identity, let 
$$N = (S^2 \times D^5) \cup_f (S^2 \times S^4 \times I )\cup_{\mathrm{id}}(S^2 \times D^5)$$ 
be the corresponding $7$-manifold and $W \to \cp^{\infty} \times B\mathrm{Spin}$ be the null-bordism for computing the $KS$-invariant of $N$. Recall that $M=(S^{2} \times D^{4}) \cup_{\varphi}M_0$,
where $M_0=( S^{2} \times D^{4})\# \#_r(S^3 \times S^3)$ and $\varphi \colon S^{2} \times S^{3} \to S^{2} \times S^{3}$ is a diffeomorphism.
 Extend $f$ to a diffeomorphism of $M$ by the identity on $M_0$. Define $V=W \cup_{\varphi \times \mathrm{id}} (M_0 \times I^2)$, where $\varphi \times \mathrm{id}$ is a diffeomorphism from  $S^2 \times S^3 \times I^2 \subset \partial (M_0 \times I^2)$ to  
$S^2 \times S^3 \times I^2 \subset S^2 \times S^4 \times I \subset N$, where $S^3 \times I$ is a neighborhood of the equator $S^3$ of $S^4$. Then $\partial V=M_f$. 
The normal smoothings of $W$ and $M_0 \times I^2$ fit together giving a normal smoothing of $V$. In  Wall's formula of the non-additivity of the signature the correcting term vanishes, therefore $\mathrm{sign}(V)=0$. Therefore we may use $V$ as the null $B$-bordism for computing the generalized Kreck-Stolz invariant of $f$. Similar to the proof of Lemma \ref{lem:IN}, the $M_0 \times I^2$ part doesn't contribute to the characteristic numbers. This proves the commutativity of the diagram.
\end{proof}

\begin{proof}[Proof of Theorem \ref {thm:kernel}
 (spin case)]
In the next section we show that the forgetful homomorphism $\K(M,D) \to \K(M)$ is an isomorphism (Theorem \ref{prop:forgetful}). Therefore it suffices to compute the group $\K(M,D)$.
By Theorem \ref{thm:S2D4} and Lemma \ref{lem:commute}, $\K(M,D)$ is isomorphic to the image of the homomorphism
$$\z/28 \oplus \z/12 \oplus \z/2 \to \z^{3}/ \bar L(M)$$
which is isomorphic to $\z^3 / L(M)$, where $ L(M)$ is the lattice generated by the column vectors of the following table
\begin{equation}\label{eq:table}
\begin{array}{c|cccccc}
& e_1 & e_2 & e_3 & e_4 & e_5 & e_6 \\
\hline
s_{1}/8 & -28 & 0 & 0 & (k^2-k)/2 & k^2 & 0 \\
s_{2}/2& 0 & 12 & 0 & (d^2-d(2k-1))/2 & d^{2}-2dk & k-d \\
s_{3} & 0 & 0 & 2 & d(2k-1) & 4dk & -2k
\end{array}
\end{equation}

The finite abelian group $\z^3/ L(M)$ can have only $2$-, $3$- and $7$-torsion, of exponents at most $2$, $1$, and $1$ respectively. Therefore we may determine the structure of $\z^3/L(M)$ by tensorizing it with $\z/4$, $\z/3$ and $\z/7$ respectively. After solving this elementary arithmetic problem we get the tables in Theorem \ref{thm:kernel}.

\begin{enumerate}
\item modulo $3$: only the image of $s_{2}/2$ is involved. Since
$$\gcd((d^{2}-d(2k-1))/2, d^{2}-2dk, k-d)$$ is divisible by $3$ if and only if $d \equiv 0 \pmod 3$, we have
$$
\z^{3} /L(M) \otimes \z/3 = \left \{ \begin{array}{cl}
\z/3 & \mathrm{if \ }d \equiv 0 \pmod 3 \\
0 & \mathrm{if \ }d \not\equiv 0 \pmod 3
\end{array} \right. $$

\item modulo $7$: only the image of $s_{1}/8$ is involved. Since $\gcd((k^{2}-k)/2, k^{2})$ is divisible by $7$ if and only if $d-l \equiv 0 \pmod 7$, we have
$$
\z^{3} /L(M) \otimes \z/7 = \left \{ \begin{array}{cl}
\z/7 & \mathrm{if \ }d -l\equiv 0 \pmod 7 \\
0 & \mathrm{if \ }d-l \not\equiv 0 \pmod 7
\end{array} \right. $$

\item modulo $4$: the entries in table (\ref{eq:table}) are $8$-periodic with respect to $d$. Therefore we only need to compute the lattice for $d=0, \cdots, 7$. The results are in the table in Theorem \ref{thm:kernel}.
\end{enumerate}

The group $\K(M)$ is generated by the Dehn twists $\mathbf{f}_c$, $\mathbf{u}=\mathbf{t}_{Z,\rho}$ and $\mathbf{z}=\mathbf{t}_{Z,\epsilon}$, with 
\begin{eqnarray*}
(\frac{1}{8}s_1(\mathbf{f}_c), \frac{1}{2}s_2(\mathbf{f}_c), s_3(\mathbf{f}_c)) & = & (1,0,0) \\
 (\frac{1}{8}s_1(\mathbf{u}), \frac{1}{2}s_2(\mathbf{u}),s_3(\mathbf{u})) & = & (0,1,0) \\
 (\frac{1}{8}s_1(\mathbf{z}), \frac{1}{2}s_2(\mathbf{z}),s_3(\mathbf{z}))& = & (0,0,1)
\end{eqnarray*}
From the simplification of the lattice one deduces the orders of $\mathbf{f}_c$, $\mathbf{u}$ and $\mathbf{z}$ are $2^p \cdot 7^{1-c}$, $2^q \cdot 3^{1-b}$ and $2^r$, respectively,  where  
$$p = \left \{ \begin{array}{cl}
0 & \mathrm{otherwise} \\
1 & \textrm{$d \equiv 4 \pmod 8$ and $l$ even, or  $d \equiv 0 \pmod 8$ and $l \equiv 2 \pmod 4$} \\
2 &  \textrm{$d \equiv 0 \pmod 8$ and $l \equiv 0 \pmod 4$} \end{array} \right. $$
$$q = \left \{ \begin{array}{cl}
1 & \mathrm{otherwise} \\
2 &  \textrm{$l$ even, or $d \equiv 2 \pmod 4$ and $l$ odd} \end{array} \right.  \ \ \ \ r = \left \{ \begin{array}{cl}
0 & \textrm{$d$ odd} \\
1 & \textrm{$d$ even} \\
\end{array} \right. $$
$$b = \left \{ \begin{array}{cl}
1 & d \equiv 0 \pmod 3 \\
0 & \mathrm{otherwise} \end{array} \right. \ \ \ c = \left \{ \begin{array}{cl}
1 & d- l \equiv 0 \pmod 7 \\
0 & \mathrm{otherwise} \end{array} \right.$$

\end{proof}

\begin{proof}[Proof of Theorem \ref{thm:pres}]
It is shown in \cite[Theorem 2]{Kry02} that $[\mathbf{a}_i, \mathbf{b}_i]=\mathbf{f}_c$. The orders of $\mathbf{f}_c$, $\mathbf{u}$ and $\mathbf{z}$ are just given above. 
\end{proof}

\section{$\K(M,D)$ is isomorphic to $\K(M)$}\label{sec:bd}
In this section we prove that in the spin case the forgetful homomorphism $\K(M,D) \to \K(M)$ is an isomorphism. Recall that the relative mapping class group $\K(M,D)$ is related to the mapping class group $\K(M)$ by the exact sequence (\ref{seq:br1})
$$\pi_1(GL_6(\mathbb R)) \to \K(M,D) \to \K(M) \to 1.$$
{The image of the generator of $\pi_1(GL_6(\mathbb R))$ in $\K(M,D)$ is called the boundary mapping class, represented by the so-called boundary diffeomorphism, described under the exact sequence (\ref{seq:br1}). The forgetful map $\K(M,D) \to \K(M)$ is an isomorphism if and only if the boundary mapping class is the trivial element in $\K(M,D)$.

The mapping torus of the boundary diffeomorphism of $M$ is diffeomorphic to $S^1 \times M$, but the preferred spin structure is the twisted one when restricted to $S^1 \times D$. Although we know that the normal map $S^{1} \times M \to \cp^{\infty} \times K(\z,4) \times B\mathrm{Spin}$ is a boundary in the corresponding spin bordism group, we don't know of an explicit construction for such a spin coboundary. Instead we will make use of a $\mathrm{spin}^c$-coboundary $V$ constructed as follows: Let $E$ be the non-trivial linear $D^6$-bundle over $S^2$, i.~e.~
$$E= (D^6 \times D^2)_{1} \cup_g (D^6 \times D^2)_{2}$$
where $g \colon D^6 \times S^1 \to D^6 \times S^1$ is given by $(x,t) \mapsto ( \alpha (t) \cdot x, t)$, with $\alpha \colon S^1 \to SO(6)$ representing the non-trivial element of the fundamental group. Let $S^{5} \times D^{2} \subset (D^{6} \times D^{2})_{1} $ be the standard inclusion. On the other hand, a tubular neighborhood of $S^{5}=\partial D \subset \{s_0\} \times M \subset S^{1} \times M$  is canonically identified with $S^{5} \times D^{2}$, define
\begin{equation}\label{eqn:V}
V=(M \times D^{2}) \cup_{S^5 \times D^2} E,
\end{equation}
then $\partial V= (M \times S^{1} - S^{5} \times D^{2}) \cup_{g} S^{5} \times D^{2}$ is diffeomorphic to $S^{1} \times M$. Let $M \times D^2 \to \cp^{\infty}$ be the classifying map of a generator of $H^2(M \times D^2)$, we extend it to $V \to \cp^{\infty}$ where the map on $E$ is a constant map.  Let $L$ be the complex line bundle over $V$ with $c_1(L)=(0,z) \in H^2(V) = H^2(M) \oplus H^2(E)$. Then the spin-structure on $\nu V \oplus L$ restricts to the non-trivial  spin-structure on $S^1\times D$. 

\begin{thm}\label{thm:bd1}
Let $M$ be a simply-connected spin $6$-manifold with $H_*(M) \cong H_*(\mathbb C \mathrm P^{3})$, $D \subset M$ be an embedded $6$-disk, $N=M - \mathring D$, $f \colon N \to N$ be the restriction of the boundary diffeomorphism defined using $D$.  Then $f$ is isotopic to the identity rel $\partial N$.
\end{thm}
\begin{proof}
The idea of the proof is the following (compare the proof of \cite[Proposition 3]{Kr78}):
Do surgery on $S ^{1} \times D \subset S^{1} \times M$ with the twisted framing we get $X=(N \times I) \cup_{\partial}(N \times I)$, where the gluing diffeomorphism is $f \colon N \times \{1\} \to N \times \{1\}$ and $\mathrm{id} \colon (\partial N \times I) \cup (N \times \{0\}) \to (\partial N \times I) \cup (N \times \{0\})$.  If $X$ bounds a compact manifold $U$ such that $(U, N \times I)$ is a relative $h$-cobordism, then $f$ is isotopic to $\mathrm{id}_{N}$ rel $\partial N$. Now $H_{3}(X) \cong \mathbb Z$ and the homology classes are spherical. Choose an embedding $S^{3} \times D^{4} \subset X$ representing a generator of $H_{3}(X)$ and do surgery, we get a simply-connected spin $7$-manifold $Y$, depending on the framing of the $3$-surgery, with $H_*(Y) \cong H_*(S^{2} \times S^{5})$. If $Y$ is diffeomorphic to $S^{2} \times S^{5}$, then we may take $U$ to be the union of the trace of the second surgery and $S^{2} \times D^{6}$ along $Y$. 

In the following we show that for some framing of the $3$-surgery, the Kreck-Stolz invariant $KS(Y)=0$, which implies that  $Y$ is diffeomorphic to $S^2 \times S^5$. This is equivalent to showing that the invariants $S_i(Y)=0$ ($i=1,2,3$) (see Remark \ref{rem:Ss}). These invariants can be computed from a $\mathrm{spin}^{c}$-coboundary of the map $Y \to \cp^{\infty}$ representing a generator of $H^{2}(Y)$ (c.~f.~\cite[(2.7)]{KrSt2}). We construct such a coboundary as follows: Let $S^1 \times M \to \cp^{\infty}$ be a map representing a generator of $H^2(S^1 \times M)$,  this map extends to the trace $W'$ of the above surgeries where on the attached  $2$-handle and $4$-handle it is a constant map. There is a unique spin structure on $W'$, which restricts to the non-trivial spin structure on $S^{1} \times M$.  Take the $\mathrm{spin}^c$-coboundary $V$ described above and let $W=W' \cup V$. Then $W \to \cp^{\infty}$ is a $\mathrm{spin}^{c}$-coboundary of $Y \to \cp^{\infty}$.

Now we have $H_2(W) \cong \z^3$, with a basis $\{ [S^2_M], [S^2_b], [S^2_0]\}$, where $S^2_M \subset M$ represents a generator of $H_2(M)$, $S^2_b$ is the zero section of $E$, and $S^2_0$ is the union of the core of the $2$-handle and $D^2 \times \{*\} \subset D^2 \times M$. From the construction we have 
$$\langle w_2(W), [S^2_M] \rangle =0,  \ \ \langle w_2(W), [S^2_b] \rangle =1,  \ \ \langle w_2(W), [S^2_0] \rangle =1.$$
Therefore $z  : =[S^2_b]^*+[S^2_0]^* \equiv w_2(W) \pmod 2$.  Let $x \in H^2(W)$ be the pull back of a generator of $H^2(\cp^{\infty})$, then 
$$\langle x, [S^2_M] \rangle =1, \ \ \langle x, [S^2_b] \rangle =0, \ \ \langle x, [S^2_0] \rangle =0,$$
therefore $x=[S^2_M]^*$.

Now $H_{4}(W) \cong \mathbb Z^{2}$ with a basis $\{ u, v\}$, where $u=i_*(u')$ is  the image of a generator $u'$ of $H_{4}(M)$ under the induced map $H_4(M) \to H_4(W)$, $v$ is represented by a cycle $C$ such that $C \cap V = D^{2} \times S^{2}_M \subset D^{2} \times M$. By standard algebraic topology we have 
\begin{equation}\label{eqn:h4}
\langle z^2, u \rangle = \langle z^2 ,v \rangle =\langle zx ,u \rangle=\langle zx ,v \rangle=0, \langle x^{2}, u \rangle =d, \ \langle x^{2}, v\rangle =0.
\end{equation}
(For example, $\langle ([S^2_b]^*)^2,u \rangle = \langle (i^*[S^2_b]^*)^2, u' \rangle$, but $i^*[S^2_b]^*=0 \in H^2(M)$ since $\langle i^*[S^2_b]^*, [S^2_M] \rangle=0$. From this and  similar computations we get above identities.)

The intersection pairings of $H_4(W)$ are
\begin{equation}\label{eqn:h42}
u \cdot u =0, \ u \cdot v =1, \ v \cdot v = \lambda,
\end{equation}
and the evaluations of the Pontrjagin class are 
\begin{equation}\label{eqn:h43}
\langle \frac{p_{1}}{2}(W), u \rangle =2k, \ \langle \frac{p_{1}}{2}(W), v \rangle= \mu,
\end{equation}
where the integers $\lambda$ and $\mu$  depend on the framing of the $3$-surgery. 

The $S_i$-invariants of $Y$ are
$$S_{i}(Y)=\langle e^{z/2}ch(E_{i}) \hat A(W), [W, \partial W]\rangle \in \mathbb Q /\mathbb Z \  \ (i=1,2,3)$$
where $E_{i}$ is a (virtual) complex vector bundle pulled back from a vector bundle over $\cp^{\infty}$. Since $z^2=0$,  we have
$S_{i}(Y)=\langle ch(E_{i}) \hat A(W), [W, \partial W]\rangle$, which
is a linear combination of 
$$\langle \bar \alpha^2, [W,\partial W] \rangle, \ \ \langle \bar \alpha \bar \beta, [W, \partial W]\rangle, \ \ \langle \bar \beta^2, [W, \partial W]\rangle,$$  
where $\bar \alpha, \bar \beta \in H^4(W, \partial W)$ is a preimage of $x^2$ and $\frac{p_1}{2}(W)$, respectively.
 From Equations (\ref{eqn:h4}), (\ref{eqn:h42}) and (\ref{eqn:h43}) we have
\begin{eqnarray*}
\langle \bar \alpha^2, [W, \partial W]\rangle & = & -d^2 \lambda \\
\langle \bar \alpha  \bar \beta, [W, \partial W] \rangle & = &d (\mu-2k \lambda)\\ 
\langle (\bar \beta^2, [W,\partial W] \rangle & = & 2k(\mu -2 \lambda k) + 2k \mu.
\end{eqnarray*}
By the formulas given in \cite[(2.4)]{KrSt2} the invariants of $Y$ are
\begin{eqnarray*}
S_1(Y) & = & \frac{1}{2^5 \cdot 7}(2k(\mu - 2 \lambda k) +2k \mu) \\
S_2(Y) & = & \frac{-1}{2^3 \cdot 3} (d^2 \lambda + d(\mu -2\lambda k)) \\
S_3(Y) & = & \frac{-1}{2 \cdot 3}(4d^2 \lambda + d(\mu -2\lambda k)).
\end{eqnarray*}

We analyse these values case by case. Since the boundary mapping class is an element of order $2$, we only need to consider the $2$-primary part. When $d$ is odd, notice that $S_2(Y)$ takes values in $\z/12$, we must have $\lambda + \mu \equiv 0 \pmod 2$. Chang the framing of the embedding $S^{3} \times D^{4} \subset X$ on which we do surgery, using $\gamma \in \pi_3SO(4)$, we get a different $W$, where $\mu$ will chang to $\mu + b$ and $\lambda$ will change to $\lambda + 2a+b$, where $\frac{1}{2}p_1(\gamma)=b$ and $\chi(\gamma)=2a+b$, $a$ and $b$ are arbitrary integers. Take $b=-\mu$ and $a=(\mu - \lambda)/2$, we get $\lambda=\mu =0$ and hence all the $S_i(Y)$'s are $0$.

When $d \equiv 0 \pmod 4$, $S_3(Y)$ is automatically $0$. We may choose $a$ and $b$ such that $\mu =0$, $\lambda =0$ or $1$. Under this choice $S_2(Y)=0$, $S_1(Y)=0$ or $k^2/2^3 \cdot 7$. In the latter case, when $l$ is even, one has $k^2 \equiv 0 \pmod{2^3}$, hence $S_1(Y)=0$; when $l$ is odd, from the simplification of the lattice $L(M)$ in the proof of Theorem \ref{thm:kernel}, one sees that disk supported diffeomorphisms are all isotopic to the identity. Therefore the boundary diffeomorphism is trivial in $\K(M,D)$.

When $d =4r+2$, $S_3(Y)$ is again automatically $0$. When $\lambda$ is odd we take $b=\mu+2$, when $\lambda$ is even we take $b=\mu$. This will make $S_2(Y)=0$, and $S_1(Y)=(2k(\mu+1)-\lambda k^2)/2^3 \cdot 7$ when $\lambda$ is odd, or $S_{1}(Y)=(2k\mu - \lambda k^{2})/2^{3} \cdot 7$ when $\lambda$ is even. Choosing appropriate $a$ such that $\lambda=0$ or $1$ according to the parity of $\lambda$ and $\mu$, we will get $S_{1}(Y)=0$. For example when $\lambda$ is odd and $\mu$ is even
$$2k(\mu+1)-\lambda k^2 \equiv (2k-k^2) \equiv 2^2 (2r+3l)(2r+3l+1) \equiv 0 \pmod{2^3}.$$
Therefore $S_1(Y)=0$. The other cases are similar.
\end{proof}

In the above theorem we have shown that for manifolds $M$ which look like $3$-dimensional complete intersections with $b_3(M)=0$, the boundary mapping class is trivial in the relative mappping class group. To extend this result to general case, we need the following lemma and proposition.

\begin{lem}
Let $M$ be a closed $n$-manifold ($n \ge 3$), $D, D' \subset M$ be closed disjointly embedded  $n$-disks. Let $N=M -\mathring D$, $N'=N - \mathring D'$; $f$ be the boundary diffeomorphism  in a neighborhood of $\partial D$,  $f'$ be the boundary diffeomorphism in a neighborhood of $\partial D'$. Assume $f$ is isotopic to $\mathrm{id}_{N}$ rel $\partial N$. Then $f$ is isotopic to either  $\mathrm{id}_{N'}$ or  $f'$ rel $\partial N'$.
\end{lem}
\begin{proof}
We have  the following observation from the proof of \cite[Hilfsatz]{Wall1}: Let $h \colon N \times I \to N \times I$ be an isotopy between $f$ and $\mathrm{id}_{N}$, $x$ be the center of  $D'$, then $h(x \times I)$ is isotopic to $x \times I$. By the ambient isotopy theorem, we may modify $h$ such that $h(x \times I)=x \times I$. Now $h(D' \times I)$ is a tubular neighborhood of $x \times I$, by the tubular neighborhood theorem we may assume that $h(D' \times I)=D' \times I$ and it is a bundle map, thus $h|\partial D' \times I$ is either the identity or the boundary diffeomorphism. Reparametrize the boundary of $\partial N' \times I$, we may view $h$ as an isotopy between $f$ and $\mathrm{id}_{N'}$ or $f'$ rel $\partial N'$.
\end{proof}

\begin{prop}\label{thm:bd3}
Let $M_{1}$, $M_{2}$ be closed oriented $n$-manifolds, $N_{i}=M_{i} - D_{i}$. Assume that the boundary diffeomorphism is isotopic to $\mathrm{id}_{N_{i}}$ rel $\partial N_{i}$ for $i=1,2$. Then for $N=(M_{1} \# M_{2})-D$, the boundary diffeomorphism is also isotopic to $\mathrm{id}_{N}$ rel $\partial N$.
\end{prop}
\begin{proof}
Let $D' \subset N_1$ be an embedded disk disjoint from $D_1$, $N_{1}'=N_{1}-D'$, then the boundary diffeomorphism $f$ in a neighborhood of $\partial D_1$ is isotopic rel $\partial N_{1}'$ to either $\mathrm{id_{N'}}$ or the boundary diffeomorphism $f'$ in a neighborhood of $\partial D'$. In the second case $f'$ is isotopic to $\mathrm{id}_{N_{2}}$ rel $\partial D'$. Put these two isotopies together we have an isotopy from $f$ to $\mathrm{id}_{N}$ rel $\partial D_{1}$
\end{proof}

Combining Theorem \ref{thm:bd1}, Proposition \ref{thm:bd3} and the fact that for $\# \, r(S^{3} \times S^{3})$ the boundary diffeomorphism is isotopic to $\mathrm{id}$ rel $D$ (\cite[Lemma 3]{Kr78}), we have proved 
\begin{thm}\label{prop:forgetful}
The forgetful homomorphism $\K(M,D) \to \K(M)$ is an isomorphism.
\end{thm}

\section{The non-spin case}\label{sec:nonspin}
So far we have studied the representation $\rho \colon \m(M) \to \Sp(b_3(M),\z)$ and the variation map $v_{x,p} \colon \T(M) \to (H^3(M) \times H^3(M))/\Delta_{(d(M),k(M))}$ parallel for $M$ spin and non-spin, and prove the theorems about $\K(M,D)$ and $\K(M)$ for $M$ spin. In this section we determine the kernels $\K(M,D)$ and $\K(M)$ in the non-spin case, alone the lines of section \ref{7-mfds} --\ref{sec:K}.

\subsection{Simply-connected nonspin $7$-manifolds}
Let $\mathcal N'$ denote the set of oriented diffeomorphism classes of simply-connected closed non-spin $7$-manifolds $N$ whose homology groups are isomorphic to that of $S^{2} \times S^{5}$. By \cite{KrSt}, the set $\mathcal N'$ is determined by the Kreck-Stolz invariant, which is defined as follows. 

The normal $2$-type of such a manifold $N$ is 
$$p \colon B=\cp^{\infty} \times B\mathrm{Spin}\to BSO \times BSO \stackrel{\oplus}{\longrightarrow} BSO,$$
where the map $\cp^{\infty} \to BSO$ is the classifying map of the canonical complex line bundle $\eta$ over $\cp^{\infty}$, the map $ B\mathrm{Spin}\to BSO$ is the projection, and the map $\oplus$ is induced by the Whitney sum. In this case, the corresponding $B$-bordism group $\Omega_{7}(B,p)=\omsp_{7}(\cp^{\infty}; \eta)=0$ (\cite{KrSt}). Therefore we may take $G \colon W^{8} \to B$ such that $\partial (W, G)=(N, \ov{\nu})$, and the signature of $W$ is zero. We have cohomology classes $\alpha = G^*(x^2)$, $\beta=G^*(\frac{p_1}{2}) \in H^4(W)$, and $\ov{\alpha}$, $\ov{\beta} \in H^{4}(W, \partial W)$ which are preimages of $\alpha$ and $\beta$. Now we define  characteristic numbers
$$\begin{array}{rcl}
\sigma_1(N) &= & \langle \ov{\alpha}^2, [W,\partial W] \rangle
\\
\sigma_2(N) & = & \langle \ov{\alpha} \cup \ov{\beta}, [W,\partial W] \rangle \\
\sigma_3(N) & = & \langle  \ov{\beta}^2, [W,\partial W] \rangle  \\
\end{array}$$
and obtain a well-defined invariant $KS(N)=([\sigma_1(N), \sigma_2(N), \sigma_3(N)]) \in \z^3/L$,  where $L\subset \z^{3}$ is the lattice generated by the corresponding characteristic numbers of closed $8$-dimensional $B$-manifolds with signature $0$. A set of generators of the bordism group is 
\begin{equation}
\begin{array}{rcl}
e_1 & =  & B^8 - 8 \cdot 28 \hp^2 \\
e_2 & = & [\cp^2 \times \cp^2 \to \cp^{\infty}]-\hp^2\\
e_3 & = & [\cp^4 \to \cp^{\infty}]-\hp^2\\
\end{array}
\end{equation}
The characteristic numbers are 
\begin{equation}
\begin{array}{c|ccc}
 & e_1 & e_2 & e_3 \\
\hline
\sigma_1 & 0 & 6 & 1  \\
\sigma_2 & 0 & -6 & -3 \\
\sigma_3 & -8 \cdot 28 & 8 & 8 \\
\end{array}
\end{equation}
and the lattice $L$ is generated by the column vectors of this table.

We simplify the lattice $L$ and get $\z^3 /L \cong \z/4 \oplus \z/(3 \cdot 8 \cdot 28)$, with the summands determined by 
$$\lambda = \sigma_1 + 3 \sigma_2 + \sigma_3 \pmod 4$$
$$\mu = 6 \sigma_1 + 10 \sigma_2 + 3 \sigma_3 \pmod{3 \cdot 8 \cdot 28}.$$
The obstruction $\theta(W,G)$ to turning a $B$-cobordism $(W,G)$ between two such manifolds to an $h$-cobordism is elementary if and only if they have the same $KS$-invariant. Therefore, analogous to Lemma \ref{lem:inv}, we have
\begin{lem}
There is an injective map between sets 
$$KS=[(\lambda, \mu)] \colon \mathcal N' \to \z^3/L = \z/4 \oplus \z/(3 \cdot 8 \cdot 28).$$ 
\end{lem}

There is an action of $\Theta_7$ on $\mathcal N'$ by taking connected-sum with a homotopy $7$-sphere. Let $\Sigma= \partial E_8$, a generator of $\Theta_7$, we may compute the invariants of $N \# \Sigma$  using $(W \natural E_8) \# (-8 \hp^2)$, where $W$ is the couboundary of $N$, we get 
\begin{equation}\label{eqn:mu}
\lambda(N \# \Sigma) = \lambda (N), \ \ \ \mu(N \# \Sigma)= \mu(N) - 24.
\end{equation}
Therefore the action of $\Theta_7$ on $\mathcal N'$ is free, and we have an injective map
$$[(\lambda, \mu)] \colon \ov{\mathcal N'} \to \z/4 \oplus \z/24.$$
where $\ov{\mathcal N'}$ is the orbit set.

\subsection{The generalized Kreck-Stolz invariant}
To defined the generalized Kreck-Stolz invariant in the non-spin case, we consider the following fibration
$$p \colon B=\cp^{\infty} \times K(\z,4) \times B\mathrm{Spin}\to BSO \times BSO \stackrel{\oplus}{\longrightarrow} BSO,$$
where $\cp^{\infty} \to BSO$ is the classifying map of the canonical complex line bundle $\eta$ over $\cp^{\infty}$ and $K(\z,4) \times  B\mathrm{Spin} \to  B\mathrm{Spin} \to BSO$ is the composition of natural projections. For a diffeomorphism $f \in \K(M,D)$, analogous to the spin case, we have maps $M_f \to \cp^{\infty}$ and $M_f \to K(\z,4)$. Now the difference between the stable normal bundle of $M_f$ and the pullback of $\eta$ has a preferred spin structure, which induces a map $M_f \to B\mathrm{Spin}$. Putting these together we have a map $g \colon M_{f} \to B$, which is a lifting of the normal Gauss map of $M_f$. All maps involved are unique up to homotopy in the appropriate sense. The corresponding bordism group is $\om_{*}(B,p) = \omsp_{*}(\cp^{\infty} \times K(\z,4);\eta)$, the twisted spin cobordism group of $\cp^{\infty} \times K(\z,4)$. A computation using the Atiyah-Hirzebruch spectral sequence shows that  $\omsp_{7}(\cp^{\infty} \times K(\z,4);\eta)=0$ (c.~f.~\cite[Theorem 6]{Kr18}). Now $H^4(B;\mathbb Z)$ is generated by $x^2\in H^4(\cp^{\infty})$, $y \in H^4(K(\z,4))$ and $\frac{p_1}{2} \in H^4(B\mathrm{Spin})$. Note that
$$g^{*}(\frac{p_{1}}{2})=\frac{p_1}{2}(\nu M_{f}-g^{*}\eta) = \frac{1}{2}(-d-48l-d)g^*y=-(d+24l)g^*y$$
Therefore $\ker g^*$ is generated by $x^2-dy$ and $\frac{p_1}{2}+(d+24l)y$. Let $G \colon W^8 \to B$ be a coboundary of $g \colon M_{f} \to B$, $G^*(x^2-dy)=\alpha$, $G^*(\frac{p_1}{2}+(d+24l)y)=\beta$, there are preimages $\ov \alpha, \ov \beta \in H^4(W, \partial W)$ of $\alpha$ and $\beta$, respectively. We have characteristic numbers
$$\left\{ \begin{array}{rcl}
\sigma_1 & = & \langle \ov \alpha \cup \ov \alpha ,[W,\partial W] \rangle \\
\sigma_2 & = & \langle \ov \alpha \cup \ov \beta ,[W,\partial W] \rangle \\
\sigma_3 & = & \langle \ov \beta \cup \ov \beta ,[W,\partial W] \rangle \\
\end{array} \right.$$
and define the generalied Kreck-Stolz invariant of $f$ as 
$$S([f])=[(\sigma_1, \sigma_2, \sigma_3)] \in \z^{3}/\bar L(M).$$

The lattice $\bar L(M)$ is given by the characteristic numbers of the generators of the subgroup of $\omsp_8(\cp^{\infty} \times K(\z,4);\eta)$ consisting of bordism classes of signature $0$. A set of generators is given as follows: (a sketch of the computation is given in the Appendix)
\begin{equation}
\begin{array}{rcl}
e_1 & =  & B^8 - 8 \cdot 28 \hp^2 \\
e_2 & = & [\cp^2 \times \cp^2 \to \cp^{\infty}]-\hp^2\\
e_3 & = & [\cp^4 \to \cp^{\infty}]-\hp^2\\
e_4 & = & [\hp^2 \to K(\z,4)]-\hp^2\\
e_5 & = & [S^4 \times S^4 \to K(\z,4)]\\
e_6 & = & [\cp^2 \times S^4 \to \cp^{\infty} \times K(\z,4)]
\end{array}
\end{equation}
The lattice $\bar L(M)$ is generated by the column vectors in the following table
\begin{equation}\label{tab:nonspin}
\begin{array}{c|cccccc}
 & e_1 & e_2 & e_3 & e_4 & e_5 & e_6\\
\hline
\sigma_1 & 0 & 6 & 1 & d^2 & 2d^2 & -2d \\
\sigma_2 & 0 & -6 & -3 & -d(d+24l-1) & -2d(d+24l) & 3d+24l\\
\sigma_3 & -8 \cdot 28 & 8 & 8 & (d+24l-1)^2-1 & 2(d+24l)^2 & -4(d+24l)\\
\hline
\end{array}
\end{equation}

Similar to Proposition \ref{prop:injhomo} we have

\begin{prop}\label{prop:ksnonspin}
The generalized Kreck-Stolz invariant defines an injective homomorphism
$$S \colon \K(M, D) \to \z^3/\bar L(M)$$
\end{prop}
In the non-spin case, it can be shown that the Kreck-Stolz invariant of the boundary diffeomorphism is always zero (see \S \ref {subsec:bdns}). Therefore in this case we  have an isomorphism $\K(M,D) \stackrel{\cong}{\rightarrow} \K(M)$.

\subsection{The group $\K(M,D)$ is isomorphic to $\K(M)$}\label{subsec:bdns}
In the non-spin case the generalized Kreck-Stolz invariant of the boundary diffeomorphism can be calculated directly from the $\mathrm{spin}^c$-coboundary $V$ whose construction is the same as Equation (\ref{eqn:V}). Let $L$ be a complex line bundle over $V$ with $c_1(L)=(x,z) \in H^2(V) = H^2(M) \oplus H^2(E)$, where $x$ and $z$ are generators of corresponding cohomology groups, then the spin structure on $\nu V \oplus L$ restricts to the non-trivial  spin-structure on $S^1\times D$. The inclusion map induces an isomorphism $H_4(D^2 \times M) \to H_4(V)$, therefore the intersection pairing $H_4(V) \times H_4(V) \to \z$ is trivial. Hence all the characteristic numbers $\sigma_i=0$ ($i=1,2,3$). This shows that the generalized Kreck-Stolz invariant of the boundary diffeomorphism is zero, hence it is isotopic to the identity rel $D$. We have an isomorphism $\K(M,D) \stackrel{\cong}{\rightarrow} \K(M)$. Together with Proposition \ref{prop:ksnonspin} this finishes the proof of Theorem \ref{thm:kernel} in the non-spin case.

\subsection{The group $\m(S^2 \wt D^4, \partial)$}
In the non-spin case, our manifold $M$ under consideration can be decomposed as a union
$$M=(S^2\wt D^{4}) \cup ((S^{2} \wt D^{4}) \# r(S^3 \times S^3))$$
where $S^{2} \wt D^{4}$ is the total space of the nontrivial $D^{4}$-bundle over $S^{2}$.
Let $S^{2} \wt S^{4}$ be the total space of the non-trivial linear $S^{4}$-bundle over $S^{2}$, $\T(S^2 \wt S^4)$ be the group consisting of isotopy classes of diffeomorphisms acting trivially on homology, similar to the constructions in \S \ref{subsec:s2} (Lemma \ref{lem:bij}), there is a surjective map $\Psi \colon \T(S^2 \wt S^4) \to \mathcal N'$, and a commutative diagram
$$\xymatrix{
\m(S^2 \wt D^4, \partial) \ar[r] & \T(S^2 \wt S^4) \ar[r]^{\ \ \ \ \Psi} \ar[d]^S & \mathcal N' \ar[d]^{KS} \\
& \z^3/L \ar[r]^= & \z^3/L}$$
where the two maps in the first row are surjective and the two vertical maps are injective.

From Equation (\ref{eqn:mu}),  the homomorphism  $\Theta_7 \to \m(S^2 \wt D^4, \partial)$ given by disk-supported diffeomorphisms is injective. Denote the cokernel by $\ov{\m}(S^2 \wt D^4, \partial)$,  then there are homomorphisms
$$\ov{\m}(S^2 \wt D^4, \partial) \to \ov{\mathcal N'} \to \z/4 \oplus \z/24$$
where the first one is surjective and the second one is injective.

An upper bound of $\ov \m(S^{2} \wt D^{4}, \partial)$ is obtained as follows. Let $\eta$ be the pull back of the canonical complex line bundle over $\cp^{\infty}$ to the 3rd stage Postnikov tower $P_{3}(S^{2})$. Then similar to Lemma \ref{lem:surj}, there is a surjective homomorphism
$$\omst_7(P_3(S^2);\eta) \twoheadrightarrow \ov{\m}(S^2 \wt D^4, \partial) $$
with $\omst_7(P_3(S^2);\eta) \cong \z/6$ (computation in the Appendix, see Lemma \ref{lem:order2}).

Let $\xi_m$ be an oriented $4$-dimensional real vector bundle over $\cp^2$ with
$$w_2(\xi_m)=0, \ \ e(\xi_m)=1, \ \ p_1(\xi_m)=4m+2 \ \ (m \in \z).$$
These bundles can be obtained by pulling back suitable vector bundles over $S^4$.
Let $W_m=D(\xi_m)$ be the disk bundle and $N_m=\partial W_{m}=S(\xi_m)$. Then we may use $W_m \# (-\hp^2)$ to compute the invariants of $N_m$ and get $$\lambda(N_m)=4m^2+6m, \ \ \  \mu(N_m)=12m^2+16m.$$
Especially, $(\lambda(N_{-1}), \mu(N_{-1}))=(-2,-4)$. Recall from Equation (\ref{eqn:mu}) we have $(\lambda((S^2 \wt S^4) \# \Sigma), \mu((S^2 \wt S^4) \# \Sigma)) = (0,-24)$. This shows that  $ \frac{1}{4}\mu \colon \mathcal N' \to \z/(6\cdot 28)$ is an isomorphism, and $N_{-1}$ is a generator.

\begin{thm}
There are isomorphisms
$$\m(S^2 \wt D^4, \partial) \stackrel{\cong}{\longrightarrow} \T(S^2 \wt S^4) \stackrel{\cong}{\longrightarrow} \z/(6 \cdot 28)$$
where the second isomorphism is given by $\mu/4$.
\end{thm}

\begin{lem}
There is a commutative diagram
$$\xymatrix{
\m(S^2 \wt D^4, \partial)  \ar[d] \ar[r]^{\ \ S} & \z^3/L \ar[d] \\
\K(M) \ar[r]^{S} & \z^3/\bar L(M) }$$
in which the vertical homomorphisms are surjective and the horizontal ones are injective.
\end{lem}

Simplify the matrix in (\ref{tab:nonspin}), by letting
$$\lambda=\sigma_1+3\sigma_2+\sigma_3$$
$$\mu=6\sigma_1+10\sigma_2+3\sigma_3$$
we have the following table (in which $\delta=d/2$, recall that in the non-spin case $d$ is even)
$$\begin{array}{c|ccccc}
\mu/4 & 6\cdot 28 & 0 & -\delta^2-48\delta l +2\delta +432l^2 -36l & -2(\delta^2+96\delta l -864l^2)& 3(\delta-4l)\\
\lambda & 0 & -4 & -4\delta^2-96\delta l +2\delta +576l^2-48l & -2(4\delta^2 +48\delta l -576l^2) & 6(\delta -4l)
\end{array}$$
and $\z^3/\bar L(M) \cong \z^2/L(M)$ with $L(M)$ given by the column vectors of this matrix.

Now consider the homomorphism $\z^3 / L \cong \z/(6 \cdot 28) \to \z^2/L(M)$. Since $(\frac{1}{4}\mu(N_{-1}), \lambda(N_{-1}))=(-1,-2)$, the image of $1 \in \z/(6 \cdot 28)$ is $ [(-1,-2)]$. We summarize the above calculations in the following commutative diagram
$$\xymatrix{
\Theta_7 \ar[d] \ar[dr]^{1 \mapsto 6} & \\
\m(S^2 \wt D^4, \partial) \ar[r]^{\ \ \mu/4}_{\ \ \cong} \ar@{->>}[d] & \z/(6\cdot 28) \ar[d]^{1 \mapsto [(-1,-2)]} \\
\K(M) \ \ar@{>->}[r]^{(\mu/4, \lambda)} & \z^2/L(M) }$$

Therefore to determine the group $\K(M)$ we only need to compute the order of $[(1,2)]$ in $\z^2/L(M)$. Since the order is necessarily a factor of $6\cdot 28$, we only need to calculate modulo $3$, $7$ and $8$ respectively. The result is presented in the table in Theorem \ref{thm:kernel} (non-spin case).

\subsection{Generators of $\m(S^2 \wt S^4)$}
Let $S \colon S^3 \times D^3 \to S^4 \wt S^2$ be an embedding such that the restriction of $S$ on $S^3 \times \{0\}$ is a generator of $\pi_3(S^2 \wt S^4)$. 
Let $\mathbf{t}_{S, \alpha}$ be the Dehn twist in $S^{3} \times D^{3} \subset S^{2} \wt S^{4}$ with parameter $\alpha \in \pi_{3}SO(4)$, $W=E_{\alpha}\cup_{S^{3}\times D^{4}} S^{2}\wt S^{4} \times D^{2}$ be the coboundary of the mapping torus of $\mathbf{t}_{S, \alpha}$ constructed in Lemma \ref{lem:dehn2}.
Similar to the construction in \S \ref{subsec:gen},  a basis of $H_4(W) \cong \z^3$ is $[S^4_b]$, $[Y]$ and $[S^4_0]$, the intersection form is represented by the same matrix as in \S \ref{subsec:gen},  and $\frac{1}{2}(p_1(\nu W)-p_1(\eta))=(2b-\chi) [S^4_b]^* + A [Y]^*$ for some $A \in \z$. By similar calculation we have 
$$\mu(\mathbf{t}_{S,\alpha}) =  (12A+20) b-(16+6A-3A^2)\chi .$$
This integer is divisible by $4$ for any $\chi, b \in \z$. This implies that $A$ is an even integer. Then we may change the embedding $S$ by some $v \in \pi_3SO(3)$ such that $A =0$ (see Equation \ref{eqn:y}). We denote this embedding by $Z \colon S^3 \times D^3 \to S^2 \wt S^4$. Then $\frac{1}{4}\mu(\mathbf{t}_{Z,\alpha})=5b - 4\chi \in \z/(6 \cdot 28)$. Let $\rho \in \pi_3SO(4)$ be the element with $\chi=0$ and $b=1$, then $\mathbf{t}_{Z,\rho}$ is a generator of $\T(S^{2} \wt S^{4})$.

\section{Group theoretic properties}\label{sec:gp}
\subsection{Generators of $\m(M)$}
In this subsection we write down explicit generators of $\m(M)$ for $r \ge 1$, and hence prove Theorem \ref{thm:gen}. For the notations of the Dehn twists we use the notations introduced before Theorem \ref{thm:gen}.

First of all, similar to the case of surfaces, from a presentation of the symplectic group $\Sp(2r,\z)$ (c.f.~\cite[Theorem 1]{Bir71}) one sees that the set of Dehn twists $\mathbf{e}_i$, $\mathbf{f}_j$ ($1 \le i,j\le r$) and $\mathbf{f}_{i,i+1}$ ($1 \le i \le r-1$) maps to a set of generators of $\Sp(2r,\z)$.

Next we give generators of the Torelli group $\T(M)$. Recall that there is a central extension
$$0  \to \K(M) \to \T(M) \stackrel{v_{x,p}}{\longrightarrow} (H \times H) /\Delta_{d,k} \to 0$$
Let $a_{i}$, $b_{j}$ and $a_{i}'$, $b_{j}'$ ($1\le i,j \le r$) be the standard symplectic basis of $0 \times H$ and $H \times 0$ respectively. Let 
$$\mathbf{a}_i=\mathbf{t}_{S_i, \rho}, \ \ \mathbf{b}_j=\mathbf{t}_{T_j,\rho}, \ \ \mathbf{a}_i'=(\mathbf{t}_{S_i, \epsilon})^{-1} \circ \mathbf{g}_i, \ \ \mathbf{b}_j'=(\mathbf{t}_{T_j,\epsilon})^{-1} \circ \mathbf{h}_j,$$ 
by the definition of $v_{x,p}$,  these are pre-images of $a_{i}$, $b_{j}$ and $a_{i}'$, $b_{j}'$ respectively. The commutators of these diffeomorphisms are in $\K(M)$. Now we compute the $S$-invariants of  the commutators.

We only need to consider the commutators of these diffeomorphisms when the corresponding embedded $3$-spheres of the Dehn twists have non-empty intersection. 

It is shown in \cite[Theorem 2]{Kry02} that $[\mathbf{a}_i, \mathbf{b}_i]=\mathbf{f}_c$, a generator of the disk supported diffeomorphisms. We will give an alternative proof of this fact by computing the generalized Kreck-Stolz invariant of the commutator  $[\mathbf{a}_i, \mathbf{b}_i]$. The commutator is a composition of Dehn twists $\mathbf{t}_{S_i, \rho} \circ \mathbf{t}_{T_i, \rho} \circ \mathbf{t}_{S_i, -\rho} \circ \mathbf{t}_{T_i, -\rho}$. Denote the $4$-dimensional vector bundles over $S^4$ corresponding to the Dehn twists by $E_1$, $E_2$, $E_1'$ and $E_2'$, then by Lemma \ref{lem:dehn2}
$$W = (M \times D^2) \cup D(E_1) \cup D(E_2) \cup D(E_1') \cup D(E_2')$$
is a coboundary by which we can compute the generalized Kreck-Stolz invariant of the diffeomorphism. Here we take $t_1$, $t_2$, $t_3$, $t_4 \in S^1$ in order compatible with the natural orientation of $S^1$, and glue the disk bundles in this order to $M \times D^2$ in a neighborhood of $S^3 \subset M \times \{t_i\}$ using the chosen framing. Therefore the $4$-dimensional homology group of $W$ is isomorphic to $\z^7$ ,with a basis consisting of the following classes
\begin{enumerate}
\item $S_1$, $S_2$, $S_1'$, $S_2'$, represented by the base spheres of the  vector bundles, 
\item $Y_1=[D_1 \cup C_1 \cup D_1']$, $Y_2=[D_2 \cup C_2 \cup D_2']$, where $D_1$, $D_1'$, $D_2$, $D_2'$ are fibers of the corresponding disk bundles, and $C_1=S^3 \times [t_1, t_3] \subset M \times S^1$, $C_2=S^3 \times [t_2, t_4] \subset M \times S^1$ whose boundaries are the cores of the embeddings $S_i(S^3 \times \{0\}) \subset M \times \{t_1\}$, $S_i(S^3 \times \{0\}) \subset M \times \{t_3\}$ and $T_i(S^3 \times \{0\}) \subset M \times \{t_2\}$, $T_i (S^3 \times \{0\})\subset M \times \{t_4\}$.
\item $u$ which is the image of a generator of $H_4(M)$.
\end{enumerate}
The intersection matrix of $W$ with respected to this basis is 
$$\left ( \begin{array}{ccccccc}
0 & 0 & 0 & 0 & 1 & 0 & 0 \\
0 & 0 & 0 & 0 & 0 & 1 & 0 \\
0 & 0 & 0 & 0 & -1 & 0 & 0 \\
0 & 0 & 0 & 0 & 0 & -1 & 0 \\
1 & 0 & -1 & 0 & 0 & 1 & 0 \\
0 & 1 & 0 & -1 & 1 & 0 & 0 \\
0 & 0 & 0 & 0 & 0 &  0 & 0 \end{array} \right )$$
Note that since we may push $C_2$ into the interior of $M \times D^2$ (with boundary fixed), the intersection number of $Y_1$ and $Y_2$ equals the intersection number of core of $S_i$ and $T_i$, which is $1$. The other intersection numbers are clear from the geometric construction. The evaluation of the Pontrjagin class of $W$ on these classes are 
$$\langle p_1(W), S_1 \rangle =\langle p_1(W), S_2 \rangle =4 \ \ \langle p_1(W), S_1' \rangle=  \langle p_1(W), S_2' \rangle =-4$$ 
$$\langle p_1(W), Y_1 \rangle = \langle p_1(W), Y_2 \rangle =0.$$
Therefore $p_1(W)=4(S_1^*+S_2^*-(S_1')^*-(S_2')^*)$, and the Poincar\' e dual of a preimage of $p_1(W)$ is $4(Y_1+Y_2-S_1-S_2)$. Note that since the Dehn twists are diffeomorphisms of $\#_r(S^3 \times S^3)$, we only need to compute the first component of the generalized Kreck-Stolz invariant, the other two components are automatically zero. By standard algebraic topology we have the first component of generalized Kreck-Stolz invariant is $\frac{1}{8}s_1=1$. Therefore $[\mathbf{a}_i, \mathbf{b}_i]=\mathbf{f}_c$.

Analogously we compute the generalized Kreck-Stolz invariant of the commutator
\begin{eqnarray*}
 \mathbf{a}_1'\mathbf{b}_1\mathbf{a}_1'^{-1}\mathbf{b}_1^{-1} & = & (\mathbf{t}_{S_1, \epsilon})^{-1} \circ \mathbf{g}_1 \circ \mathbf{t}_{T_1,\rho} \circ (\mathbf{g}_1)^{-1} \circ  \mathbf{t}_{S_1, \epsilon} \circ (\mathbf{t}_{T_1,\rho})^{-1} \\ 
& = & \mathbf{t}_{S_1, -\epsilon} \circ \mathbf{t}_{S_1\#Z, \epsilon} \circ \mathbf{t}_{T_1,\rho} \circ \mathbf{t}_{S_1\#Z, -\epsilon} \circ  \mathbf{t}_{S_1, \epsilon} \circ \mathbf{t}_{T_1,-\rho}
\end{eqnarray*}
Denote the $4$-dimensional vector bundles over $S^4$ corresponding to the parameters of these Dehn twists by $E_1'$, $E_{01}$, $E_2$, $E_{01}'$, $E_1$ and $E_2'$, respectively, then
$$W=M \times D^2 \cup D(E_{1}') \cup D(E_{01}) \cup D(E_2)  \cup D(E_{01}') \cup D(E_1) \cup D(E_{2}')$$
is a coboundary from which we can compute the generalized Kreck-Stolz invariant. The $4$-dimensional homology group of $W$ is isomorphic to $\z^{11}$, with a basis consisting of 
\begin{enumerate}
\item $e_1'$, $e_{01}$, $e_2$, $e_{01}'$, $e_1$, and $e_2'$ represented by the base $4$-spheres of the corresponding vector bundles;
\item $Y_1=[D_{1}' \cup C_1 \cup D_{01}]$, $Y_2=[D_{01} \cup C_2 \cup D_{01}']$, $Y_3=[D_{01}' \cup C_3 \cup D_1]$ and $Y_4=[D_2 \cup C_4 \cup D_2']$, where $D_1'$, $D_{01}$, $D_{01}'$, $D_1$, $D_2$ and $D_2'$ are fibers of the corresponding disk bundles, and $C_i  \subset M \times D^2$ ($i=1, \cdots 4$) are submanifolds whose boundaries are the union of the cores of the corresponding embeddings;
\item the image of a generator of $H_4(M)$.
\end{enumerate}
The intersection matrix with respect to this basis is 
$$\left ( \begin{array}{ccccccccccc}
-1 & 0 & 0 & 0 & 0 & 0 & 1 & 0 & 0 & 0& 0 \\
0  & 1 & 0 & 0 & 0 & 0 & -1 & 1 & 0 & 0& 0 \\
0 & 0 & 0 & 0 & 0 & 0 & 0 & 0 & 0 & 1 & 0 \\
0 & 0 & 0 & -1 & 0 & 0 & 0 & 0 & -1 & 0& 0 \\
0 & 0 & 0 & 0 & 1 & 0 & 0 & 0 & -1 & 0& 0 \\
0 & 0 & 0 & 0 & 0 & 0 & 0 & 0 & 0& -1 & 0 \\
1 & -1 & 0 & 0 & 0 & 0 & 0 & 0 & 0& 0& 0 \\
0 & 1 & 0 & 0 & 0 & 0 & 0 & 0 & 0& 1 & 0 \\
0  & 0 & 0 & -1 & -1 & 0 & 0 & 0 & 0& 0& 0 \\
0  & 0 & 1 & 0 & 0 & -1 & 0 & 1 & 0& 0& 0 \\
0 & 0 & 0 & 0 & 0 & 0 & 0 & 0 & 0& 0& 0 \end{array} \right )$$
The evaluation of the Pontrjagin class of $W$ on these classes are 
$$\langle p_1(W), e_{1}' \rangle =  \langle p_1(W), e_{01}' \rangle = -2, \ \ \langle p_1(W), e_{1} \rangle = \langle p_1(W), e_{01} \rangle =2$$
$$\langle p_1(W), e_{2} \rangle =4, \ \ \langle p_1(W), e_{2}' \rangle = -4$$
$$\langle p_1(W), Y_1 \rangle = \langle p_1(W), Y_2 \rangle = \langle p_1(W), Y_3 \rangle= \langle p_1(W), Y_4 \rangle =0$$

In the non-spin case we get $\mu/4= 5 \in \z/2^a \times \z/3^b \times \z/7^c$,
which is a generator of  this group. Therefore the commutator $[\mathbf{a}_1', \mathbf{b}_1]$ is a generator of $\K(M)$.

In the spin case we have $\frac{1}{2}s_2 =  -1$,  $s_3 =   2  \equiv 0 \pmod 2$. Similar computations of the other commutators show that the $s_3$-invariant is always $0$. From the simplification of the lattices in the proof of Theorem \ref{thm:kernel}, we see when $d$ is odd the generalized Kreck-Stolz invariants of elements in $\K(M)$ all have zero $s_3$-component. Therefore $\K(M)$ is generated by commutators of these Dehn twists. When $d$ is even, we need an additional generator $\mathbf{z}=\mathbf{t}_{Z,\epsilon}$ to produce diffeomorphisms with nontrivial $s_3$-invariant (see Theorem \ref{thm:s2s4}).

Therefore the Torelli group $\T(M)$ is generated by the Dehn twists $\mathbf{a}_i$, $\mathbf{b}_j$, $\mathbf{a}_i'=(\mathbf{e}_i)^{-1} \circ \mathbf{g}_{i}$, $\mathbf{b}_j'=(\mathbf{f}_j)^{-1} \circ \mathbf{h}_{j}$ ($1 \le i,j \le r$), and $\mathbf{z}$ when $M$ is spin and $d$ is even. Together with the Dehn twists $\mathbf{e}_i$, $\mathbf{f}_j$ ($1 \le i,j\le r$) and $\mathbf{f}_{i,i+1}$ ($1 \le i \le r-1$) they form a set of generators of $\m(M)$. This proves Theorem \ref{thm:gen}.



\subsection{Lower homology groups}
In this subsection we compute the lower homology groups of $\m(M)$. The abelianization $(\m(M))^{\mathrm{ab}}$ equals to the first homology group $H_{1}(\m(M))$.  When $r=0$ the group $\m(M)$ itself is an abelian group. So we assume $r \ge 1$. Notice that from the commuting relations of the generators of $\T(M)$ described in the previous subsection, one also obtains the abelianization of $\T(M)$. We summarize the results in the following proposition. Recall that $H=H^{3}(M;\mathbb Z)$.

\begin{prop}
When $M$ is spin
$$(\T(M))^{\mathrm{ab}}= \left \{
\begin{array}{ll}
(H \times H )/\Delta_{d,k} & \textrm{$d$ odd} \\
(H \times H) /\Delta_{d,k} \oplus \z/2 & \textrm{$d$ even}
\end{array} \right. $$

When $M$ is non-spin $(\T(M))^{\mathrm{ab}}= (H \times H )/\Delta_{d,k}$.
\end{prop}

\begin{proof}[Proof of Theorem \ref{thm:ab}, Theorem \ref{thm:h2q}, Corollary \ref{cor:h2} and Corollary \ref{cor:diff0}]
Denote $\Gamma_{r}=\Sp(2r,\z)$, we compute $(\m(M))^{\mathrm{ab}}=H_1(\m(M))$ using the Hochschild-Serre spectral sequence of the exact sequence
$$1 \to \T(M) \to \m(M) \to \Gamma_{r} \to 1$$
The relevant $E_2$-terms are as follows: first of all (c.~f.~\cite[Appendix]{BCRR18})
$$E^2_{1,0}=H_1(\Gamma_{r})=\left \{
\begin{array}{ll}
\z/12 & \textrm{$r=1$} \\
\z/2 & \textrm{$r=2$} \\
0 & \textrm{$r \ge 3$}\end{array} \right.  , \  \ E^2_{2,0}=H_2(\Gamma_{r})=\left \{
\begin{array}{ll}
0 & \textrm{$r=1$} \\
\z \oplus \z/2 & \textrm{$r=2$} \\
\z & \textrm{$r \ge 3$}\end{array} \right.$$
Secondly from the exact sequence of $\Gamma_{r}$-modules
$$0 \to H \to (H\times H)/\Delta_{d,k} \to H/d \to 0$$
and the fact (c.~f.~\cite[Lemma A.2] {Kra19}) that for an abelian group $A$ the co-invariants 
$(\z^{2g} \otimes A)_{\Gamma_{r}} = 0$,
we have $E^2_{0,1}=H_1(\T(M))_{\Gamma_{r}}=0$ except for the case $M$ spin and $d$ even, where $E^{2}_{0,1} = \z/2$.

To determine the differential $d \colon  E^2_{2,0} \to E^2_{0,1}$ (in the case $M$ spin, $d$ even and $r \ge 2$), we compare the spectral sequences of the upper and lower lines of the following commutative diagram
$$ \xymatrix{
\T(\#_{r}(S^{3} \times S^{3})) \ar[r] \ar[d]  & \m(\#_r(S^{3} \times S^{3})) \ar[r] \ar[d] &  \Gamma_{r} \ar[d]^{=} \\
\T(M) \ar[r] &  \m(M) \ar[r] & \Gamma_{r} }$$
where $\T(\#_{r}(S^{3} \times S^{3}))^{\mathrm{ab}}=H$ is the standard $\Gamma_{r}$-module \cite{Kry02}, we see that the differential $d \colon E^2_{2,0}=H_2(\Gamma_r) \to E^2_{0,1}$ factors through $H_2(\Gamma_r) \to H_{\Gamma_{r}} \to E^2_{0,1}$. But $H_{\Gamma_{r}} =0$. Therefore the differential $d \colon E^2_{2,0} \to E^2_{0,1}$ is trivial.

So we have a short exact sequence
$$0 \to E^2_{0,1} \to H_1(\m(M)) \to H_1(\Gamma_r) \to 0$$
When $r \ge 3$, $H_1(\Gamma_r)=0$ therefore $H_1(\m(M))=E^2_{0,1}$. When $r=1$ or $2$, there is a splitting $H_{1}(\Gamma_r)= H_{1}(\m(\#_{r}(S^{3} \times S^{3}))) \to H_{1}(\m(M))$, therefore $H_1(\m(M))=H_1(\Gamma_r) \oplus E^2_{0,1}$. This proves Theorem \ref{thm:ab}.

Since the differential  on $E^2_{2,0}$ is trivial, this term survives to infinity and therefore there is a surjective homomorphism $H_2(\m(M)) \to H_2(\Gamma_r)$. The short exact sequence of groups $\mathrm{Diff}_0(M) \to \mathrm{Diff}(M) \to \m(M)$ induces a fiber bundle $B\mathrm{Diff}_0(M) \to B\mathrm{Diff}(M) \to B\m(M)$ with simply-connected fiber $B\mathrm{Diff}_0(M)$. There is a surjective homomorphism $H_2(B\mathrm{Diff}(M)) \to H_2(\m(M))$. Composed with the surjective homomorphism  $H_2(\m(M)) \to H_2(\Gamma_r)$ we get a surjective homomorphism $H_2(B\mathrm{Diff}(M)) \to H_2(\Gamma_r)$. This proves Corollary \ref{cor:h2}.

To compute $H_2(\m(M));\mathbb Q)$ ($r \ge 3$) we may use the short exact sequence
$$1 \to (H \times H)/\Delta_{d,k} \to \m(M)/\K(M) \to \Gamma_r \to 1$$
since $\K(M)$ is a finite group. The relevant terms in the Hochschild-Serre spectral sequence are
\begin{enumerate}
\item $E^2_{2,0}=H_2(\Gamma_r;\mathbb Q)\cong \mathbb Q$;
\item $E^2_{1,1}=H_1(\Gamma_r;(H \times H)/\Delta_{d,k}\otimes \mathbb Q)=0$ (cf.~\cite[Lemma A.3]{Kra19});
\item \begin{eqnarray*}
E^2_{0,2}& = & H_0(\Gamma_r;H_2((H\times H)/\Delta_{d,k};\mathbb Q)) \\
 & = & (\wedge^2H \otimes \mathbb Q) \otimes_{\Gamma_r} \mathbb Q= \wedge^2 \mathbb Q^{2r} \otimes_{\Sp(2r,\mathbb Q)} \mathbb Q \cong \mathbb Q \end{eqnarray*}
 where the reason for the last isomorphism is that the second exterior power $\wedge^2 \mathbb Q^{2r}$ of the standard representation of $\Sp(2r,\mathbb Q)$ is isomorphic to the direct sum of the trivial representation $\mathbb Q$ and a non-trivial irreducible representation.

\item $E^2_{2,1}=H_2(\Gamma_r;(H \times H)/\Delta_{d,k}\otimes \mathbb Q)=0$ (cf.~\cite[Lemma A.3]{Kra19});

\item $E^2_{3,0}=H_3(\Gamma_r;\mathbb Q)=0$ when $r \ge 3$ (cf.~\cite{Bor74}).
\end{enumerate}
From this we have $E^{\infty}_{2,0} \cong E^{\infty}_{0,2} \cong \mathbb Q$, $E^{\infty}_{1,1}=0$, therefore $H_2(\m(M);\mathbb Q) \cong \mathbb Q^2$.

Corollary \ref{cor:diff0} is a direct consequence of the Leray-Serre spectral sequence of the fiber bundle $B\mathrm{Diff}_0(M) \to B\mathrm{Diff}(M) \to B\m(M)$, and the facts $H_2(\m(M);\mathbb Q) \cong \mathbb Q^2$ and $H_2(B\mathrm{Diff}(M);\mathbb Q) \cong \mathbb Q^4$ (\cite[\S5.3]{GaRW18}).
\end{proof}

\subsection{Other properties}
\begin{thm}\label{thm:center}
The center of $\m(M)$ is $\K(M)$.
\end{thm}
\begin{proof}
Recall that elements in $\K(M)$ are represented by diffeomorphisms in $S^{2} \times D^{4}$ or $D^{4} \widetilde{\times} S^{2}$, hence they commute with the generators of $\m(M)$ represented by the Dehn twists described in Theorem \ref{thm:gen}. Therefore $\K(M)$ is in the center of $\m(M)$. We will show that $\m(M)/\K(M)$ is centerless. By equations (1.11) and (1.12), there is a short exact sequence
$$0 \to H \to \m(M)/\K(M) \to \Sp(\pi_3(M)) \to 1$$
where $\Sp(\pi_3(M))= (\z/d)^{2r} \rtimes \Sp(2r,\z)$. It's easy to compute the center of the semi-direct product $(\z/d)^{2r} \rtimes \Sp(2r,\z)$ --- it is trivial when $d \ne 2$, and is   $\{(0, \pm I)\}$ when $d=2$ (cf.~\cite[Lemma A.2]{Kra19}).

When the center of $\Sp_{3}(\pi_{3}(M))$ is trivial, the center of $\m(M)/\K(M)$ is contained in $H$. The action of $\Sp(2r,\z)$ on $H$ is the standard linear action, with $0$ being the only fixed element. This shows the center of $\m(M) / \K(M)$ is trivial. When the center of $\Sp_{3}(\pi_{3}(M))$ is $(0, \pm I)$,
let $a \in \m(M)/\K(M)$ be an element  whose image in $\Sp(2r,\z)$ is $-I$, then for any $b \in H $, the commutator $[a,b]=(-I)(b)-b=-2b$. Therefore $a$ is not in the center.
\end{proof}

\begin{proof}[Proof of Theorem \ref{thm:resfin}]
To show that $\m(M)/\K(M)$ is virtually torsion free we construct torsion free subgroups of $\m(M)/\K(M)$ of finite index as follows:  Let $\Sp(2r,\z)[m]$ be the principal congruence subgroup of $\Sp(2r,\z)$ of level $m$ ($m \ge 3$),  then it is a torsion-free subgroup of finite index of $\Sp(\pi_3(M))$. Let $(\m(M)/\K(M))[m]$ be the preimage of $\Sp(2r,\z)[m]$ in $\m(M)/\K(M)$, it fits in the exact sequence
$$0 \to H \to (\m(M)/\K(M))[m] \to \Sp(2r,\z)[m] \to 1$$
Therefore $(\m(M)/\K(M))[m]$ is a torsion-free subgroup of finite index in $\m(M)/\K(M)$.

Finally we show that $\m(M)/\K(M)$ is residually finite. From the commutative diagram of exact sequences \ref{diag:module} we see $ \m(\# r(S^{3} \times S^{3}))/\Theta_{7}$ is a subgroup of $\m(M)/\K(M)$ of finite index. It suffices to show that  $\m(\# r(S^{3} \times S^{3}))/\Theta_{7}$ is residually finite. This group is a subgroup of $H \rtimes \Gamma_{r}$ of index $2^{2r}$. Both $H$ and $\Gamma_{r}$ are residually finite, then the semi-direct product $H \rtimes \Gamma_{r}$ is also residually finite (e.~g.~by Theorem II of  \cite{Hew94} which says that an extension of a residually finite group by a residually finite group is residually finite if and only if the extension is residually virtually splittable). Therefore  $\m(\# r(S^{3} \times S^{3}))/\Theta_{7}$ is residually finite, hence also $\m(M)/\K(M)$.
\end{proof}

\section{Appendix: Computation of the bordism groups}
In the appendix we sketch the computation of the bordism groups used in the paper.

Let $P=P_3(S^2)$ be the 3rd stage Postnikov tower of $S^2$, which is a fibration over $K(\z,2)$ with fiber $K(\z,3)$ and $k$-invariant $x^2 \in H^4(K(\z,2))$, where $x \in H^2(K(\z,2))$ is a generator.  Let $L$ be the tautological complex line bundle over $K(\z,2)$, $\eta$ be the pullback line bundle over $P$, whose first Chern class equals a generator of $H^2(P)$. We compute the string bordism groups $\omst_7(P)$ and $\omst_7(P;\eta)$ using the Atiyah-Hirzebruch spectral sequence.

We first compute the (co)-homology groups of $P$. The homology groups of $K(\z,3)$ were computed in the classical paper \cite[Theorem 23.1-23.4]{EiMa}
$$\begin{array}{c|cccccc}
 & 3 & 4 & 5 & 6 & 7 & 8 \\
\hline
H_*(K(\z,3)) & \z & 0 & \z/2 & 0 & \z/3 & \z/2\\
H_*(K(\z,3);\z/2) & \z/2 & 0 & \z/2 & \z/2 & 0 & \z/2 \\
\end{array}$$
Using the Leray-Serre spectral sequence of the fibration $K(\z,3) \to P \to K(\z,2)$ one may compute the homology groups of $P$ in low dimensions
$$\begin{array}{c|ccccccccc}
& 0 & 1& 2& 3& 4& 5 & 6 & 7 & 8\\
\hline
H_*(P) & \z & 0 & \z & 0 & 0 & \z/2 & 0 & \z/2 \oplus \z/3 & \z/2
\end{array}$$
From this we can deduce the cohomology groups with $\z/2$-coefficients
$$\begin{array}{c|ccccccccc}
& 0 & 1& 2& 3& 4& 5 & 6 & 7 & 8\\
\hline
H^*(P;\z/2) & \z/2 & 0 & \z/2 & 0 & 0 & \z/2 & \z/2 & \z/2 & (\z/2)^2
\end{array}$$
Now if we look at the Leray-Serre spectral sequence for the cohomology with $\z/2$-coefficients, the differential $d \colon E_2^{0,6} \to E_2^{2,5}$ must be zero. From this we obtain the cup product structure of $H^*(P;\z/2)$ as follows: let $x \in H^2(P;\z/2)$, $a \in H^5(P;\z/2)$, $b \in H^6(P;\z/2)$ be generators, then $ax \in H^7(P;\z/2)$ is a generator, and $H^8(P;\z/2)$ has generators $bx$ and $c$, where $c$ restricts to a generator of $H^8(K(\z,3);\z/2)$.

We also need the information on the Steenrod operations in $H^*(P;\z/2)$. This is obtained as follows. Let $P'$ be the restriction of $P$ on $S^2 \subset K(\z,2)$. Then from the above spectral sequence we see that $H^*(P;\z/2) \to H^*(P';\z/2)$ are isomorphisms for $*=6,7,8$; and $H^5(P;\z/2) \to H^5(P';\z/2)$ is injective. Since the $k$-invariant of $P'$ is trivial, $P'$ is homotopy equivalent to $S^2 \times K(\z,3)$. Let $x \in H^2(K(\z,2);\z/2)$, $y \in H^3(K(\z,3);\z/2)$ be generators, then $Sq^2y \in H^5(K(\z,3);\z/2)$ and $y^2 \in H^6(K(\z,3);\z/2)$ are generators, and $H^5(P';\z/2)$ has generators $xy$ and $Sq^2y$. The image of $H^5(P;\z/2) \to H^5(P';\z/2)$ is generated by $Sq^2y$. By the Adem relation $Sq^2 Sq^2 y = Sq^3 Sq^1 y =0$. Also $Sq^2y^2=0$. Therefore $Sq^2 \colon H^5(P;\z/2) \to H^7(P;\z/2)$ and $Sq^2 \colon H^6(P;\z/2) \to H^8(P;\z/2)$ are trivial.

Now we are ready to apply the Atiyah-Hirzebruch spectral sequence to compute $\omst_*(P)$ and $\omst_*(P;\eta)$. The string cobordism groups up to dimension $7$ are (the first six groups are the same as the stable stems and the 7-th by \cite{Giam})
$$\begin{array}{c|cccccccc}
& 0 & 1& 2& 3& 4& 5 & 6 & 7 \\
\hline
\omst_* & \z & \z/2 & \z/2 & \z/24 & 0 & 0 & \z/2 & 0
\end{array}$$

In the Atiyah-Hirzebruch spectral sequence of $\omst_*(P)$, the differentials $d_2$ are dual to $Sq^2$ (\cite{Tei93}). Thus $d_2 \colon E^{2}_{7,0} \to E^{2}_{5,1}$, $d_2 \colon E^{2}_{7,1} \to E^{2}_{5,2}$ and $d_2 \colon E^{2}_{8,0} \to E^{2}_{6,1}$ are trivial. Therefore we have exact sequences
$$0 \to G \to \omst_7(P) \to \z/2 \oplus \z/3 \to 0$$
$$ \z/2 \to G \to \z/2 \to 0$$
Therefore we have
\begin{lem}\label{lem:order1}
$|\omst_7(P_3(S^2))| \le 24 $.
\end{lem}

In the Atiyah-Hirzebruch spectral sequence of $\omst_*(P;\eta)$, the differentials $d_2$ are dual to $Sq^2+x$ (\cite{Tei93}), where $x \in H^2(P;\z/2)$ is the nontrivial element. Thus $d_2 \colon E^{2}_{7,1} \to E^{2}_{5,2}$ is an isomorphism, $d_2 \colon E^{2}_{7,0} \to E^{2}_{5,1}$ is surjective and $d_2 \colon E^{2}_{8,0} \to E^{2}_{6,1}$ is trivial. Therefore we have
\begin{lem}\label{lem:order2}
$\omst_7(P_3(S^2);\eta) \cong \z/2 \oplus \z/3$.
\end{lem}

Now we apply the Atiyah-Hirzebruch spectral sequence to compute $\omst_{*}(P_{3}(S^{2}) \times K(\z^{2r},3))$ and $\omst_{*}(P_{3}(S^{2}) \times K(\z^{2r},3); \eta)$.
For simplicity denote $K(\z^{2r},3)$ by $K$. Let $x \in H^{2}(P;\z/2)$, $b_{i} \in H^{3}(K;\z/2)$ ($i=1, \cdots, 2r$)  be generators. Then $H^{5}(K;\z/2)$ is generated by $Sq^{2}b_{i}$ ($i =1, \cdots , 2r$). In the spin case in the Atiyah-Hirzebruch spectral sequence the differential $d_{2} \colon E^{2}_{7,1} \to E^{2}_{5,2}$ is dual to $Sq^{2} \colon H^{5}(P \times K ;\z/2) \to H^{7}(P \times K ;\z/2)$. The mixed term in $H^{5}(P \times K ;\z/2)$ is $H^{2}(P;\z/2) \otimes H^{3}(K;\z/2)$, with generators $a \otimes b_{i}$. The mixed term at $E^{2}_{7,0}=H_{7}(P \times K)$ is $H_{2} ( P)\otimes H_{5}(K) \cong (\z/2)^{2r}$, which reduces to $H_{2}(P;\z/2) \otimes H_{5}(K;\z/2)$ isomorphically. Note that $Sq^{2}(x \otimes b_{i})=x \otimes Sq^{2}b_{i}$ are the generators of $H^{2}(P;\z/2) \otimes H^{5}(K;\z/2)$. And at $E^{2}_{6,1} = H_{6}(P \times K;\z/2)$ there is no mixed term. Therefore no mixed terms survive in the spectral sequence. In the non-spin case, in the Atiyah-Hirzebruch spectral sequence for $\omst_{7}(P\times K; \eta)$, the differential $d_{2} \colon E^{2}_{7,1} \to E^{2}_{5,2}$ is dual to $Sq^{2}+x\cup$. Since $x^{2}=0 \in H^{4}(P;\z/2)$, again we have $(Sq^{2}+x \cup)(x \otimes b_{i})=x \otimes Sq^{2}b_{i}$. Therefore we have
\begin{lem}\label{lem:string}
There are isomorphisms
$$\omst_{7}(P_{3}(S^{2})) \oplus \omst_{7}(K(\z^{2r}, 3)) \to \omst_{7}(P_{3}(S^{2}) \times K(\z^{2r},3))$$
$$\omst_{7}(P_{3}(S^{2});\eta) \oplus \omst_{7}(K(\z^{2r}, 3)) \to \omst_{7}(P_{3}(S^{2}) \times K(\z^{2r},3);\eta)$$

\end{lem}

Finally we compute the twisted spin bordism group $\omsp_8(\cp^{\infty} \times K(\z,4);L)$, where $L$ is the tautological complex line bundle over $\cp^{\infty}$. This group can be identified with
$$\pi_{10}^S(M\mathrm{Spin} \wedge \mathrm{Th}(L))= \widetilde{\omsp_{10}} (\mathrm{Th}(L))=\omsp_{10}(\cp^{\infty} \times K(\z,4))/\omsp_{10}( \mathrm{pt}\times K(\z,4)).$$
We apply the Atiyah-Hirzebruch-Leray-Serre spectral sequence to the fiber bundle $K(\z,4) \to \cp^{\infty} \times K(\z,4) \to \cp^{\infty}$. By a routine calculation with the Atiyah-Hirzebruch spectral sequence we have $\widetilde{\omsp_i}(K(\z,4)) = 0$ for $i \le 7$ and $i \ne 4$, $\widetilde{\omsp_4}(K(\z,4)) \cong H_4(K(\z,4)) \cong \z$, $\widetilde{\omsp_8}(K(\z,4)) \cong \z \oplus \z$ (c.~f.~\cite[Theorem 8]{Kr18}). From this data, by a straightforward computation one has
$$\begin{array}{cl}
 & \omsp_8(\cp^{\infty} \times K(\z,4);L)   \\
\cong & \omsp_8 \oplus \widetilde{\omsp_8}(\cp^{\infty};L) \oplus \widetilde{\omsp_8}(K(\z,4)) \oplus H_4(\cp^{\infty}) \otimes H_4(K(\z,4))
\end{array}$$
where the last factor $H_4(\cp^{\infty}) \otimes H_4(K(\z,4)) \cong \z$ is generated by the natural inclusion $\cp^2 \times S^4 \to \cp^{\infty} \times K(\z,4)$.


\end{document}